\documentclass[]{article}
\usepackage{amsmath}
\usepackage{enumerate}
\usepackage{enumitem}
\usepackage{graphicx}
\usepackage{float}
\usepackage{hyperref}
\usepackage{verbatim}
\usepackage{amssymb}
\usepackage{color}
\usepackage{amsthm}
\usepackage{amssymb}
\usepackage{verbatim}
\usepackage{geometry}
\usepackage{comment}
\usepackage{mathrsfs}
\def\R{{\mathbb R}}
\def\N{{\mathbb N}}

\def\C{{\mathbb C}}
\def\T{{\mathbb{T}}}
\def\S{{\mathbb{S}}}
\def\ds{\displaystyle}
\newcommand{\norm}[1]{\left\Vert#1\right\Vert}
\newcommand{\BUC}{\mathrm{BUC}}

\renewcommand{\leq}{\leqslant}                 
\renewcommand{\geq}{\geqslant}
\renewcommand{\tilde}{\widetilde}

\renewcommand{\i}{\mathrm{\bf i}}                 
    
\newcommand{\enstq}[2]{\left\{#1~\middle|~#2\right\}}             

 \newtheorem{theorem}{Theorem}[section]
 \newtheorem{remark}[theorem]{Remark}

 \newtheorem{lemma}[theorem]{Lemma}
 \newtheorem{corollary}[theorem]{Corollary}

 \newtheorem{proposition}[theorem]{Proposition}
 \newtheorem{definition}[theorem]{Definition}
 \numberwithin{equation}{section}

\title{On the reachable space for parabolic equations\footnote{S. E. is partially supported by the ANR projects TRECOS ANR 20-CE40-0009,  NumOpTes ANR-22-CE46-0005, CHAT ANR-24-CE40-5470.}
}

\author{
Sylvain \textsc{Ervedoza}\footnote{Univ. Bordeaux, CNRS, Bordeaux INP, IMB, UMR 5251, F-33400 Talence, France,  
\texttt{sylvain.ervedoza@math.u-bordeaux.fr}}, 
\and
Adrien \textsc{Tendani-Soler}\footnote{Laboratoire Interdisciplinaire Carnot de Bourgogne,
CNRS UMR 6303, Université Bourgogne, BP 47870, 21078 Dijon, France,  
\texttt{adrien.tendani@u-bourgogne.fr}}
}

\begin{document}

\maketitle

\begin{abstract}
	In this article, we provide a description of the reachable space for the heat equation with various lower order terms, set in the euclidean ball of $\R^d$ centered at $0$ and of radius one and controlled from the whole external boundary. Namely, we consider the case of linear heat equations with lower order terms of order $0$ and $1$, and the case of a semilinear heat equations. 
	In the linear case, we prove that any function which can be extended as an holomorphic function in a set of the form $\Omega_\alpha = \enstq{ z\in\C^d}{ |\Re(z)| + \alpha |\Im(z)| < 1}$ for some $\alpha \in (0,1)$ and which admits a continuous extension up to $\overline\Omega_\alpha$ belongs to the reachable space. In the semilinear case, we prove a similar result for sufficiently small data. Our proofs are based on well-posedness results for the heat equation in a suitable space of holomorphic functions over $\Omega_\alpha$ for $\alpha > 1$.
\end{abstract}

\noindent\textbf{Keywords:} Control theory, Parabolic equations, Reachable space. 
\smallskip

\noindent\textbf{MSC 2020:} 35K05, 35K08, 93C05.
\setcounter{tocdepth}{1}
\tableofcontents

\section{Introduction}

The goal of this article is to give a description of the reachable space for heat equations with lower order terms when the equation is set in a ball of $\R^d$ and controlled from the whole external boundary. 

\subsection{Main results}

The geometrical setting is then the following one. We set $d \in \N$, $d\geq 1$,  $\Omega := B_{\R^d} (1)$, i.e. the euclidean ball of $\R^d$ centered at $0$ and of radius one.
\paragraph{Reachable space for the linear heat equation.}
We are first interested in the controllability properties of the following heat equation with lower order terms:
\begin{equation}
	\label{Controlled-Heat}
	\left\{
		\begin{array}{ll}
			\partial_t y - \Delta y + q y + W \cdot \nabla y= 0, \quad & \hbox{ in } (0,T) \times \Omega,
			\\ 
			y = u, \quad & \hbox{ on } (0,T) \times \partial\Omega, 
			\\
			y(0, \cdot ) = y_0, \quad & \hbox{ in } \Omega.
		\end{array}
	\right.
\end{equation}
In \eqref{Controlled-Heat}, $y$ is the state, $y_0$ is the initial datum, $q=q(t,x)$ is a potential of order $0$, $W=W(t,x)$ is a potential of order $1$, and $u$ is a control function. 

Our goal is to give a description of the reachable set $\mathscr{R}_{lin}(y_0, T)$, which is defined by 
\begin{equation}
	\label{Def-R-y0-T}
	\mathscr{R}_{lin}(y_0, T) := \{ y(T, \cdot), \ \text{ for $y$ solving \eqref{Controlled-Heat} starting from $y_0$ with control function } u \in L^2(0,T; L^2(\partial\Omega))\}, 
\end{equation}
that is the set of all states that can be reached by solutions of the system \eqref{Controlled-Heat} at time $T$ when starting from $y_0$ by choosing an appropriate control $u$ in $L^2(0,T; L^2(\partial\Omega))$. 

In order to state our main result, we first need to introduce some notation.

For any $\alpha >0 $, we set
\begin{equation}
	\label{Def-Omega-alpha}
	\Omega_\alpha := \{ a + \i b, \, \text{ with } a \in \R^d \text{ and } b \in \R^d \text{ satisfying } |a| + \alpha |b| < 1\}.
\end{equation}
We also introduce the space 
\begin{equation}
	\label{Def-R-alpha}
	\mathscr{R}_\alpha := \mathcal{C}(\overline{\Omega}_\alpha)\cap\mathrm{Hol}(\Omega_\alpha).
\end{equation}
which we endow with the norm 
\begin{equation}
	\| f\|_{\mathscr{R}_\alpha} := \| f \|_{L^\infty(\Omega_\alpha)}.
\end{equation}

Our main result is the following:
\begin{theorem}
	\label{Thm-Main-Linear}
	Let $d \geq 1$ and $\Omega = B_{\R^d} (1)$, and assume that for some $\alpha_0 \in (0,1)$,
	\begin{equation}
		\label{Assumptions-Linear-q-W}
		q \in L^\infty_{loc}(\R_+; \mathscr{R}_{\alpha_0}) 
		\quad 
		\hbox{ and } 
		\quad 
		W \in  L^\infty_{loc}(\R_+; (\mathscr{R}_{\alpha_0})^d).
	\end{equation}
	
	Then for any $y_0 \in L^2(\Omega)$
     and for any $T >0$, we have
	\begin{equation}
		\bigcup_{ \alpha \in (0,1)} \mathscr{R}_\alpha \subset \mathscr{R}_{lin}(y_0, T)\subset\mathrm{Hol}(\Omega_1).
	\end{equation}
	Moreover, for all $\alpha \in (0,1)$ and for all $T>0$, there exists a constant $C_\alpha(T)$ such that for any $y_1 \in \mathscr{R}_\alpha$, there exists a control $u \in L^\infty(0,T; L^\infty(\partial\Omega))$ such that the solution of \eqref{Controlled-Heat} with initial datum $y_0 = 0$ satisfies $y(T ) = {y_1}_{|_\Omega}$ and $\|u \|_{L^\infty(0,T; L^\infty(\partial\Omega))} \leq C_\alpha(T) \| y_1\|_{L^{\infty}(\Omega_\alpha)}$.
\end{theorem}
\paragraph{Reachable space for semilinear parabolic equations.}
Our second result focuses on the semilinear parabolic equations. More precisely, we consider the following control system
\begin{equation}
		\label{Controlled-Heat-Non-linear}
	\left\{
		\begin{array}{ll}
			\partial_t y - \Delta y + g(y, \nabla y)= 0, \quad & \hbox{ in } (0,T) \times \Omega,
			\\ 
			y = u, \quad & \hbox{ on } (0,T) \times \partial\Omega, 
			\\
			y(0, \cdot ) = y_0, \quad & \hbox{ in } \Omega.
		\end{array}
	\right.
	\end{equation}
where $g=g(t,x,s,s_d)$ is a nonlinear function and we aim to describe the reachable set $\mathscr{R}_{nonlin}(y_0, T)$ defined as follows
\begin{equation}
	\label{Def-R-y0-T-nonlinear}
	\mathscr{R}_{non lin}(y_0, T):=\{ y(T, \cdot), \ \text{for $y$ solving \eqref{Controlled-Heat-Non-linear} starting from $y_0$ with control function } u \in L^2(0,T; L^2(\partial\Omega))\}. 
\end{equation}
We also define, for any $\delta>0$ and $\alpha>0$, the set 
$$
\mathscr{R}_{\alpha}^{\delta}:=\enstq{f\in \mathcal{C}^1(\overline{\Omega}_\alpha)\cap\mathrm{Hol}(\Omega_\alpha)}{\|f\|_{W^{1,\infty}(\Omega_\alpha)}\leq\delta}.
$$
Let us now state the result concerning the reachable spaces for the semilinear heat equation. 
%
%
\begin{theorem}
	\label{Thm-Main-Semilinear}
Let $d\geq 1$, $T>0$ and $\Omega = B_{\R^d} (1)$. 

Assume that, for some $\alpha_0 \in (0,1)$ and $\varepsilon>0$, we have a semilinearity 
	$$
		g: (t,x, s, s_d) \in [0,T] \times \overline\Omega_{\alpha_0} \times \overline B_\C (\varepsilon) \times \overline B_{\C^d} (\varepsilon)\longmapsto g(t,x,s,s_d)\in\C, 
	$$
	such that 
	\begin{equation}
		\label{Intro-g-in-BUC}
		g \in L^\infty(0,T; \mathcal{C}(\overline\Omega_{\alpha_0}\times \overline B_\C (\varepsilon) \times  \overline B_{\C^d} (\varepsilon)))
	\end{equation}
	\begin{equation}
		\label{Intro-Holomorphic-Condition-g}
		\text{a.e. in } t\in\R_+, 
		\, 
		g(t, \cdot, \cdot,\cdot) \text{ is holomorphic in }\Omega_{\alpha_0} \times \overline B_\C(\varepsilon)\times \overline B_{\C^d}(\varepsilon), 
	\end{equation}
	\begin{equation}
		\label{Intro-g-Lipschitz-3}
		g \in L^\infty([0,T] \times \overline\Omega_{\alpha_0}; W^{3, \infty} (\overline B_\C (\varepsilon) \times  \overline B_{\C^d} (\varepsilon))).
	\end{equation}
	\begin{equation}
		\label{Intro-g-for-s=0}
g(\cdot,\cdot, 0_\C, 0_{\C^d}) = 0,
	\end{equation}
Then there exist $\delta_{0}>0$ and a family of positive real numbers $(\delta_\alpha)_{\alpha\in (0,1)}$ such that for any $y_0\in \mathcal{C}^{1}(\overline{\Omega})$ satisfying 
$$
\|y_0\|_{W^{1,\infty}(\Omega)}\leq \delta_{0},
$$
we have
$$
\bigcup_{ \alpha \in (0,1)} \mathscr{R}_{\alpha}^{\delta_{\alpha}} \subset \mathscr{R}_{nonlin}(y_0, T).
$$
In particular, for all $\alpha\in (0,1)$, there exist $C_\alpha$ such that for all $y_0\in\mathcal{C}^1(\overline{\Omega})$ satisfying $\|y_0\|_{W^{1,\infty}({\Omega})}\leq \delta_0$,  and $y_1\in \mathscr{R}_{\alpha}$ with $\|y_1\|_{W^{1,\infty}(\Omega_\alpha)}\leq \delta_\alpha$, there exists a control function $u\in L^{\infty}((0,T)\times\partial\Omega)$ such that the solution $y$ of \eqref{Controlled-Heat-Non-linear} with initial datum $y_0$ satisfies $y(T)={y_1}_{|_\Omega}$ and $\|u\|_{L^{\infty}((0,T)\times\Omega)}\leq C_{\alpha}(\|y_0 \|_{W^{1,\infty}(\Omega)}+\|y_1\|_{L^{\infty}(\Omega_\alpha)})$.
\end{theorem}

\subsection{Related references and comparison to the existing results} 
The interest of describing the reachable set is clear in the context of control theory. Several results in the literature on control theory for parabolic equation can thus be interpreted in terms of this space. First, due to the strong regularization properties satisfied by the solutions of parabolic equations such as \eqref{Controlled-Heat}, at least when $q = 0$ and $W= 0$, it is known that the reachable space for \eqref{Controlled-Heat} cannot be the whole space $H^{-1}(\Omega)$ nor any reasonable Sobolev space of the form $W^{k,p}(\Omega)$. Despite this, approximate controllability holds, at least when $q$ and $W$ belong to $L^\infty((0,T) \times \Omega)$, see the theory developed in \cite{Fabre-Puel-Zuazua-Approx} when $W = 0$, and \cite{FursikovImanuvilov} for the proof of the corresponding unique continuation result. In other words, we know that $\mathscr{R}_{lin}(y_0, T)$ is a dense subset of $H^{-1}(\Omega)$.

Another line of works show that system \eqref{Controlled-Heat} is null-controllable, that is for any $y_0 \in H^{-1}(\Omega)$, there exists a control function $u \in L^2(0,T; L^2(\partial\Omega))$ such that the solution $y$ of \eqref{Controlled-Heat} satisfies $y(T) = 0$. For $q$ and $W$ in $L^\infty((0,T) \times \Omega)$, this is achieved in \cite{FursikovImanuvilov} using Carleman type estimates, see also \cite{Imanuvilov-Yamamoto-2001}. As a consequence of the linearity of \eqref{Controlled-Heat}, any trajectory of \eqref{Controlled-Heat} can be reached at time $T$. It then follows from \cite{Seidman} that the reachable space $\mathscr{R}_{lin}(y_0,T)$ in fact does not depend on the initial state $y_0$,  nor on the time horizon $T$ when $q$ and $W$ do not depend on time. 

Although important, these works do not precisely characterize the set $\mathscr{R}_{lin}(y_0,T)$.
\paragraph{In dimension $d=1$.} When $W = 0$, a first attempt was done in the $1$-d setting in the work \cite{FattoriniRussel71} describing the reachable space in terms of the coefficients in the expansion on the basis of eigenvectors of the Laplacian. Already there, see \cite[p.280]{FattoriniRussel71}, it was remarked that this earlier description implies that the subset of $\mathscr{R}_{lin}(0,T)$ obtained in \cite{FattoriniRussel71} corresponds to functions which have an holomorphic expansion in some part of the complex plane (namely a suitable strip in this case), see also \cite{ErvZuazuaARMA} for a similar result. 

This point of view was adopted only recently, starting from the work \cite{MartinRosierRouchon-2016}, which shows that, when $q = 0$ and $W = 0$ and in dimension $1$, the reachable space is sandwiched between the space of functions which are holomorphic in the square 
$$
	\Omega_1 = \enstq{ z\in\C}{|\Re(z)| +|\Im(z)| < 1 },
$$
and the space of functions which are holomorphic in some complex ball $B_\C(0,R)$ for $R > \exp(1/(2e))$  ($\simeq 1.2$).

Still in the case $q = 0$, $W=0$ and in $1$-d, it was then improved in a series of work \cite{Darde-Ervedoza-16,Hartmann-Kellay-Tucsnak,Kellay-Normand-Tucsnak,Orsoni-JFA-2021}, see also  \cite{Aikawa-Hayashi-Saitoh-90} for a precise description of the reachable space of the $1$-d heat equation on the half-line when controlled from the boundary), culminating with the result \cite{Hartmann-Orsoni-2021} which proves that the reachable space coincides with the Bergman space $A^2(\Omega_1)$, i.e. the functions in $L^2(\Omega_1)$ which are holomorphic in $\Omega_1$.

When considering non-trivial lower order terms in the $1$-d case, it has been shown in \cite{Hartman-Orsoni-2022-Hermite} that when $q(x) = x^2$ and $W (x) = 0$, the reachable space is sandwiched between the space of functions which are holomorphic in the square $S$ and the space of functions which are holomorphic in $(1+\varepsilon) \Omega_1$ for some $\varepsilon >0$. Another result was obtained in \cite{Ervedoza-LeBalch-Tucsnak} by a perturbative argument showing that, when $W = 0$, for potentials $q$ which have an holomorphic expansion in $\Omega_1$ which belongs to $L^\infty(\Omega_1)$ and whose $L^\infty(\Omega_1)$ norm is small enough, the reachable space of \eqref{Controlled-Heat} is not modified by the potential $q$, and coincides with $A^2(\Omega_1)$. 

Another approach has been developed in \cite{Laurent-Rosier} which allows to handle semi-linear lower order terms $f(x, y, \partial_x y) $ in \eqref{Controlled-Heat}, allowing to show that if the lower order terms are holomorphic in all three variables on the complex ball of radius $> 4 \exp( 1/(2e))$, then the functions which are holomorphic in the ball of radius {$B_\C(0,R)$} for some $R> 4 \exp( 1/(2e))$ and small enough belong to the reachable space $\mathscr{R}_{nonlin}(0,T)$. Theorem  \ref{Thm-Main-Semilinear} is thus more precise than the one obtained in \cite{Laurent-Rosier}. 
Also note that, in a recent work \cite{Exactcontrollabilityofanisotropic1Dpartialdifferentialequationsinspacesofanalyticfunctions}, the approach of \cite{Laurent-Rosier} has been developed for a much more general class of anisotropic semilinear $1$-d PDE in order to give a subspace of the reachable space in those cases.

\paragraph{In dimension $d \geq 2$.} The results are very seldom in dimension greater or equal than $2$. A description in terms of the coefficients in the expansion on the basis of eigenfunctions is given in \cite[Remark 6.1]{FernandezCaraZuazua1} when $q = 0$ and $W = 0$ based on Carleman estimates (and without any geometric condition on $\Omega$ and $\omega$), generalizing an earlier attempt appearing in \cite[Section 6]{FatRus}, but which seem far from being optimal. Recently, a significant step forward has been achieved in \cite{Strohmaier-Waters-2022} in our geometric setting when $q = 0 $ and $W = 0$, showing that the reachable space is sandwiched between the set of functions which are holomorphic in 
\begin{equation}
	\label{Def-Sd}
	{\Omega_1 = \enstq{ z\in\C^d}{|\Re(z)| +|\Im(z)| < 1 }}
\end{equation}
and the set of functions which are holomorphic in $(1+ \varepsilon) {\Omega_1}$ for some $\varepsilon >0$.  Theorem \ref{Thm-Main-Linear} and Theorem \ref{Thm-Main-Semilinear} thus generalize the approach of \cite{Strohmaier-Waters-2022} and extends it to the case of non-trivial linear and non-linear lower order terms. 
\bigskip

To sum up:
\begin{itemize}
	
	\item Theorem \ref{Thm-Main-Linear} is similar to the results in \cite{Darde-Ervedoza-16,Hartman-Orsoni-2022-Hermite} obtained in the $1$-d case for potentials $(q,W) = (0,0)$ and potentials $(q,W) = (x^2, 0)$ respectively, but obviously generalizes it to much more general potentials. 
	
	\item Theorem \ref{Thm-Main-Linear} is similar to the results in \cite{Strohmaier-Waters-2022} obtained for potentials $(q,W ) = (0,0)$. 
	
	\item Theorem \ref{Thm-Main-Semilinear} improves the results obtained in \cite{Laurent-Rosier} in the $1$-d case by allowing semilinearities depending on time and mostly by giving more precise conditions on the holomorphic expansion of the states in the reachable space. 
\end{itemize}

	Finally, let us also point out that the precise description of the reachable space for PDE is still a challenging issue from various viewpoints.
	
	In particular, we point out that the few existing results in dimension greater or equal than $2$ only consider the case in which the domain $\Omega$ is a ball controlled from the whole external boundary. It is thus an important challenge to understand the interplay of the geometry of the domains $\Omega$ and the control set with the reachable space. 
	
	Similarly, our results only concern heat-type parabolic equations, and it would be interesting to develop them for more involved parabolic equations, such as Stokes or Boussinesq equations, and more general PDE, as initiated by the work \cite{Exactcontrollabilityofanisotropic1Dpartialdifferentialequationsinspacesofanalyticfunctions}.

\subsection{Ideas and strategy of the proof of Theorems \ref{Thm-Main-Linear} and \ref{Thm-Main-Semilinear}}

Our strategy is based on the following facts.
For a linear control problem
$$
	y' = A y + B u, 
$$
in which $A$ generates a strongly continuous semigroup $\T = (\T_t)_{t \geq 0}$ on a Hilbert space $H$, and $B$ is an admissible control operator in $\mathscr{L}(U, \mathscr{D}(A^*)')$ for some Hilbert space $U$, and which is null-controllable in any positive time, the restriction of the semigroup $\T$ to the reachable space $\mathscr{R} $ (as said above, by \cite{Seidman}, we know that in this case the reachable does not depend on the initial state nor on the time horizon $T$) is a strongly continuous semigroup in $\mathscr{R}$,  see \cite{VanNeerven-2005,Ervedoza-LeBalch-Tucsnak}.  

Accordingly, a good candidate for a functional space $X$ to be the reachable space of the linear heat equation \eqref{Controlled-Heat} should be a functional space in which the heat semigroup is well-posed. Also, in view of the results in the literature, see the above review, this functional space $X$ should be a space of functions which extends in a holomorphic way in a domain similar to the set $\Omega_1$ in \eqref{Def-Sd}.

The main step in the proof of Theorem \ref{Thm-Main-Linear} is thus the study of the heat semigroup on a suitable space of holomorphic functions. To make it easier, we consider the heat equation on $\R^d$, and we introduce a space $\tilde X$ of holomorphic functions which satisfies the following properties: 
\begin{itemize}
	\item $\tilde X$ is a space of functions which are defined on $\R^d$ and whose restriction to $B_{\R^d}(1)$ admit an holomorphic expansion in a set of $\C^d$ close to the set $\Omega_1$ in \eqref{Def-Sd}; 
	\item the heat semigroup is an analytic semigroup on $\tilde X$; 
	\item $\tilde X$ is an algebra. 
\end{itemize}
We refer to Section \ref{Sec-Heat-WP} for the definition of such spaces $\tilde X$ and the proof of the above properties (in particular Definition \ref{Def-X-alpha} and Theorem \ref{Thm-Well-Posedness-X-alpha}). 

Once this is done, we develop in Section \ref{Sec-WP-Heat+lower-order} the usual machinery in the context of analytic semigroups to be able to solve the heat equation with linear and semilinear lower order terms in these suitable spaces $\tilde X$.

Section \ref{Sec-NC-in-X-alpha} then shows how one can use the usual null-controllability properties of parabolic equations in $L^2$ to prove null-controllability properties of parabolic equations in these spaces.

Finally, to get a description of the reachable space for the heat equation with linear and non-linear lower order terms, we use in Section \ref{Sec-Reachable-Heat} an idea of \cite{Strohmaier-Waters-2022} and simply use Wick's rotation (multiplication of an element of $\C^d$ by $\i$). Indeed, it is not difficult to check that the image of the spaces $\tilde X$ by this transformation belongs to the reachable space, up to a minor rescaling.

\bigskip
\noindent{\bf Acknowledgments.} The authors are indebted to El Maati Ouhabaz and Armand Koenig for stimulating discussions during the writing of this article.  
 
\section{Well-posedness of the heat equation in a space of holomorphic functions}\label{Sec-Heat-WP}

The goal of this section is to discuss the action of the heat semigroup $\mathbb{T} = (\mathbb{T}_t)_{t \geq 0}$ on $\R^d$ on a suitable space of holomorphic functions. 

To be more precise, we consider the heat equation on $\R^d$, given by 
\begin{equation}
	\label{Heat-Eq}
	\left\{
		\begin{array}{ll}
			\partial_t y - \Delta_x y = 0, \quad & \text{ in } (0, \infty) \times \R^d, 
			\\
			y(0, \cdot) = y_0, \quad & \text{ in } \R^d.
		\end{array}
	\right.
\end{equation}
Here, $y_0$ is the initial datum, defined on $\R^d$. 

It is well-known (see for instance \cite[Section 2.13]{Engel-Nagel}) that, if $y_0\in L^{2}(\R^d)$, then the solution $y$ of \eqref{Heat-Eq} is given by 
\begin{equation}
	\label{Heat-Semigroup}
	y(t,x) = (\mathbb{T}_t y_0)(x) := (G_d(t) \star y_0) (x), \quad t > 0, \, x \in \R^d, 	
\end{equation}
where $\star$ denotes the convolution in the space variable and $G_d$ is the heat kernel given by the following formula
\begin{equation}
\label{Heat-Kernel}
G_d(t,z) := \frac{1}{(4 \pi t)^{d/2}} \exp\left( - \frac{z^2}{4t} \right), \quad  t>0, z\in\C^d, 
\end{equation}
where for any $z=\!^t(z_1,\dots,z_d)\in \C^d$, we write $z^2$ to denote the quantities 
\begin{equation}
\label{eq:notation carre d un vecteur a coordonnes complex}
z^2:=\sum_{j=1}^{d}z_{j}^{2}.
\end{equation}
(Note that we defined the heat kernel for $t >0$ and $z \in \C^d$ for later use, but of course the formula \eqref{Heat-Semigroup} only involves the restriction of the heat kernel to $ x \in \R^d$.)

From the formula \eqref{Heat-Semigroup}, one can also easily check that the heat semigroup $\T=(\T_t)_{t>0}$ generates a $C_0$ semigroup on the Banach space of bounded uniformly continuous functions 
$$
	\BUC(\R^d) := \enstq{ f : \R^d \to \C}{ f \in L^\infty(\R^d) \text{ and } f \text{ is uniformly continuous}}, 
$$
endowed with the norm $\| \cdot \|_{L^\infty(\R^d)}$.
 
We will go further and check that the heat semigroup $\T$ also generates a $C_0$ semigroup on an appropriate space of holomorphic functions. 

To be more precise, with $\Omega_\alpha$ defined for $\alpha >0 $ in \eqref{Def-Omega-alpha}, we introduce the following variant of the space $\mathcal{C}(\overline{\Omega}_\alpha)\cap \mathrm{Hol}(\Omega_\alpha)$.
\begin{definition}
\label{Def-X-alpha}
Let $\alpha>0$ and $d \in \N$. We define the space $X_\alpha(\R^d)$ as the space of functions $f\in \mathrm{BUC}(\R^d)$ such that $f_{|_{B_{\R^d}(0,1)}}$ admits a continuous extension $f_e$ on $\overline{\Omega}_\alpha$ which is holomorphic on $\Omega_\alpha$, endowed with the norm
\begin{equation}
	\label{Def-Norm-X-alpha}
	\| f \|_{X_\alpha(\R^d)} := \| f \|_{L^\infty(\R^d)} + \| f_e \|_{L^\infty(\Omega_\alpha)}.
\end{equation}
To alleviate notation, we will simply denote $X_\alpha(\R^d)$ by $X_\alpha$ in the following, as $d$ will be clear from the context. 
\end{definition}

We then have the following result: 
\begin{theorem}\label{Thm-Well-Posedness-X-alpha}
	Let $d \in \N$ and $\alpha >1$. Then the heat semigroup $\T$ is an analytic semigroup on $X_\alpha$, and we have the following estimates: there exists $C>0$ such that for all $t > 0$ and for all $y_0 \in X_\alpha$, 
	\begin{equation}
		\label{Est-Well-Posedness-X-alpha}
		\| \T_t y_0 \|_{X_\alpha} 
		+ 
		\sqrt{t } \| \nabla_x (\T_t y_0) \|_{X_\alpha}
		\leq C \| y_0\|_{X_\alpha}.
	\end{equation}
\end{theorem}

Theorem \ref{Thm-Well-Posedness-X-alpha} is proved in the next sections. To start with, we explain in Section \ref{Subsec-Est-WP-1-d} how to get estimates \eqref{Est-Well-Posedness-X-alpha} in the $1$-d setting. In Section \ref{Subsec-Est-WP-any-d}, we then deduce estimates \eqref{Est-Well-Posedness-X-alpha} in the general case. 

	Note that, in the above statement and in the rest of the article, the derivative operators $\partial_j$ ($j \in \{1, \cdots, d\}$) stands for the derivative with respect to $x_j$, also sometimes denoted by $\partial_{x_j}$. Accordingly, saying that $\partial_{x_j} y \in X_\alpha$ for $y \in X_\alpha$ whose extension in $\overline \Omega_\alpha$ is denoted by $y_e$, should be understood as follows: $y \in \mathcal{C}^1(\R^d)$, and the element $z = \partial_{x_j} y $ belongs to $\BUC(\R^d)$, and can be extended as a continuous function $z_e$ in $\overline \Omega_\alpha$, holomorphic in $\Omega_\alpha$. In particular,  this notation $(\partial_j y)_e$ does not stand for the usual complex derivative of $y_e$ on $\Omega_\alpha$.

\subsection{Proof of the estimates (\ref{Est-Well-Posedness-X-alpha}) in the $1$-d setting.}
\label{Subsec-Est-WP-1-d}

In this subsection, we set $d = 1$, and we simply write $G$ for the Gaussian kernel $G_1$ in space dimension $1$ (defined in \eqref{Heat-Semigroup}).

Our goal is to prove the following: 
\begin{theorem}\label{Thm-Estimates-X-alpha-d=1}
	Let $d =1$ and $\alpha >1$. Then there exists $C>0$ such that for all $t > 0$ and for all $y_0 \in X_\alpha$, 
	\begin{equation}
		\label{Est-Well-Posedness-X-alpha-d=1}
		\| \T_t y_0 \|_{X_\alpha} + \sqrt{t } \| \partial_x (\T_t y_0) \|_{X_\alpha} \leq C \| y_0\|_{X_\alpha}.
	\end{equation}
\end{theorem}

\begin{proof}
	Let $y_0 \in X_\alpha$ and $t >0$.  In the proof given below, all the constants are independent of $y_0 \in X_\alpha$ and $t >0$.  
	\medskip

	\noindent{\bf Estimates on $\T_t y_0$ on $\R$.}
	To show that $\T_t y_0$ in $\BUC(\R)$, we simply write 
	\begin{equation}
		\label{Stab-L-infty}
		\| \T_t y_0 \|_{L^\infty(\R)} 
		\leq \| G(t)\|_{L^1(\R)} \| y_0\|_{L^\infty(\R)} 
		\leq \| y_0 \|_{L^\infty(\R)}, 
	\end{equation}
	and, for any $\delta \in \R$, 
	\begin{equation}
		\label{Stab-Modulus-Continuity}
		\| \T_t y_0(\cdot + \delta) - \T_t y_0 \|_{L^\infty(\R)} 
		\leq \| G(t)\|_{L^1(\R)} \| y_0(\cdot + \delta) - y_0 \|_{L^\infty(\R)} 
		\leq \| y_0(\cdot + \delta) - y_0 \|_{L^\infty(\R)}. 
	\end{equation}
	Similarly, using that $\| \partial_x G(t)\|_{L^1(\R)}\leq C/\sqrt{t}$ for some constant $C$ independent of $t>0$, we get $\partial_x \T_t y_0 \in \BUC(\R)$ and 
	\begin{equation}
		\label{Stab-W-1-infty}
		\| \partial_x \T_t y_0 \|_{L^\infty(\R)} 
		\leq \frac{C}{\sqrt{t}} \| y_0 \|_{L^\infty(\R)}.
	\medskip
	\end{equation}

	\noindent{\bf Estimates on $\T_t y_0$ in $\overline\Omega_\alpha$.}
	The delicate part to estimate is the one corresponding to the holomorphic extension in 
	$$
		\Omega_\alpha = \enstq{ a + \i b}{ a, b \in\R \text{ and } |a| + \alpha |b|  < 1}.
	$$
	To better understand what we do, let us write \eqref{Heat-Semigroup} explicitly:
	$$
		\T_t y_0(x) = \int_{\R} G(t,x-x_0) y_0(x_0) \,{\rm d} x_0 , \quad \text{ for } x  \in \R.
	$$
	Thus, since the Gaussian $G(t)$ is an holomorphic function, a natural way to get an holomorphic extension
of $\T_t y_0$ on $\Omega_\alpha$, is by setting 
	\begin{equation}
		\T_t y_0(a + \i b) = \int_{\R} G(t,(a+\i b)-x_0) y_0(x_0) \, {\rm d} x_0, \quad \text{ for } (a, b) \in \R^2\text{ such that }a+\i b\in\Omega_\alpha.
	\end{equation}
	In view of the assumption $y_0 \in X_\alpha$, it is convenient to split this formula: for $a + \i b \in \overline \Omega_\alpha$, 
	\begin{equation}
		\label{Decomp-T-t-y0}
		\T_t y_0(a + \i b) 
		= y_1(t,a+ \i b) + y_2 (t,a + \i b),
	\end{equation}
	where 
	\begin{align}
		& y_1(t,a+\i b) :=\int_{|x_0| >1 } G(t,(a+\i b)-x_0) y_0(x_0) \, {\rm d} x_0, 
		\\
		&y_2(t,a+\i b) :=\int_{|x_0| <1 } G(t,(a+\i b)-x_0) y_0(x_0) \, {\rm d} x_0.
	\end{align}
	Similarly, using $\partial_x G(t,x) = - x G(t,x)/(2t)$, the natural extension of $\partial_x \T_t y_0$ is given by 
	\begin{equation}
		(\partial_x \T_t y_0)(a + \i b) = - \frac{1}{2t} \int_{ \R} (a+\i b - x_0) G(t,(a+\i b)-x_0) y_0(x_0) \, {\rm d} x_0, \quad \text{ for } (a, b) \in \R^2, 
	\end{equation}
	that we decompose into 
	\begin{equation}
		\label{Decomp-partial-x-T-t-y0}
		(\partial_x \T_t y_0)(a + \i b) 
		= z_1(t,a+ \i b) + z_2 (t,a + \i b),
	\end{equation}
	where 
	\begin{align}
		& z_1(t,a+\i b) :=-\frac{1}{2t} \int_{|x_0| >1 }   (a+\i b - x_0) G(t,(a+\i b)-x_0) y_0(x_0) \, {\rm d} x_0, 
		\\
		&z_2(t,a+\i b) :=- \frac{1}{2t} \int_{|x_0| <1 }   (a+\i b - x_0) G(t,(a+\i b)-x_0) y_0(x_0) \, {\rm d} x_0.
	\end{align}
	We will thus estimate these functions separately, relying on the following lemmas (proved afterwards): 
	\begin{lemma}
		\label{Lem-Est-y-1-and-z-1}
		There exists a constant $C_\alpha$ such that for all $t>0$ and $y_0 \in X_\alpha$,
		\begin{equation}
			\label{Est-y-1-and-z-1}
			\| y_1(t,\cdot) \|_{L^{\infty}(\Omega_\alpha)} 
			+
			\sqrt{t} \| z_1(t,\cdot) \|_{L^{\infty}(\Omega_\alpha)} 
			\leq  
			C_\alpha \| y_0\|_{L^\infty} \leq C_\alpha \| y_0 \|_{X_\alpha}. 
		\end{equation}
	\end{lemma}

	\begin{lemma}
		\label{Lem-Est-y-2-and-z-2}
		There exists a constant $C_\alpha$ such that for all $t>0$ and $y_0 \in X_\alpha$, 
		\begin{equation}
			\label{Est-y-2-and-z-2}
			\| y_2(t,\cdot) \|_{L^{\infty}(\Omega_\alpha)} 
			+
			\sqrt{t} \| z_2(t,\cdot) \|_{L^{\infty}(\Omega_\alpha)} 
			\leq 
			C_\alpha \| y_0 \|_{X_\alpha}. 
		\end{equation}
	\end{lemma}

	\begin{proof}[Proof of Lemma \ref{Lem-Est-y-1-and-z-1}]
		For $a+ \i b \in \overline\Omega_\alpha$ and $x_0 \in \R$ with $|x_0| >1$, we write
		$$
			|G(t,a+ \i b -x _0)| = \frac{1}{\sqrt{4 \pi t} } \exp\left( - \frac{1}{4t} \Re ( ((a+ \i b) - x_0)^2) \right),
		$$
		and estimate $\Re ( ((a+ \i b) - x_0)^2)$ as follows:
		\begin{align*}
			- \Re ( ((a+ \i b) - x_0)^2) 
			& = - (a-x_0)^2 + b^2 
			\\
			& \leq - (|x_0|-|a|)^2 +\frac{1}{\alpha^2} (1 - |a|)^2
			\\
			& =  - (|x_0| - 1)^2 - 2 ( 1 - |a|)(|x_0| -1) - \left(1 - \frac{1}{\alpha^2}\right) (1 - |a|)^2
			\\
			& \leq  - (|x_0| - 1)^2. 
		\end{align*}
		Accordingly, for $a+ \i b \in \overline\Omega_\alpha$, 
		\begin{equation}
			\label{Est-y-1}
			|y_1(t, a + \i b ) | \leq \int_{|x_0|>1} \frac{1}{\sqrt{4 \pi t} } \exp\left( - \frac{1}{4t} (|x_0|-1)^2 \right)| y_0(x_0)| \, {\rm d} x_0 \leq  \| y_0 \|_{L^\infty(\R)}.
		\end{equation}
		Similarly, for $a+ \i b \in \overline\Omega_\alpha$, using that, for $x_0 \in \R\setminus [-1,1]$,
		$$
			|a + i b - x_0 | \leq |a-x_0| + |b| \leq ||x_0| - |a|| + \frac{1}{\alpha} |1 - |a|| \leq 2 | |x_0|- |a| |,
		$$
		we get
		\begin{align*}
			|z_1(t, a + \i b ) |& \leq \frac{1}{t} \int_{|x_0|>1} \frac{||x_0| - |a||}{\sqrt{4 \pi t} } \exp\left( - \frac{1}{4t} (|x_0|-|a|)^2+ \frac{1}{4t \alpha^2}(1 - |a|)^2 \right)| y_0(x_0)| \, {\rm d} x_0 \
			\\
			&
			\leq  \frac{2}{\sqrt{\pi t}} \exp\left( - \frac{1}{4t} (1-|a|)^2+ \frac{1}{4t \alpha^2}(1 - |a|)^2 \right) \| y_0 \|_{L^\infty(\R)}
			\\
			& \leq \frac{2}{\sqrt{\pi t}} \| y_0 \|_{L^\infty(\R)}.
		\end{align*}

		Together with estimate \eqref{Est-y-1}, this ends the proof of Lemma \ref{Lem-Est-y-1-and-z-1}.
	\end{proof}

	\begin{proof}[Proof of Lemma \ref{Lem-Est-y-2-and-z-2}]
		For $a + \i b \in \overline\Omega_\alpha$, we will rely on Cauchy's formula in the set delimited by $[-1,1]$ and 
		\begin{align*}
			& \Gamma_{\ell} := \enstq{ -(1-\tau) + \tau (a+ \i b)}{ \tau \in [0,1] },
			\\
			& \Gamma_{r} := \enstq{ (1- \tau) (a + \i b) + \tau }{ \tau \in [0,1] },
		\end{align*}
		(In the above formula, $\Gamma_{\ell}$ and $\Gamma_{r}$ respectively are the left and right sides of the triangle of basis $[-1,1]$ and of summit $a + \i b$.) 
		Note that $\Gamma_{\ell}$ and $\Gamma_{r}$ could be part of the boundary of $\overline\Omega_\alpha$, and we can still apply Cauchy's formula in this case since the function $y_{0,e}$ is holomorphic in $\Omega_\alpha$ and is continuous on $\overline\Omega_\alpha$.

		This gives, for $a+ \i b\in \overline\Omega_\alpha$, 
		\begin{multline*}
			y_2(t, a+ \i b) = 
			(a+1 + \i b) \int_{0}^1 G(t,(1- \tau)(a+1+\i b)) y_{0,e}( -(1-\tau) + \tau (a+ \i b)) \, {\rm d}\tau
			\\
			-
			(a-1 + \i b) \int_{0}^1 G(t,\tau(a - 1 + \i b)) y_{0,e}( (1- \tau) (a + \i b) + \tau) \, {\rm d}\tau.
		\end{multline*}
		We then use
		\begin{align*}
			& |G(t, (1- \tau)(a+1 + \i b))| = \frac{1}{\sqrt{4 \pi t}} \exp\left( - \frac{(1-\tau)^2}{4t} \Re((a+1+ \i b)^2)\right) = \frac{1}{\sqrt{4 \pi t}} \exp\left( - \frac{(1-\tau)^2}{4t} ((a+1)^2-  b^2)\right),
			\\
			& |G(t, \tau (a-1 + \i b))| = \frac{1}{\sqrt{4 \pi t}} \exp\left( - \frac{\tau^2}{4t} \Re((a-1+ \i b)^2)\right)= \frac{1}{\sqrt{4 \pi t}} \exp\left( - \frac{\tau^2}{4t} ((a-1)^2-  b^2)\right),
		\end{align*}
		so that 
		\begin{align*}
			&|a + 1+ \i b| \int_0^1 |G(t,(1- \tau)(a+1+\i b))| \, {\rm d}\tau
			 \leq 
			\frac{|a + 1 + \i b| }{\sqrt{ (a+ 1)^2 - b^2}} 
			\leq 
			\frac{ |a +1| \sqrt{1 + 1/\alpha^2}}{ |a +1| \sqrt{1 - 1/\alpha^2}} \leq \sqrt{\frac{\alpha^2 +1}{\alpha^2 -1}}, 
			\\
			&
			|a - 1+ \i b| \int_0^1 |G(t,\tau(a-1+\i b))| \, {\rm d}\tau  
			 \leq \frac{ |a -1 + \i b| }{\sqrt{(a-1)^2 - b^2}} 
			\leq 
			\frac{ |a -1| \sqrt{1 + 1/\alpha^2}}{ |a -1| \sqrt{1 - 1/\alpha^2}} \leq \sqrt{\frac{\alpha^2 +1}{\alpha^2 -1}}, 
		\end{align*}
		where we used that, for $a + \i b \in \overline\Omega_\alpha$, 
		\begin{equation}
			\label{Comparison-b-vs-a}
			|b| \leq \frac{1}{\alpha} (1 - |a|) \leq \frac{1}{\alpha} \min\{ |a+1 |, |a-1| \}. 
		\end{equation}
		We thus deduce that for $a + \i b \in \overline\Omega_\alpha$, 
		\begin{equation}
			\label{Est-y-2}
			|y_2(t, a+\i b)| \leq 2 \sqrt{\frac{\alpha^2 +1}{\alpha^2 -1}} \| y_{0,e} \|_{L^\infty(\overline\Omega_\alpha)}.
		\end{equation}

		It remains to prove the estimate on $z_2$. As above, we start with Cauchy's formula to get that for $a + \i b \in \overline{\Omega_\alpha}$, 
		\begin{multline*}
			z_2(t, a+ \i b) 
			= 
			- \frac{(a+1 + \i b)^2}{2t} \int_{0}^1(1- \tau) G(t,(1- \tau)(a+1+\i b)) y_{0,e}( -(1-\tau) + \tau (a+ \i b)) \, {\rm d}\tau
			\\
			+
			\frac{(a-1 + \i b)^2}{2t} \int_{0}^1 \tau G(t,\tau(a - 1 + \i b)) y_{0,e}( (1- \tau) (a + \i b) + \tau) \, {\rm d}\tau.
		\end{multline*}

		Arguing as before and using \eqref{Comparison-b-vs-a}, we get
		\begin{align*}
			&\frac{|a + 1+ \i b|^2}{2t} \int_0^1 (1- \tau) |G(t,(1- \tau)(a+1+\i b))| \, {\rm d}\tau
			 \leq 
			\frac{|a + 1 + \i b|^2 }{\sqrt{\pi t} ((a+1)^2 - b^2)} 
			\leq 	
			\frac{1}{\sqrt{\pi t}} \frac{\alpha^2 +1}{\alpha^2 -1}, 
			\\
			&
			\frac{|a - 1+ \i b|^2}{2t} \int_0^1 \tau |G(t,\tau(a-1+\i b))| \, {\rm d}\tau 
			 \leq \frac{ |a - 1 + \i b|^2 }{\sqrt{\pi t} ((a-1)^2 - b^2)} 
			\leq 
			\frac{1}{\sqrt{\pi t}} \frac{\alpha^2 +1}{\alpha^2 -1}, 
		\end{align*}
		which yields that, for all $a + \i b \in \overline\Omega_\alpha$, 
		$$
			|z_2(t, a+\i b)| \leq \frac{2}{\sqrt{\pi t}} \sqrt{\frac{\alpha^2 +1}{\alpha^2 -1}} \| y_{0,e} \|_{L^\infty(\Omega_\alpha)}.
		$$
		Together with estimate \eqref{Est-y-1}, this ends the proof of Lemma \ref{Lem-Est-y-2-and-z-2}.
	\end{proof}

	{\bf Conclusion.} The estimate \eqref{Est-Well-Posedness-X-alpha-d=1} in the $1$-d setting then follows immediately from the combination of \eqref{Stab-L-infty}, \eqref{Stab-W-1-infty}, and \eqref{Est-y-1-and-z-1}--\eqref{Est-y-2-and-z-2} (recall \eqref{Decomp-T-t-y0} and \eqref{Decomp-partial-x-T-t-y0}).
\end{proof}

\subsection{Proof of the estimates (\ref{Est-Well-Posedness-X-alpha}) in the general case.}\label{Subsec-Est-WP-any-d}
The goal of this section is to prove the following result
\begin{theorem}\label{Thm-Estimates-X-alpha-Gal-d}
	Let $d \in \N$ and $\alpha >1$. Then there exists $C>0$ such that for all $t > 0$ and for all $y_0 \in X_\alpha$, 
	\begin{equation}
		\label{Est-Well-Posedness-X-alpha-d}
		\| \T_t y_0 \|_{X_\alpha} + \sqrt{t } \| \nabla_x (\T_t y_0) \|_{X_\alpha}
		\leq C \| y_0\|_{X_\alpha}.
	\end{equation}
\end{theorem}

\begin{proof}
	As in \eqref{Stab-L-infty}--\eqref{Stab-W-1-infty}, using that the gaussian kernel $G_d$ satisfies, for some constant $C$, that  
	$$
		\forall t >0, 
		\quad 
		\| G_d(t) \|_{L^1(\R^d)} = 1, 
		\quad \text{ and } \quad 
		\| \nabla_x G_d(t) \|_{L^1(\R^d)} = \frac{C}{\sqrt{t}},
	$$
	we immediately deduce that there exists a constant $C>0$ such that for all $t>0$ and $y_0 \in \BUC(\R^d)$, $\mathbb{T}_t y_0 \in \BUC (\R^d)$ and $\nabla_x  \mathbb{T}_t y_0 \in (\BUC (\R^d))^d$ and 
	\begin{equation}
		\label{Est-WP-W-1-infty-d}
		\| \T_t y_0 \|_{L^\infty(\R^d)} +\sqrt{t} \| \nabla_x \T_t y_0\|_{L^\infty(\R^d)} \leq C \| y_0\|_{X_\alpha}.
	\end{equation}

	Let then $y_0 \in X_\alpha$, and let us focus on the estimates of $\T_t y_0$ in $\overline\Omega_\alpha$. 

{We use the generic notation $A+\i B\in\Omega_\alpha$ with $A$ and $B$ in $\R^d$ to denote the elements of $\Omega_\alpha$.}
As before, the natural holomorphic extension of $\T_t y_0$ is simply given by 
	\begin{equation}
		\T_t y_0(A + \i B) = \int_{ \R^d} G_d(t,(A+\i B)-x_0) y_0(x_0) \, {\rm d} x_0, \quad \text{ for } {A+iB \in \Omega_\alpha.} 
	\end{equation}
	Similarly, the natural holomorphic extension of $\nabla_x \T_t y_0$ is simply given by
	\begin{equation}
		\nabla_x (\T_t y_0)(A + \i B) = - \int_{\R^d} \frac{(A+\i B - x_0)}{2t} G_d(t,(A+\i B)-x_0) y_0(x_0) \, {\rm d} x_0, \quad \text{ for } {A+\i B \in \Omega_\alpha.} 
	\end{equation}
	%
	%
\textit{Invariance by real rotation.} Let us consider a rotation $\mathcal{R}_B$ of $\R^d$ satisfying
\begin{equation}
\label{eq:caracterisation de RB}
\mathcal{R}_B=|B|e_1, 
\end{equation}
where $e_1 =\, ^t\!(1, 0, \cdots, 0)$ is the first canonical vector of $\R^d$.  
Then, we have
\begin{equation}
\label{eq:changement de variable avec RB}
\mathbb{T}_ty_0(A+\i B)=[\mathbb{T}_t y_{0}^{B}](\mathcal{R}_BA+\i |B| e_1),
\text{ and }
\nabla_x (\T_t y_0)(A + \i B) = \nabla_x (\T_t y_0^B) (\mathcal{R}_BA+\i |B| e_1), 
\end{equation}
where $y_{0}^{B}:=y_{0}^{B}(\mathcal{R}_{B}^{-1}\cdot)$. Besides,  since $\mathcal{R}_B$ is a rotation of $\R^d$, it follow from the definitions of $\Omega_\alpha$ and $X_{\alpha}$ that
$$
\mathcal{R}_{B}^{-1}(\Omega_\alpha)=\Omega_\alpha\text{, }\  y_{0}^{B}\in X_\alpha\ \text{ and }\ \|y_{0}^{B}\|_{X_\alpha}=\|y_{0}\|_{X_\alpha}.
$$

To get estimates \eqref{Est-Well-Posedness-X-alpha-d}, it is then sufficient to prove that there exists a constant $C_\alpha > 0$ such that for all $y_0 \in X_\alpha$, for all $(A, B_1) \in \R^d \times \R$ satisfying
	$$
		\sqrt{A_1^2+ |A'|^2} + \alpha |B_1| \leq 1, 
	$$
	and for all $t >0$, 
\begin{equation}
\label{eq:goal 1}
|\mathbb{T}_t y_{0}(A+\i B_1 e_1)|\leq C_\alpha\|y_0\|_{X_\alpha},
\end{equation}
and 
\begin{equation}
\label{eq:goal 2}
\sqrt{t} |\nabla_x\mathbb{T}_t y_{0}(A+\i B_1 e_1 )|\leq C_\alpha\|y_0\|_{X_\alpha}.
\end{equation}

\textit{Estimates}. In the following, we write $x_0 \in \R^d$ as $x_0 = (x_{1}, x')$ for $x_1 \in \R$ and $x' \in \R^{d-1}$. This yields, for $t >0$ and $(A, B_1) \in \R^d \times \R$,
	\begin{equation*}
		\T_t y_{0}(A + \i B_1 e_1) 
		= 
		\int_{x' \in \R^{d-1}} G_{d-1}(t,A'-x') 
			y_1(t, A_1, B_1, x') 
		 {\rm d}x', 
\end{equation*}	
where
\begin{equation}
		y_1(t, A_1, B_1, x') 
		=
		\int_{x_1 \in \R} G_1(t, A_1 + \i B_1 - x_1) y_{0}(x_{1}, x') \, {\rm d}x_1,
		\quad \text{ for } x' \in \R^{d-1}, 
\end{equation}
	and 	
	\begin{multline*}
		\partial_{x_1} (\T_t y_{0})(A + \i B_1 e_1) 
		= 
		\int_{x' \in \R^{d-1}}  G_{d-1}(t,A'-x') 
			z_1 (t, A_1, B_1, x') 
		 {\rm d}x', 
		\\
		\text{ where } 
		z_1(t, A_1, B_1, x') 
		=
		-\int_{x_1 \in \R} \frac{(A_1+ \i B_1 - x_1)}{2t} G_1(t, A_1 + \i B_1 - x_1) y_{0}(x_{1}, x') \, {\rm d}x_1,
		\quad \text{ for } x' \in \R^{d-1}, 
	\end{multline*}
	while, for $j \in \{2, \cdots, d\}$,
	\begin{equation*}
		\partial_{x_j} (\T_t y_{0})(A + \i B_1 e_1) 
		= 
		- \int_{x' \in \R^{d-1}} \frac{(A_j-x_j)}{2t} G_{d-1}(t,A'-x') 
		y_1(t,A_1, B_1, x')		 {\rm d}x'. 
	\end{equation*}
	We shall thus estimate precisely $y_1$ and $z_1$ for $A+ \i B_1$ satisfying
	$$
		\sqrt{A_1^2+ |A'|^2} + \alpha |B_1| \leq 1.
	$$
	In order to do so, we shall remark that, since $y_{0} \in X_\alpha$, $y_{0}(\cdot, x')$ belongs to $\BUC(\R)$ and, if $|x'| < 1$, its restriction to $(-\sqrt{1 - |x'|^2}, \sqrt{1 - |x'|^2})$  admits an holomorphic expansion on 
	\begin{equation}
		\label{Def-Omega-Alpha-x'}
		\Omega_{\alpha,x'} := \enstq{ a_1+ \i b_1}{ a_1, b_1 \in \R^d \text{ satisfying } \sqrt{a_1^2 + |x'|^2} + \alpha |b_1| < 1}, 
	\end{equation}
	which is continuous in $\overline\Omega_{\alpha, x'}$.

	We then distinguish several cases to estimate $y_1$ and $z_1$:
	\begin{enumerate}

		\item When $|x' | \geq 1$; 

		\item When $|x'| < 1$
			\begin{enumerate}
				\item and $\sqrt{A_1^2 + |x'|^2} \geq 1$; 
				\item and $\sqrt{A_1^2 + |x'|^2} + \alpha |B_1| < 1$; 
				\item and $\sqrt{A_1^2 + |x'|^2} < 1$ and $\sqrt{A_1^2 + |x'|^2} + \alpha |B_1| \geq 1$.
		\end{enumerate}
	\end{enumerate}

	\noindent{\bf 1. The case $|x'| \geq 1$.} Then we write: 
	\begin{align*}
		|y_1(t,A_1, B_1, x') |
		&\leq \frac{1}{\sqrt{4 \pi t}} \int_{x_1 \in \R} \exp\left( \frac{1}{4t}(B_1^2 - (A_1 - x_1)^2 )\right) \, {\rm d}x_1 \, {\|y_{0} \|_{L^\infty(\R^d)}}
		\\
		&\leq\exp\left( \frac{B_1^2}{4t} \right) {\|y_{0} \|_{L^\infty(\R^d)}}. 
	\end{align*}
	Similarly, we get 
	\begin{equation*}
		|z_1(t,A_1, B_1, x') |
		\leq 
		\frac{C}{\sqrt{t}} \left(1+ \frac{|B_1|}{\sqrt{t}} \right)\exp\left( \frac{B_1^2}{4t} \right) {\|y_{0} \|_{L^\infty(\R^d)}}. 
	\end{equation*}
	Then, since $A' \in B_{\R^{d-1}}(1)$ and $|B_1 | \leq  (1 - |A'|)/\alpha$,
	\begin{align}
		& \int_{x' \in \R^{d-1}, \, |x'| \geq 1} G_{d-1}(t,A' - x')	|y_1(t,A_1, B_1, x') | \, {\rm d}x'
		\notag\\
		& \leq 
		\int_{x' \in \R^{d-1}, \, |x'| \geq 1} \frac{1}{(4 \pi t)^{(d-1)/2}} \exp\left( - \frac{|A' - x'|^2}{4t} \right)	|y_1(t,A_1, B_1, x') | \, {\rm d}x' 
		\notag\\
		&
		\leq C \left( 1 + \left(\frac{ 1 - |A'|}{\sqrt{t}}\right)^{d-2} \right)\exp\left( - \frac{(1- |A'|)^2}{4t} + \frac{B_1^2}{4t} \right) \|y_{0}\|_{L^\infty(\R^d)} 
		\notag\\
		& 
		\leq C \left( 1 + \left(\frac{ 1 - |A'|}{{\sqrt{t}}}\right)^{d-2} \right)\exp\left( - \left( 1 - \frac{1}{\alpha^2}\right)\frac{(1- |A'|)^2}{4{t}}  \right) \|y_{0}\|_{L^\infty(\R^d)}
		\notag\\
		& \leq C \| y_{0}\|_{L^\infty(\R^d)},
	\label{Est-Tt-y0-x2>1} 
	\end{align}
	where we used that 
	\begin{align*}
		& \int_{x' \in \R^{d-1}, \, |x'| \geq 1} \frac{1}{(4 \pi t)^{(d-1)/2}} \exp\left( - \frac{|A' - x'|^2}{4t} \right) \, {\rm d}x'
		\\
		&\leq 
		\int_{x' \in \R^{d-1}, \, |x'-A'| \geq 1- |A'|} \frac{1}{(4 \pi t)^{(d-1)/2}} \exp\left( - \frac{|A' - x'|^2}{4t} \right) \, {\rm d}x'
		\\
		& \leq |\S^{d-2}| \int_{r \geq \frac{1- |A'|}{\sqrt{t}}} \frac{1}{(4 \pi)^{(d-1)/2}} \exp\left( - \frac{r^2}{4} \right) \, r^{d-2} {\rm d}r. 
		\\
		&\leq C \left( 1 + \left(\frac{ 1 - |A'|}{\sqrt{t}}\right)^{d-2} \right)\exp\left( - \frac{(1- |A'|)^2}{4t} \right) ,
	\end{align*}
	for some constant $C$ independent of the parameters $A'$ and $t$.

	The above proof can be adapted in an easy manner to deal with the gradient terms: 
	\begin{align}
		& \int_{x' \in \R^{d-1}, \, |x'| \geq 1} G_{d-1}(t,A' - x') |z_1(t,A_1, B_1, x') | \, {\rm d}x'
		\leq \frac{C}{\sqrt{t}} \| y_0\|_{L^\infty(\R^d)},
		\label{Est-partial1-Tt-y0-x2>1} 
		\\
		& \int_{x' \in \R^{d-1}, \, |x'| \geq 1} \frac{|A' -x'|}{2t} G_{d-1}(t,A' - x')	|y_1(t,A_1, B_1, x') | \, {\rm d}x'
		\leq \frac{C}{\sqrt{t}} \| y_0\|_{L^\infty(\R^d)}.
		\label{Est-Nabla'-Tt-y0-x2>1} 
	\end{align}
	\bigskip

	\noindent{\bf 2. The case $|x'| < 1$.} We then write 
	\begin{align*}
		y_1(t, A_1, B_1, x') 
		=
		y_{1,1}(t, A_1, B_1, x')
		+ 
		y_{1,2}(t, A_1, B_1, x'),
		\\
		z_1(t, A_1, B_1, x') 
		=
		z_{1,1}(t, A_1, B_1, x')
		+ 
		z_{1,2}(t, A_1, B_1, x'),
	\end{align*}
	where 
	\begin{align*}
		y_{1,1}(t, A_1, B_1, x') & = \int_{x_1 \in \R, \, |x_1| \geq \sqrt{1-|x'|^2}} G_1(t, A_1 + \i B_1 - x_1) y_{0}(x_{1}, x') \, {\rm d}x_1, 
		\\
		y_{1,2}(t, A_1, B_1, x') & = \int_{x_1 \in \R, \, |x_1| < \sqrt{1-|x'|^2}} G_1(t, A_1 + \i B_1 - x_1) y_{0}(x_{1}, x') \, {\rm d}x_1, 
	\end{align*}
	and
	\begin{align*}
		z_{1,1}(t, A_1, B_1, x') & = - \int_{x_1 \in \R, \, |x_1| \geq \sqrt{1-|x'|^2}} \frac{(A_1+\i B_1 -x_1)}{2t} G_1(t, A_1 + \i B_1 - x_1) y_{0}(x_{1}, x') \, {\rm d}x_1, 
		\\
		z_{1,2}(t, A_1, B_1, x') & = \int_{x_1 \in \R, \, |x_1| < \sqrt{1-|x'|^2}} \frac{(A_1+\i B_1 -x_1)}{2t} G_1(t, A_1 + \i B_1 - x_1) y_{0}(x_{1}, x') \, {\rm d}x_1, 
	\end{align*}
	For $y_{1,1}$, proceeding as before, 
	\begin{align}
		& 
		\int_{x' \in \R^{d-1}, \, |x'| \leq 1} G_{d-1}(t,A' - x')	|y_{1,1}(t,A_1, B_1, x') | \, {\rm d}x'
		\notag\\
		&
		\leq \int_{x \in \R^d,\, |x| \geq 1} \frac{1}{(4 \pi t)^{d/2}} \exp\left( \frac{1}{4t}(B_1^2 - |x- A|^2 ) \right) \, {\rm d}x \| y_0 \|_{L^\infty(\R^d)}
		\notag\\ 
		& \leq C \left(1+ \left(\frac{1 - |A|}{\sqrt{t}}\right)^{d-1} \right)  \exp\left( \frac{1}{4t}(B_1^2 - (1- |A|)^2 ) \right)  \| y_0 \|_{L^\infty(\R^d)}
		\notag\\
		& \leq C \left(1+ \left(\frac{1 - |A|}{\sqrt{t}}\right)^{d-1} \right)  \exp\left( - \left(1 - \frac{1}{\alpha^2} \right)\frac{ (1- |A|)^2 }{4t}\right) \| y_0 \|_{L^\infty(\R^d)}
		\notag\\
		& \leq C \| y_0 \|_{L^\infty(\R^d)}.
		\label{Est-y11-x2<1}
	\end{align}	
	We get similarly that 
	\begin{align}
		& \int_{x' \in \R^{d-1}, \, |x'| \leq 1} \frac{|A' - x'|}{t} G_{d-1}(t,A' - x')	|y_{1,1}(t,A_1, B_1, x') | \, {\rm d}x'
		\notag\\
		&
		\leq \int_{x \in \R^d,\, |x| \geq 1} \frac{1}{(4 \pi t)^{d/2}} \frac{|A - x|}{t} \exp\left( \frac{1}{4t}(B_1^2 - |x- A|^2 ) \right) \, {\rm d}x \| y_0 \|_{L^\infty(\R^d)}
		\notag\\
		&
		\leq \int_{x \in \R^d,\, |x-A| \geq 1-|A|} \frac{1}{(4 \pi t)^{d/2}} \frac{|A - x|}{t} \exp\left( \frac{1}{4t}(B_1^2 - |x- A|^2 ) \right) \, {\rm d}x \| y_0 \|_{L^\infty(\R^d)}
		\notag\\
		& 
		\leq 
		\frac{C}{\sqrt{t}} \left(1+ \left(\frac{1 - |A|}{\sqrt{t}}\right)^{d} \right)  \exp\left( - \left(1 - \frac{1}{\alpha^2} \right)\frac{ (1- |A|)^2 }{4t}\right) \| y_0 \|_{L^\infty(\R^d)}
		\notag \\
		& \leq \frac{C}{\sqrt{t}} \| y_0\|_{L^\infty(\R^d)}, 
		\label{Est-nabla'-y-11}	
	\end{align}
        and, using that $|B_1| \leq (1 - |A|)/\alpha$, 
        \begin{align}
            	& \int_{x' \in \R^{d-1}, \, |x'| \leq 1} G_{d-1}(t,A' - x')	|z_{1,1}(t,A_1, B_1, x') | \, {\rm d}x'
            	\notag\\
            	&
            	\leq \int_{x \in \R^d,\, |x| \geq 1} \frac{1}{(4 \pi t)^{d/2}} \frac{|A_1 - x_1| + |B_1|}{t} \exp\left( \frac{1}{4t}(B_1^2 - |x- A|^2 ) \right) \, {\rm d}x \| y_0 \|_{L^\infty(\R^d)}
            	\notag\\
            	&
            	\leq \int_{x \in \R^d,\, |x-A| \geq 1-|A|} \frac{1}{(4 \pi t)^{d/2}} \frac{|A - x|+ |B_1|}{t} \exp\left( \frac{1}{4t}(B_1^2 - |x- A|^2 ) \right) \, {\rm d}x \| y_0 \|_{L^\infty(\R^d)}
            	\notag\\
            	& 
            	\leq 
            	\frac{C}{\sqrt{t}} \left(1+ \left(\frac{1 - |A|}{\sqrt{t}}\right)^{d} + \frac{|B_1|}{\sqrt{t}}\left(\frac{1 - |A|}{\sqrt{t}}\right)^{d-1}  \right)  \exp\left( - \left(1 - \frac{1}{\alpha^2} \right)\frac{ (1- |A|)^2 }{4t}\right) \| y_0 \|_{L^\infty(\R^d)}
            	\notag \\
            	& \leq \frac{C}{\sqrt{t}} \| y_0\|_{L^\infty(\R^d)}.
            	\label{Est-z-11} 	
        \end{align}
        
        It thus remains to study $y_{1,2}$ (and the corresponding term $z_{1,2}$). This will lead to the cases mentioned above depending on the range of $\sqrt{A_1^1 + |x'|^2}$ under consideration. 
        \medskip
        
        \noindent{\bf 2.a. The case $|x'|< 1$ and $\sqrt{A_1^2 + |x'|^2} \geq 1$, i.e. $1 - A_1^2 \leq |x'|^2 \leq 1$.} In this case, $A_1^2 \geq 1 - |x'|^2$. We thus have 
        \begin{align*}
	        	|y_{1,2}(t,A_1, B_1, x') |
        		&\leq \frac{1}{\sqrt{4 \pi t}} \int_{|x_1| \leq \sqrt{1-|x'|^2}} \exp\left( \frac{1}{4t}(B_1^2 - (A_1 - x_1)^2 )\right) \, {\rm d}x_1 \|y_{0}(\cdot, x')\|_{L^\infty(\R)}
        		\\
	        	&\leq\exp\left( \frac{B_1^2}{4t} - \frac{(|A_1|- \sqrt{1-|x'|^2})^2}{4t} \right)\|y_{0}(\cdot, x')\|_{L^\infty(\R)},  
        \end{align*}
        and, similarly, 
        $$
	        	|z_{1,2}(t,A_1, B_1,x')| 
        		\leq 
        		\frac{C}{\sqrt{t}} \left( 1 + \frac{|B_1|}{\sqrt{t}} \right)\exp\left( \frac{B_1^2}{4t} - \frac{(|A_1|- \sqrt{1-|x'|^2})^2}{4t} \right)\|y_{0}(\cdot, x')\|_{L^\infty(\R)}.
        $$
        Then, 
        \begin{align*}
	        	&\int_{x' \in \R^{d-1}, \, |x'| \in (\sqrt{1-|A_1|^2},1)} G_{d-1}(t,A' - x')	|y_{1,2}(t,A_1, B_1, x') | \, {\rm d}x'
        		\\
	        	& \leq 
        		\int_{x' \in \R^{d-1}, \, |x'| \in (\sqrt{1-|A_1|^2},1)} \frac{1}{(4 \pi t)^{(d-1)/2}} \exp\left( - \frac{|A' - x'|^2}{4t} + \frac{B_1^2}{4t} - \frac{(|A_1|- \sqrt{1-|x'|^2})^2}{4t} \right) \, {\rm d}x'	
	        	\| y_0\|_{L^\infty(\R^d)}
        		\\
	        	& \leq
        		\exp\left( \frac{B_1^2}{4t}  \right) \int_{x' \in \R^{d-1}, \, |x'| \in (\sqrt{1-|A_1|^2},1)}\frac{1}{(4 \pi t)^{(d-1)/2}} \exp\left( - \frac{|A' - x'|^2}{4t} - \frac{(|A_1|- \sqrt{1-|x'|^2})^2}{4t} \right) \, {\rm d}x'	\| y_0\|_{L^\infty(\R^d)}.
        \end{align*}
        We then look at the function 
        $$
	        	f(x') = |A' - x'|^2 + (|A_1| - \sqrt{1-|x'|^2})^2, 
        $$
        defined for $x' \in B_{\R^{d-1}}(1)$. Since $x' \in B_{\R^{d-1}}(1) \mapsto (x',  \sqrt{1 - |x'|^2})$ is a parametrization of the half sphere $\S^d_+$, we immediately get that 
        $$
        		\inf_{B_{\R^{d-1}}(1)} f = (1 - |A|)^2, 
        $$
         while 
        $$
	        	\forall x' \in B_{\R^{d-1}}(1), \qquad 	D^2 f(x') \geq \frac{2 |A_1|}{\sqrt{1 - |x'|^2}}. 
        $$
        Accordingly, 
        $$
        		\forall x' \in B_{\R^{d-1}}(1) \text{ with } |x'| \geq \sqrt{1-A_1^2}, \qquad
        	D^2 f(x') \geq 2, 
        $$
        and thus, setting 
        $$
        	x'_{*} \in \text{Argmin\,}\enstq{f (x')}{ x' \in B_{\R^{d-1}}(1)\setminus B_{\R^{d-1}}\left(\sqrt{1-A_1^2}\right)},	
        $$ 
        we get
        $$
        		\forall x' \in B_{\R^{d-1}}(1) \text{ with } |x'| \geq \sqrt{1-A_1^2}, 
        		\qquad 
        		f(x') \geq f(x'_{*}) + |x'-x'_{*}|^2 \geq (1 - |A|)^2 + |x'- x'_{*}|^2.
        $$
        It follows that 
        \begin{multline*}
	        	\exp\left( \frac{B_1^2}{4t}  \right) \int_{x' \in \R^{d-1}, \, |x'| \in (\sqrt{1-|A_1|^2},1)} \frac{1}{(4 \pi t)^{(d-1)/2}} \exp\left( - \frac{|A' - x'|^2}{4t} - \frac{(|A_1|- \sqrt{1-|x'|^2})^2}{4t} \right) \, {\rm d}x'
        		\\
	        	\leq 
        		C \exp\left( \frac{1}{4t} ( B_1^2 - (1- |A|)^2)  \right) 
	        	\leq 
        		C \exp\left( - \left(1 - \frac{1}{\alpha^2} \right)\frac{1}{4t} (1- |A|)^2  \right),  
        \end{multline*}
        and, consequently,
        \begin{equation}
        		\label{Est-y12-x2>sqrt(1-a_1^2)}
        		\int_{x' \in \R^{d-1}, \, |x'| \in (\sqrt{1-|A_1|^2},1)} G_{d-1}(t,A' - x')	|y_{1,2}(t,A_1, B_1, x') | \, {\rm d}x'
        		\leq C \| y_0\|_{L^\infty(\R^d)}.
        \end{equation}
        In fact, we can go further and show that
        $$
	        	x'_* = \frac{A'}{|A'|} \sqrt{1 - A_1^2}, \quad \text{ and } f(x'_*) = \left(\sqrt{1 - A_1^2} - |A'|\right)^2.
        $$
        Using $f(x_*') \geq (1 - |A|)^2$, we thus deduce 
        \begin{align*}
        		& \int_{x' \in \R^{d-1}, \, |x'| \in (\sqrt{1-|A_1|^2},1)} \frac{1}{(4 \pi t)^{(d-1)/2}} \frac{|x' - A'|}{2t} \exp\left( - \frac{|A' - x'|^2}{4t} - \frac{(|A_1|- \sqrt{1-|x'|^2})^2}{4t} \right) \, {\rm d}x'
            	\\
            	& \leq 
            	\int_{x' \in \R^{d-1}, \, |x'| \in (\sqrt{1-|A_1|^2},1)} \frac{1}{(4 \pi t)^{(d-1)/2}} \frac{|x' - x_*'|}{2t} \exp\left( - \frac{1}{4t} (f(x_*') + |x' -x'_*|^2 )\right) \, {\rm d}x'
            	\\
            	& \quad +
            	\int_{x' \in \R^{d-1}, \, |x'| \in (\sqrt{1-|A_1|^2},1)} \frac{1}{(4 \pi t)^{(d-1)/2}} \frac{|x'_* - A'|}{2t} \exp\left( - \frac{1}{4t} (f(x_*') + |x' -x'_*|^2 )\right) \, {\rm d}x'
            	\\
            	&\leq 
            	\frac{C }{\sqrt{t}} \left( 1+\sqrt{ \frac{f(x_*')}{t}}\right) \exp \left( - \frac{1}{4t} f(x_*') \right)
            	\leq 
            	\frac{C }{\sqrt{t}} \exp \left( - \frac{1}{4t} \frac{f(x_*')}{\alpha} \right)
            	\leq 
            	\frac{C }{\sqrt{t}} \exp \left( - \frac{1}{4t \alpha} (1 - |A|)^2 \right).
        \end{align*}
	for some constant $C$ independent of the parameter $A_1$, $A'$, $B_1$ and $t$. As a consequence, using again $B_1^2 \leq (1- |A|)^2/\alpha^2$, we deduce
        \begin{multline}
            	 \int_{x' \in \R^{d-1}, \, |x'| \in (\sqrt{1-|A_1|^2},1)} \frac{|A'-x'|}{2t} G_{d-1}(t,A' - x')	|y_{1,2}(t,A_1, B_1, x') | \, {\rm d}x'
            	\\
            	\leq \frac{C}{\sqrt{t}} \exp\left( - \frac{1}{4t} \left(\frac{1}{\alpha} - \frac{1}{\alpha^2} \right) (1 - |A|)^2 \right) \| y_0\|_{L^\infty(\R^d)}
            	\leq \frac{C}{\sqrt{t}}  \| y_0\|_{L^\infty(\R^d)},
            	\label{Est-nabla-x'-y12-x2>sqrt(1-a_1^2)}
        \end{multline}
	{for some constant $C$ independent of the parameter $A_1$, $A'$, $B_1$ and $t$.} Using similar computations as the ones to obtain \eqref{Est-y12-x2>sqrt(1-a_1^2)} and the fact that $|B_1| \leq (1 - |A|)/\alpha$, we get 
        \begin{equation}
	        	\label{Est-z12-x2>sqrt(1-a_1^2)}
        		\int_{x' \in \R^{d-1}, \, |x'| \in (\sqrt{1-|A_1|^2},1)} G_{d-1}(t,A' - x')	|z_{1,2}(t,A_1, B_1, x') | \, {\rm d}x'
	        	\leq \frac{C}{\sqrt{t}} \| y_0\|_{L^\infty(\R^d)},
        \end{equation}
	{for some constant $C$ independent of the parameter $A_1$, $A'$, $B_1$ and $t$.}       
        \medskip
        
        \noindent{\bf 2.b. The case $|x'| < 1$ and $\sqrt{A_1^2 + |x'|^2} + \alpha |B_1| < 1$, i.e. $|x'|^2 < (1 - \alpha |B_1|)^2 - A_1^2$.} In this case, $A_1 + \i B_1$ belongs to the set $\Omega_{\alpha,x'}$ in \eqref{Def-Omega-Alpha-x'}. It is thus very natural to do as in $1$-d and to modify the contour $[-\sqrt{1 - |x'|^2}, \sqrt{1 - |x'|^2}]$ into the union of the two contours 
        \begin{align*}
	        	\Gamma_{\ell} & := \enstq{ - (1- \tau) \sqrt{1-|x'|^2} + \tau (A_1 + \i B_1)}{ \tau \in [0,1]}, 
        		\\
	        	\Gamma_{r} & := \enstq{  (1- \tau) (A_1 + \i B_1)+ \tau \sqrt{1-|x'|^2}}{ \tau \in [0,1]}.  
        \end{align*}
        It can be easily checked that the triangle delimited by $\Gamma_{ \ell} \cup \Gamma_{r } \cup [-\sqrt{1 - |x'|^2}, \sqrt{1 - |x'|^2}]$ belongs to $\overline\Omega_{\alpha, x'}$ when $\sqrt{A_1^2 + |x'|^2} + \alpha |B_1| < 1$, so that we can use Cauchy's formula (in $x_1$) to modify the contour $[-\sqrt{1 - |x'|^2}, \sqrt{1 - |x'|^2}]$ into $\Gamma_{\ell} \cup \Gamma_{r}$.
        We thus get 
        \begin{multline*}
        		y_{1,2}(t, A_1, B_1,x') =
	        	\\
        		\frac{(A_1+ \sqrt{1 -|x'|^2} + \i B_1)}{\sqrt{4 \pi t}} \int_{0}^1 \exp\left( - \frac{(1-\tau)^2}{4t} (A_1 + \sqrt{1 - |x'|^2} + \i B_1)^2 \right) y_{0,e}( - (1- \tau) \sqrt{1-|x'|^2} + \tau (A_1 + \i B_1), x') \, {\rm d}\tau
	        	\\
        		- 
	        	\frac{(A_1- \sqrt{1 -|x'|^2} + \i B_1)}{\sqrt{4 \pi t}} \int_{0}^1 \exp\left( - \frac{\tau^2}{4t} (A_1 - \sqrt{1 - |x'|^2} + \i B_1)^2 \right) y_{0,e}( - (1- \tau) \sqrt{1-|x'|^2} + \tau (A_1 + \i B_1), x') \, {\rm d}\tau.
        \end{multline*}
        As done in the $1$-d case, we then check that 
        \begin{align*}
        		&
	        	\frac{|A_1+ \sqrt{1 -|x'|^2} + \i B_1|}{t} \int_{0}^{1} \exp\left( - \frac{(1-\tau)^2}{4t} ( (A_1 + \sqrt{1 - |x'|^2})^2- B_1^2 )\right) \, {\rm d}\tau \leq C,
        		\\
	        	&
        		\frac{|A_1- \sqrt{1 -|x'|^2} + \i B_1|}{t} \int_{0}^{1} \exp\left( - \frac{\tau^2}{4t} ((A_1 - \sqrt{1 - |x'|^2})^2 - B_1^2) \right) {\rm d}\tau \leq C, 
        \end{align*}
        for some constant $C$ independent of $A_1, A', B_1$ and $x'$ and $t$. The key inequality is the following one:
        \begin{equation}
	        	\forall (A_1,x') \in \R \times \R^{d-1} \text{ with } |x'| \leq 1, \qquad
        		\left( 1 - \sqrt{A_1^2+ |x'|^2}\right)^2 \leq \left( |A_1| - \sqrt{1 - |x'|^2}\right)^2.
        \end{equation}
        Indeed, this can be deduced from the following equivalences 
        \begin{align*}
            	&\left( 1 - \sqrt{A_1^2+ |x'|^2}\right)^2 \leq \left( |A_1| - \sqrt{1 - |x'|^2}\right)^2
            	\\
            	& \Leftrightarrow  
            	1 + A_1^2 + |x'|^2 - 2 \sqrt{A_1^2 + |x'|^2} \leq A_1^2 + 1 - |x'|^2 - 2 |A_1| \sqrt{1 - |x'|^2}
            	\\ 
            	& \Leftrightarrow  
            	|x'|^2 +  |A_1| \sqrt{1 - |x'|^2} \leq \sqrt{A_1^2 + |x'|^2}
            	\\ 
            	& \Leftrightarrow  
            	|x'|^4 + A_1^2 - A_1^2 |x'|^2 + 2 |A_1| |x'|^2 \sqrt{1 - |x'|^2} \leq A_1^2 + |x'|^2
            	\\ 
            	& \Leftrightarrow  
            	|x'|^2 - A_1^2 + 2 |A_1| \sqrt{1 -|x'|^2} \leq 1	
            	\\
            	& \Leftrightarrow  
            	0 \leq  (1 -|x'|^2) + A_1^2 - 2 |A_1| \sqrt{1 - |x'|^2} \leq 0
            	\\
            	& \Leftrightarrow  
            	0 \leq ( |A_1| - \sqrt{1 - |x'|^2} )^2 .
        \end{align*}
        Accordingly, using that 
        $$
	        	|B_1|^2 \leq \frac{1}{\alpha^2} \left( 1 - \sqrt{A_1^2+ |x'|^2}\right)^2, 
        $$
        we get that 
        \begin{align*}
            	& |B_1|^2 - (A_1 + \sqrt{1 - |x'|^2})^2 \leq - \left( 1 - \frac{1}{\alpha^2}\right) (A_1 + \sqrt{1 - |x'|^2})^2, 
            	\\
            	& |B_1|^2 - (A_1 - \sqrt{1 - |x'|^2})^2 \leq - \left( 1 - \frac{1}{\alpha^2}\right) (A_1 - \sqrt{1 - |x'|^2})^2, 
            	\\
            	& |A_1+ \sqrt{1 -|x'|^2} + \i B_1| \leq {C} | A_1 + \sqrt{1 - |x'|^2} |, 
            	\\
            	& |A_1- \sqrt{1 -|x'|^2} + \i B_1| \leq {C} | A_1 - \sqrt{1 - |x'|^2} |, 
        \end{align*}
        As in $1$-d, it follows that, for some constant $C$ independent of $t, A_1, B_1$ and $x'$,
        \begin{equation}
            	|y_{1,2}(t, A_1, B_1, x')| 
            	\leq 
		{C}\| y_{0,e}(\cdot, x' )\|_{L^\infty(\Omega_{\alpha,x'})} 
            	\leq 
            	{C} \| y_{0,e} \|_{L^\infty(\Omega_\alpha)}.
        \end{equation}
        This implies that, for some $C$ independent of $t$, $A_1, A'$ and $B_1$ with $A + \i B_1 e_1 \in \overline\Omega_\alpha$, 
        \begin{equation}
            	\label{Est-y12-x2-petit}
            	\int_{x' \in \R^{d-1}, \, |x'| \leq \sqrt{(1- \alpha |B_1|)^2 - A_1^2}}
            	G_{d-1}(t,A'-x') |y_{1,2}(t, A_1, B_1, x')| \, {\rm d}x'
            	\leq 
		{C} \| y_{0,e}\|_{L^\infty(\Omega_\alpha)}, 
        \end{equation}
        and
        \begin{equation}
            	\label{Est-nabla-x'-y12-x2-petit}
            	\int_{x' \in \R^{d-1}, \, |x'| \leq \sqrt{(1- \alpha |B_1|)^2 - A_1^2}}
            	\frac{|A'-x'|}{2t} G_{d-1}(t,A'-x') |y_{1,2}(t, A_1, B_1, x')| \, {\rm d}x'
            	\leq 
            	\frac{{C}}{\sqrt{t}} \| y_{0,e} \|_{L^\infty(\Omega_\alpha)}. 
        \end{equation}

        Regarding $z_{1,2}$, the same estimates as above can be done (similarly as well as in the $1$-d case), and we get 
        \begin{equation}
            	\label{Est-z12-x2-petit}
            	\int_{x' \in \R^{d-1}, \, |x'| \leq \sqrt{(1- \alpha |B_1|)^2 - A_1^2}}
            	G_{d-1}(t,A'-x') |z_{1,2}(t, A_1, B_1, x')| \, {\rm d}x'
            	\leq 
            	\frac{{C}}{\sqrt{t}} \| y_{0,e} \|_{L^\infty(\Omega_\alpha)}, 
        \end{equation}
	{for some constant {$C$} independent of $t, A_1, A'$ and $B_1$.}     
        
        \medskip
        
        \noindent {\bf 2.c. Case $|x'| < 1$, $\sqrt{A_1^2 + |x'|^2} < 1$ and $\sqrt{A_1^2 + |x'|^2} + \alpha |B_1| \geq 1$, i.e. $(1 -\alpha |B_1|)^2 - A_1^2 \leq |x'|^2 < 1 - A_1^2$.} This corresponds to a new case compared to the $1$-d setting. Note that this corresponds to an intermediate case between the last two cases. This will lead to the use of a Cauchy's formula in $1$-d, but there will still remain some exponential growth that will be cancelled by the term $- |A' - x'|^2$.
        
        To be more precise, we use Cauchy's formula in $1$-d to push the contour $(-\sqrt{1 - |x'|^2} , \sqrt{1 - |x'|^2})$ to the half boundary of $\Omega_{\alpha, x'}$ whose imaginary part has the same sign of $B_1$. If $\epsilon_B := \text{sign}(B_1)$, we consider the contour 
        $$
      	  	\Gamma := \enstq{ a_1 + \i b_1(a_1)}{ a_1 \in [-\sqrt{1-|x'|^2}, \sqrt{1-|x'|^2}] \text{ and } b_1(a_1) = \frac{\epsilon_B}{\alpha} \left(1- \sqrt{a_1^2+|x'|^2} \right) }.
        $$
        Then
        \begin{multline*}
            	y_{1,2}(t,A_1, B_1, x') 
            	\\
            	= 
            	\int_{- \sqrt{1-|x'|^2}}^{\sqrt{1-|x'|^2}}
            		\frac{1}{(4 \pi t)^{1/2}}\exp\left(- \frac{1}{4t} (A_1 + \i B_1 - (a_1 +\i b_1(a_1)))^2  \right)  y_{0,e} (a_1+\i b_1(a_1),x')  
            		\left( 1 - \i \frac{a_1}{\alpha \sqrt{a_1^2 + |x'|^2}}\right) \, da_1.
        \end{multline*}
        Accordingly, 
        \begin{equation*}
            	|y_{1,2}(t,A_1, B_1, x') |
            	\leq C
            	\int_{- \sqrt{1-|x'|^2}}^{\sqrt{1-|x'|^2}}
            		\frac{1}{t^{1/2}}\exp\left(- \frac{1}{4t} ((A_1 -a_1)^2 - (B_1 - b_1(a_1))^2)  \right)  
            	da_1
            	\| y_{0,e}\|_{L^\infty(\Omega_\alpha)}.
        \end{equation*}
        We then use the following estimate:
        \begin{align}
         	&(A_1 -a_1)^2 - (B_1 - b_1(a_1))^2
            	\notag\\
            	&= 
            	\left(1- \frac{1}{\alpha^2} \right)(A_1 - a)^2
            	+
            	\frac{1}{\alpha^2} 
            	\left(
            	(A_1 - a_1)^2 - (1 - \alpha |B_1|)^2 -(a_1^2+ |x'|^2) + 2 (1 - \alpha |B_1|) \sqrt{a_1^2 + |x'|^2}  \right)
            	\notag\\
            	& = 
            	\left(1- \frac{1}{\alpha^2} \right)(A_1 - a_1)^2 
            	+
            	\frac{1}{\alpha^2} 
            	\left(
            	A_1^2 -2 A_1 a_1 - (1 - \alpha |B_1|)^2 -|x'|^2+ 2 (1 - \alpha |B_1|) \sqrt{a_1^2 + |x'|^2}  \right)
            	\notag\\
            	& \geq 
            	\left(1- \frac{1}{\alpha^2} \right)(A_1 - a_1)^2 
            	+ 
            	\frac{1}{\alpha^2} 
            	\inf_{a_1}
            	\left\{ 
            	A_1^2 -2 A_1 a_1 - (1 - \alpha |B_1|)^2 -|x'|^2+ 2 (1 - \alpha |B_1|) \sqrt{a_1^2 + |x'|^2}  		\right\}
            	\\
            	\label{Algebra-Contour-Gamma}
            	& \geq 
            	\left(1- \frac{1}{\alpha^2} \right)(A_1 - a_1)^2 
            	+ 
            	\frac{1}{\alpha^2} 
            	\left( A_1^2  - (1 - \alpha |B_1|)^2 -|x'|^2+ 2 |x'| \sqrt{ (1 - \alpha |B_1|)^2 - A_1^2} \right).
        \end{align}
        It follows that 
        \begin{equation*}
            	|y_{1,2}(t,A_1, B_1, x') |
            	\leq {C} \exp\left( \frac{1}{4 t \alpha^2} 
            	\left((1 - \alpha |B_1|)^2 -A_1^2 + |x'|^2 - 2 |x'| \sqrt{ (1 - \alpha |B_1|)^2 - A_1^2} \right)
            \right)
            	\| y_{0,e}\|_{L^\infty(\Omega_\alpha)},
        \end{equation*}
	{for some constant $C$ independent of $t, A_1, A'$ and $B_1$.} 
        Accordingly, 
        \begin{align*}
            	& 
            	G_{d-1}(t,A'-x') |y_{1,2}(t, A_1, B_1, x')|
            	\\
            	& \leq 
            	\frac{{C}}{t^{(d-1)/2}} 
            	\exp
            		\left(
            			\frac{1}{4t}
            			\left(
            				- |x' - A'|^2 + \frac{1}{\alpha^2} 
            				\left(
            					(1 - \alpha |B_1|)^2 -A_1^2 + |x'|^2 - 2 |x'| \sqrt{ (1 - \alpha |B_1|)^2 - A_1^2} 		
            				\right)
            			\right)
            		\right)
            	\| y_{0,e} \|_{L^\infty(\Omega_\alpha)}
            	\\
            	& \leq 
            	\frac{{C} }{t^{(d-1)/2}} 
            	\exp
            		\left(
            			\frac{1}{4t}
            			\left(
            				- \left(1- \frac{1}{\alpha^2} \right)|x' - A'|^2 
            				+ 
            				\frac{1}{\alpha^2} h_\alpha(A, B_1,x')
            			\right)
            		\right)
            	\| y_{0,e} \|_{L^\infty(\Omega_\alpha)}. 
        \end{align*}
        where 
        \begin{align*}
            	h_\alpha (A, B, x')	
            	= 
            	- |x' - A'|^2 + (1 - \alpha |B_1|)^2 -A_1^2 + |x'|^2 - 2 |x'| \sqrt{ (1 - \alpha |B_1|)^2 - A_1^2} 		
            	\\
            	= (1 - \alpha |B_1|)^2 -(A_1^2 +|A'|^2) + 2 A' \cdot x' - 2 |x'| \sqrt{ (1 - \alpha |B_1|)^2 - A_1^2}. 
        \end{align*}
        Since $A + \i B_1 e_1 \in \Omega_\alpha$, $A_1^2 + |A'|^2 \leq (1 - \alpha |B_1|)^2$, and thus
        $$
	        	h_\alpha(A, B, x') \leq 2 A' \cdot x' - 2 |A'| |x'| \leq 0.
        $$
        We then immediately deduce that 
        \begin{equation}
        		\label{Est-y12-x2-intermediate}
	        	\int_{x' \in \R^{d-1}, \, |x'| \in (\sqrt{(1- \alpha |B_1|)^2 - A_1^2}, \sqrt{1 - A_1^2})}
		G_{d-1}(t,A'-x') |y_{1,2}(t, A_1, B_1, x')| \, {\rm d}x'
        		\leq 
		{C} \|y_{0,e}\|_{L^\infty(\Omega_\alpha)},
        \end{equation}
        and
        \begin{equation}
	        	\label{Est-nabla-x'-y12-x2-intermediate}
        		\int_{x' \in \R^{d-1}, \, |x'| \in (\sqrt{(1- \alpha |B_1|)^2 - A_1^2}, \sqrt{1 - A_1^2})}
	        	\frac{|A' - x'|}{2t}G_{d-1}(t,A'-x') |y_{1,2}(t, A_1, B_1, x')| \, {\rm d}x'
        		\leq 
	        	\frac{{C}}{\sqrt{t}} \|y_{0,e}\|_{L^\infty(\Omega_\alpha)},
        \end{equation}
	{for some constant {$C$} independent of $t, A_1, A'$ and $B_1$.} 
        To estimate $z_{2,1}$, we start as above by moving the contour of integration from $(-\sqrt{1 - |x'|^2} , \sqrt{1 - |x'|^2})$ to $\Gamma$, and then by noticing that $|A_1 + \i B_1 -(a_1 + \i b_1(a_1))| \leq |A_1 - a_1| + |B_1 - b_1(a_1)|$ and the estimate 
        $$
        		\alpha^2 |B_1 - b_1(a_1)|^2 \leq |A_1 - a_1|^2 + g(A_1, B_1,x'),
 $$
where  $$
 g(A_1, B_1,x')= (1 - \alpha |B_1|)^2 -A_1^2 + |x'|^2 - 2 |x'| \sqrt{ (1 - \alpha |B_1|)^2 - A_1^2}, 
        $$
        (recall \eqref{Algebra-Contour-Gamma}), we obtain that 
        \begin{align*}
            	|z_{1,2}(t,A_1, B_1,x')| 
            	& \leq 
            	\frac{{C}}{\sqrt{t}} \left(1+ \sqrt{\frac{g(A_1,B_1,x')}{t}} \right) \exp\left(\frac{1}{4 t \alpha^2} g(A_1, B_1, x') \right) \|y_{0,e}\|_{L^\infty(\Omega_\alpha)}
            	\\
            	&\leq
            	\frac{{C}}{\sqrt{t}} \exp\left(\frac{1}{4 t \alpha} g(A_1, B_1, x') \right) \|y_{0,e}\|_{L^\infty(\Omega_\alpha)},            		
        \end{align*}
	{for some constant $C$ independent of $t, A_1, A'$ and $B_1$.} 
        Arguing as previously for \eqref{Est-y12-x2-intermediate}, we deduce 
        \begin{equation}
            	\label{Est-z12-x2-intermediate}
            	\int_{x' \in \R^{d-1}, \, |x'| \in (\sqrt{(1- \alpha |B_1|)^2 - A_1^2}, \sqrt{1 - A_1^2})}
            	G_{d-1}(t,A'-x') |z_{1,2}(t, A_1, B_1, x')| \, {\rm d}x'
            	\leq 
            	\frac{C_\alpha}{\sqrt{t}} \|y_{0,e}\|_{L^\infty(\Omega_\alpha)}.
        \end{equation}
        \medskip
        
        \noindent {\bf Conclusion.} Then, \eqref{eq:goal 1} follows by combining the estimates \eqref{Est-Tt-y0-x2>1}, \eqref{Est-y11-x2<1}, \eqref{Est-y12-x2>sqrt(1-a_1^2)}, \eqref{Est-y12-x2-petit} and \eqref{Est-y12-x2-intermediate}.
Similarly, combining the estimates \eqref{Est-partial1-Tt-y0-x2>1}, \eqref{Est-z-11}, \eqref{Est-z12-x2>sqrt(1-a_1^2)}, \eqref{Est-z12-x2-petit} and \eqref{Est-z12-x2-intermediate}, there is a constant $C_{1}$ such that for all $t \geq 0$ and $A + \i B_1 e_1 \in \overline\Omega_\alpha$,
        \begin{equation*}
	        	\sqrt{t} |\partial_1 \T_t y_{0}( A + \i B_1 e_1)| 
        		\leq 
	        	{C_{1}} \| y_0\|_{X_\alpha}.
        \end{equation*}
        and, combining the estimates \eqref{Est-Nabla'-Tt-y0-x2>1}, \eqref{Est-nabla'-y-11}, \eqref{Est-nabla-x'-y12-x2>sqrt(1-a_1^2)}, \eqref{Est-nabla-x'-y12-x2-petit} and \eqref{Est-nabla-x'-y12-x2-intermediate}, there is a constant $C_{2}$ such that, for all $j \in \{2, \cdots, d\}$, for all $t \geq 0$ and $A + \i B_1 e_1 \in \overline\Omega_\alpha$,
        \begin{equation*}
        		\sqrt{t} |\partial_j \T_t y_{0}( A + \i B_1 e_1)| 
	        	\leq 
        		C_{2}\| y_0\|_{X_\alpha}.
        \end{equation*}
Then the estimate \eqref{eq:goal 2} follows from the two above inequalities. This concludes the proof of Theorem \ref{Thm-Estimates-X-alpha-Gal-d}.
\end{proof}

\subsection{Proof of the strong continuity of the heat semigroup on $X_\alpha$}

This section aims at proving the strong continuity of the heat semigroup on $X_\alpha$. 
\begin{proposition}
	\label{Prop-C0-Semigroup}
	Let $d \geq 1$ and $\alpha >1$. Then the heat semigroup is a strongly continuous semigroup on $X_\alpha$.
\end{proposition}
\begin{proof}
	Let $y_0 \in X_\alpha$. For $\lambda >1 $, introduce the functions 
	$$
		y_{0, \lambda} (x) = y_0\left( \frac{x}{\lambda} \right), \qquad x \in \R^d.
	$$
	Since $y_0 \in \BUC (\R^d)$, for all $K> 0$, $\lim_{\lambda\to 1^+}\| y_{0,\lambda} - y_0\|_{L^\infty(B_{\R^d}(K))}  = 0$. 

	Also, since $y_0\large|_{B_{\R^d}(1)}$ admits an holomorphic extension $y_{0,e}$ in $\Omega_\alpha$
which can be extended as a continuous function on $\overline\Omega_\alpha$, this is also case for $y_{0,\lambda}$ and its holomorphic extension $y_{0,\lambda,e}$ in $\Omega_\alpha$ is simply given by 
	$$
		y_{0, \lambda,e}( a+ \i b ) =  y_{0,e}\left( \frac{a+\i b}{\lambda} \right), \qquad (a + \i b) \in \overline\Omega_\alpha.
	$$
	Since $y_{0,e}$ is continuous in $\overline\Omega_\alpha$, it is uniformly continuous in $\overline\Omega_\alpha$ and $\lim_{\lambda\to 1^+}\| y_{0,\lambda,e} - y_{0,e}\|_{L^\infty(\overline\Omega_\alpha)}  = 0$.

	We thus introduce a smooth cut-off function $\eta$ taking value in $[0,1]$, compactly supported in $B_{\R^d}(3)$, and taking value $1$ in $B_{\R^d}(2)$, and set 
	$$
		z_{0, \lambda} (x) = \eta(x) y_0\left( \frac{x}{\lambda} \right) + (1 - \eta(x)) y_0(x), \quad x \in \Omega.
	$$
	Then 
	$$
		\lim_{\lambda \to 1^+} \| y_0 -z_{0,\lambda} \|_{X_\alpha} = 0, 
	$$
	and, for $\lambda \in (1, 2]$, $z_{0,\lambda}$ belongs to the set 
	\begin{multline}
		\label{Def-X-alpha-lambda}
		X_{\alpha,\lambda} := 
		\enstq{ f \in \BUC (\R^d)}{ f{\large|}_{B_{\R^d}(\lambda)} \text{ admits a continuous extension } f_e \text{ on } \lambda \overline\Omega_\alpha
		 \text{, which belongs to } \mathrm{Hol}(\lambda\Omega_\alpha)} 
		 , 
	\end{multline}
	endowed with the norm
	\begin{equation}
		\label{Def-Norm-X-alpha-lambda}
		\| f \|_{X_{\alpha,\lambda}} 
		:= 
		\| f \|_{L^\infty(\R^d)} 
		+ 
		\| f_e \|_{L^{\infty}(\lambda \Omega_\alpha)}.
	\end{equation}

	Since the heat semigroup is invariant by scaling $(t,x) \mapsto (\lambda^2 t, \lambda x)$, we thus have, with the same constant as in Theorem \ref{Thm-Estimates-X-alpha-Gal-d}, that for all $\lambda \in (1, 2]$, for all $t \geq 0$, 
	\begin{equation}
		\label{Bound-T-t-z-lambda-Xalpha-lambda}
		\| \mathbb{T}_t z_{0,\lambda} \|_{X_{\alpha, \lambda}}
		\leq 
		C \| z_{0,\lambda}\|_{X_{\alpha,\lambda}} 
		\leq 
		C \| y_0\|_{X_\alpha}.
	\end{equation}
	
	Let us then take $\varepsilon >0$. Then there exists $\lambda \in (1,2]$ such that $\| z_{0,\lambda} - y_0\|_{X_\alpha} \leq \varepsilon$. Thus, for all $t>0$, using Theorem \ref{Thm-Estimates-X-alpha-Gal-d}, we get
	\begin{equation}
		\| \T_t y_0 - y_0 \|_{X_{\alpha}} 
		\leq 
		\| \T_t (y_0 - z_{0,\lambda} ) \|_{X_\alpha}
		+ 
		\| \T_t z_{0,\lambda} - z_{0,\lambda} \|_{X_{\alpha}} 
		+ 
		\| z_{0,\lambda} - y_0 \|_{X_\alpha}
		\leq 
		C \varepsilon + 	\| \T_t z_{0,\lambda} - z_{0,\lambda} \|_{X_{\alpha}}.
	\end{equation}
	Since the heat semigroup is strongly continuous on $\BUC(\R^d)$, $\lim_{t\to 0} \| \T_t z_{0,\lambda} - z_{0,\lambda} \|_{L^\infty(\R^d)} = 0$. Now, from \eqref{Bound-T-t-z-lambda-Xalpha-lambda}, $\T_t z_{0,\lambda}$ admits an holomorphic expansion in $\lambda \Omega_{\alpha}$, which is bounded and continuous in $\lambda \overline\Omega_{\alpha}$. Since the set $\mathcal{C}(\lambda \overline\Omega_\alpha) \cap \text{Hol}(\lambda \Omega_\alpha)$ is compact in $\mathcal{C}(\overline\Omega_\alpha)\cap \text{Hol}(\lambda \Omega_\alpha)$, and since $\T_t z_{0,\lambda}$ converges to $z_{0,\lambda}$ in $L^\infty(B_{\R^d}(1))$, any accumulation point of $(\T_t z_{0,\lambda})$ as $t \to 0$ is an holomorphic function on $\Omega_\alpha$ which coincides with $z_{0,\lambda}$ in $\overline B_{\R^d}(1)$. Accordingly, we necessarily have $\lim_{t\to 0} \| \T_t z_{0,\lambda} - z_{0,\lambda} \|_{X_\alpha} = 0$.
	
	With the previous estimate, this yields that 
	$$
		\limsup_{t \to 0} \| \T_t y_0 - y_0 \|_{X_{\alpha}} 
		\leq 
		C \varepsilon.
	$$
	Since $\varepsilon>0$ is arbitrary, we have obtained 
	$$
		\lim_{t \to 0} \| \T_t y_0 - y_0 \|_{X_{\alpha}} = 0.
	$$
	This concludes the proof of Proposition \ref{Prop-C0-Semigroup}.
\end{proof}	

\subsection{Proof of Theorem \ref{Thm-Well-Posedness-X-alpha}}

In view of Theorem \ref{Thm-Estimates-X-alpha-Gal-d} and Proposition \ref{Prop-C0-Semigroup}, it only remains to prove that the heat semigroup is analytic on $X_\alpha$ and the estimate on $\T_t \partial_j$. 

The analyticity of the heat semigroup on $X_\alpha$ follows from the fact that the heat semigroup is a strongly continuous semigroup on $X_\alpha$, whose generator is, according to \cite[Section 2.3 p.60]{Engel-Nagel} thus given by 
$$
	A y = \Delta y, 
	\text{ with domain } \mathscr{D}(A) = \{ y \in \BUC^2(\R^d) \cap X_\alpha, \, \text{ such that } \Delta y \in X_\alpha\}.
$$
Here, the set $\BUC^2(\R^d)$ is the set of all functions $y \in \mathcal{C}^2(\R^d)$ such that for all $(j, k) \in \{1, \cdots, d\}^2$, $y$, $\partial_j y$ and $\partial_{j,k} y $ all belong to $\BUC(\R^d)$, which coincides with the domain of the heat semigroup on $\BUC(\R^d)$.

Now, using the estimates of Theorem \ref{Thm-Estimates-X-alpha-Gal-d} and the fact that, for each $j \in \{1, \cdots, d\}$, $\partial_j$ and the heat semigroup commutes on the set $\BUC^2(\R^d)$, we get that for $y_0 \in X_\alpha\cap \BUC^2(\R^d)$, and $t > 0$, 
\begin{equation*}
	\| A \T_t y_0 \|_{X_\alpha} 
	\leq 
	\sum_{j = 1}^d \| \partial_{jj} \T_t y_0\|_{X_\alpha}
	\leq
	\sum_{j = 1}^d \| \partial_{j} \T_{t/2} \partial_j \T_{t/2} y_0\|_{X_\alpha}\leq 
	\frac{C}{\sqrt{t}} \sum_{j = 1}^d \|\partial_j \T_{t/2} y_0\|_{X_\alpha}
	\leq 
	\frac{C}{t} \| y_0 \|_{X_\alpha}.
\end{equation*}
This can of course be extended to any $y_0 \in X_\alpha$ by density. We have thus obtained one of the characterization of analyticity of the semigroup $(\T_t)_{t \geq 0}$ on $X_\alpha$, see for instance \cite[Theorem 4.6 p.101]{Engel-Nagel}. 
This ends the proof of Theorem \ref{Thm-Well-Posedness-X-alpha}.
\hfill $\qed$

\subsection{Additional estimates}
We end this section with a corollary of Theorem \ref{Thm-Well-Posedness-X-alpha} which will be used in the following. We introduce the functional space $X_\alpha^1$ given for $\alpha >0$ by 
\begin{equation}
	\label{Def-X-alpha-1}
	X_\alpha^1 = \{ y \in X_\alpha, \text{ such that for all $j \in \{1, \cdots, d\}$, } \partial_j y \in X_\alpha \}
\end{equation}
endowed with the topology given by the norm
$$
	\| y \|_{X_\alpha^1 } 
	= 
	\| y \|_{X_\alpha } 
	+ 
	\sum_{j = 1}^d \|\partial_j y \|_{X_\alpha }.
$$
We then get the following result:
\begin{theorem}\label{Thm-Well-Posedness-X-alpha-1}
	Let $d \in \N$ and $\alpha >1$. Then the heat semigroup $\T$ is an analytic semigroup on $X_\alpha^1$, and we have the following estimates: there exists $C>0$ such that for all $t > 0$ and for all $y_0 \in X_\alpha^1$, 
	\begin{equation}
		\label{Est-Well-Posedness-X-alpha-1}
		\| \T_t y_0 \|_{X_\alpha^1} 
		+ 
		\sqrt{t } \| \nabla (\T_t y_0) \|_{X_\alpha^1}
		\leq C \| y_0\|_{X_\alpha^1}.
	\end{equation}
\end{theorem}
\begin{proof}
	Here again, we use that on $\BUC^2(\R^d)$, for each $j \in \{1, \cdots, d\}$, $\partial_j$ commutes with the heat semigroup. The estimates \eqref{Est-Well-Posedness-X-alpha-1} then follows from Theorem \ref{Thm-Estimates-X-alpha-Gal-d}. The strong continuity of the heat semigroup can then be done as in Proposition \ref{Prop-C0-Semigroup}. Details are left to the reader.
\end{proof}

\section{Well-posedness results of the heat equations with various lower order terms in spaces of holomorphic functions}\label{Sec-WP-Heat+lower-order}

{\bf Notations.} From now on, to alleviate notation, for $X$ a Banach space, $p \in [1, \infty]$ and $T>0$, the norms $\| \cdot \|_{L^p_T(X)}$ will denote the norms of $L^p(0,T; X)$. Similarly, when the Banach space $X$ is a Sobolev space of the form $L^q(\R^d)$ or $W^{1,q}(\R^d)$, we will simply write $ \| \cdot \|_{L^p_T(L^q)}$, or $\| \cdot \|_{L^p_T(W^{1,q})}$, instead of $ \| \cdot \|_{L^p(0,T;L^q(\R^d))}$ or $\| \cdot \|_{L^p(0,T;W^{1,q}(\R^d))}$.
\medskip

In this section, we gather several well-posedness results for the heat equation in $\R^d$ with various lower terms, based on the analyticity of the heat semigroup on $X_\alpha$ proved in Theorem \ref{Thm-Well-Posedness-X-alpha} and on $X_\alpha^1$ in Theorem \ref{Thm-Well-Posedness-X-alpha-1}.

Let us start with the following result:
\begin{theorem}
	\label{Thm-WP-with-source-terms}
	Let $d \geq 1$, $T>0$, and $\alpha >1$. 
	
	There exists $C>0$ independent of $T>0$ such that, for $y_0 \in X_\alpha$ and $f \in L^1(0,T; X_\alpha)$, 
	the solution $y$ of 
	\begin{equation}
		\label{Heat-Eq-Source-Terms}
		\left\{
			\begin{array}{ll}
				\partial_t y - \Delta y = f,
				 \quad & \text{ in } (0, T) \times \R^d, 
				\\
				y(0, \cdot) = y_0, \quad & \text{ in } \R^d.
			\end{array}
		\right.
	\end{equation}
	belongs to $\mathcal{C}([0,T]; X_\alpha)$, and satisfies
	\begin{equation}
	\label{Est-WP-source-terms}
		\| y \|_{L^{\infty}_{T}(X_\alpha)}
		\leq
		C 
		\left(
			\|y_0\|_{X_\alpha}
			+
			\| f \|_{L^{1}_{T}(X_\alpha)}
		\right).
	\end{equation}
	If furthermore $t \mapsto \sqrt{t} f(t,\cdot) \in L^\infty(0,T; X_\alpha)$, we also have 
	\begin{equation}
	\label{Est-WP-source-terms-mieux}
		 \|\sqrt{t}\nabla_x y(t,\cdot) \|_{L^{\infty}_T(X_\alpha)}
		\leq
		C 
		\left(
			\|y_0\|_{X_\alpha}
			+
			\sqrt{T} \| \sqrt{t} f(t,\cdot) \|_{L^{\infty}_{T}(X_\alpha)}
		\right).
	\end{equation}

	\emph{Additional regularity.} There exists $C>0$ independent of $T>0$ such that, for $y_0 \in X_\alpha^1$ and $f \in L^\infty(0,T; X_\alpha)$, the solution $y$ of \eqref{Heat-Eq-Source-Terms} belongs to $\mathcal{C}([0,T]; X_\alpha^1)$, and satisfies 
	\begin{equation}
	\label{Est-WP-source-terms-1-0}
		\| y \|_{L^{\infty}_{T}(X_\alpha^1)}
		\leq
		C 
		\left(
			\|y_0\|_{X_\alpha^1}
			+
			\sqrt{T} \| f \|_{L^{\infty}_{T}(X_\alpha)}
		\right). 
	\end{equation}
\end{theorem}

\begin{proof}
	Let us start with the proof of \eqref{Est-WP-source-terms}. The solution $y$ of \eqref{Heat-Eq-Source-Terms} writes
	\begin{equation}
		\label{Duhamel}
		y(t ) = \T_t y_0 + \int_0^t \T_{t-s} f(s) \, {\rm d}s 
		, \qquad t \in [0,T].
	\end{equation}
	Accordingly, using Theorem \ref{Thm-Well-Posedness-X-alpha}, for all $t \in  [0,T]$,
	\begin{equation}
		\| y(t)\|_{X_\alpha}
		\leq 
		C\| y_0\|_{X_\alpha}
		+ 
		C \int_0^t \| f(s) \|_{X_\alpha} \, {\rm d}s.
		\label{Est-y-t-X-alpha-1st}
	\end{equation}
	The continuity of $y$ in $X_\alpha$ follows from a similar argument and the strong continuity of the heat semigroup on $X_\alpha$. Indeed, for $t_1$ and $t_2$ in $[0,T]$, setting $t_m = \min\{t_1, t_2\}$, $t_M = \max \{t_1,t_2\}$ and $\epsilon_{t_1,t_2} = 1$ if $t_1 > t_2$, and $\epsilon_{t_1,t_2} = -1$ if $t_1 < t_2$, we get 
	\begin{multline*}
		y(t_1) - y(t_2)
		 = 
		\epsilon_{t_1,t_2} \T_{t_m} ( \T_{t_M - t_m} -  \text{Id} ) y_0
		\\
		+
		\epsilon_{t_1,t_2} \int_0^{t_m} \T_{t_m-s}   ( \T_{t_M - t_m} -  \text{Id})  f(s) \, {\rm d}s
		+ 
		\epsilon_{t_1, t_2} \int_{t_m}^{t_M} \T_{t_M-s} f(s) \, {\rm d}s 
	\end{multline*}
	so that, using $t_M - t_m = |t_1 - t_2|$,
	\begin{multline*}
		\| y(t_1 ) - y(t_2)\|_{X_\alpha}
		\leq 
		C \| ( \T_{|t_1-t_2|} - \text{Id}) y_0\|_{X_\alpha}
		+ 
		C \int_0^{T} \|  ( \T_{|t_1-t_2|} -  \text{Id})  f(s)\|_{X_\alpha} {\rm d}s+ 
		C \int_{t_m}^{t_M} \| f(s)\|_{X_\alpha}\, {\rm d}s.
	\end{multline*}
	Lebesgue's dominated convergence theorem then implies that $y \in \mathcal{C}([0,T]; X_\alpha)$ for $f$ in  $L^1([0,T]; X_\alpha)$.

	To get estimate \eqref{Est-WP-source-terms-mieux}, we simply apply Theorem \ref{Thm-Well-Posedness-X-alpha} to the identity \eqref{Duhamel}: 
	\begin{align*}
		\| \nabla y(t, \cdot)\|_{X_\alpha}
		\leq
		\frac{C}{\sqrt{t}} \| y_0\|_{X_\alpha} 
		+ 
		C \int_0^t \frac{1}{\sqrt{t-s}} \| f(s) \|_{X_\alpha} \, {\rm d}s 
		\leq
		\frac{C}{\sqrt{t}} \| y_0\|_{X_\alpha} 
		+ 
		C \| \sqrt{s} f(s) \|_{L^{\infty}_{t}( X_\alpha)}.
	\end{align*}
	
	The proof of estimate \eqref{Est-WP-source-terms-1-0} can be done similarly, since for all $t \in (0,T)$,
	$$
		\| \nabla y(t, \cdot)\|_{X_\alpha}
		\leq
		C \| y_0\|_{X_\alpha^1} 
		+ 
		C \int_0^t \frac{1}{\sqrt{t-s}} \| f(s) \|_{X_\alpha} \, {\rm d}s 
		\leq
		C \| y_0\|_{X_\alpha^1} 
		+ 
		C \sqrt{T} \| f \|_{L^{\infty}_{T}( X_\alpha)}.
	$$
\end{proof}
\begin{remark}
	\label{Rk-Other-Time-Weights-on-f}
	If $\gamma \in [0,1)$, one can prove that if $t^{1-\gamma} f(t,\cdot) \in L^\infty(0,T; X_\alpha)$, then 
	\begin{equation}
	\label{Est-WP-source-terms-mieux-gamma}
		 \|\sqrt{t}\nabla_x y(t,\cdot) \|_{L^{\infty}_{T}(X_\alpha)}
		\leq
		C 
		\left(
			\|y_0\|_{X_\alpha}
			+
			T^{1-\gamma} \| t^\gamma f(t,\cdot) \|_{L^{\infty}_{T}(X_\alpha)}
		\right).
	\end{equation}
	instead of \eqref{Est-WP-source-terms-mieux}. In Theorem \ref{Thm-WP-with-source-terms}, we only considered the case $\gamma = 1/2$ because that choice allows to handle lower order terms involving gradient terms, as we will see in Theorem \ref{Thm-WP-linear-LOT} afterwards.
\end{remark}
\begin{remark}
	\label{Rk-A-compact-form}
	Note that 
	$$
		\| f \|_{L^{1}_{T}(X_\alpha)} 
		=
		\int_0^T \| f(s) \|_{X_\alpha} \, {\rm d}s
		\leq 
		C \sqrt{T} \| \sqrt{t} f \|_{L^{\infty}_{T}(X_\alpha)}, 
	$$
	so we get from  \eqref{Est-WP-source-terms}--\eqref{Est-WP-source-terms-mieux} that
	\begin{equation}
	\label{Est-WP-source-terms-mieux-compact}
		\| y \|_{L^{\infty}_{T}(X_\alpha)}
		+
		 \|\sqrt{t}\nabla_x y(t,\cdot) \|_{L^{\infty}_{T}(X_\alpha)}
		\leq
		C 
		\left(
			\|y_0\|_{X_\alpha(\R^d)}
			+
			\sqrt{T} \| \sqrt{t} f(t,\cdot) \|_{L^{\infty}_{T}(X_\alpha)}
		\right). 
	\end{equation}
\end{remark}

We are now in position to consider various lower terms.
\begin{theorem}
	\label{Thm-WP-linear-LOT}
	Let $d \geq 1$, $T>0$, and $\alpha >1$. 
	
	Let 
	\begin{equation}
		q \in L^{\infty} (0,T; X_\alpha), 
		\quad
		W \in L^{\infty}(0,T;X_\alpha). 
	\end{equation}
	Then there exists a constant $C>0$ such that for any $y_0 \in X_\alpha $ and $\sqrt{t} f \in L^\infty(0,T; X_\alpha)$, the solution $y$ of 
	\begin{equation}
		\label{Heat-Eq-LOT}
		\left\{
			\begin{array}{ll}
				\partial_t y - \Delta y + q y + W \cdot \nabla y
				 = f, \quad & \text{ in } (0,T) \times \R^d, 
				\\
				y(0, \cdot) = y_0, \quad & \text{ in } \R^d, 
			\end{array}
		\right.
	\end{equation}
	belongs to $\mathcal{C}([0,T]; X_\alpha)$, and satisfies, for some $C$ depending only on the time horizon $T$ and the norms $\| q\|_{L^{\infty}_{T}( X_\alpha)}$ and $\| W\|_{L^{\infty}_{T}(X_\alpha)}$, 
	\begin{equation}
		\label{Est-WP-y-with-potentials}
		\| y \|_{L^{\infty}_{T}(X_\alpha)} + \| \sqrt{t} \nabla_x y\|_{L^{\infty}_{T}(X_\alpha)}
		\leq 
		C \left( 
		\|y_0\|_{X_\alpha} 
		+
		\| \sqrt{t } f \|_{L^{\infty}_{T}(X_\alpha)}
		\right).
	\end{equation}
\end{theorem}	
	
\begin{proof}
	We construct the solution $y$ using a fixed point argument in a time interval $(0,T_0)$, where $T_0 \in (0,T]$ will be chosen suitably small at the end. 
	
	We set 
	$$
		\mathscr{C}(T_0) = 
		\enstq{ 
			y \in \mathcal{C}([0,T_0]; X_\alpha)}{\sqrt{t} \nabla_x y \in L^\infty(0,T_0; X_\alpha)},
	$$
	and we define the map
	$$
		\Lambda_{T_0} : \widehat y \in \mathscr{C}(T_0) 
		\mapsto 
		y \text{ solution of }
		\left\{
			\begin{array}{ll}
				\partial_t y - \Delta y + q \widehat y + W \cdot \nabla \widehat y
				 = f, \quad & \text{ in } (0, T_0) \times \R^d, 
				\\
				y(0, \cdot) = y_0, \quad & \text{ in } \R^d.
			\end{array}
		\right.
	$$
	First, from Theorem \ref{Thm-WP-with-source-terms}, it is clear that $\Lambda_{T_0}$ maps $\mathscr{C}(T_0)$ to itself when $q \in L^\infty(0,T; X_\alpha)$ and $W \in L^\infty(0,T; X_\alpha)$.
	
	Second, for $\widehat y_1$ and $\widehat y_2$ in $\mathscr{C}(T_0)$, $y_1 = \Lambda_{T_0} (\widehat y_1)$ and $y_2 = \Lambda_{T_0} (\widehat y_2)$ satisfy (using \eqref{Est-WP-source-terms-mieux-compact}) 
	\begin{align*}
		& \| (y_1 - y_2) \|_{L^\infty_{T_0}(X_\alpha)} +  \| \sqrt{t} \nabla (y_1 - y_2)\|_{L^\infty_{T_0}(X_\alpha)}
		\\
		& \leq 
		C \sqrt{T_0} \| \sqrt{t}  q  (\widehat y_1 - \widehat y_2) \|_{L^\infty_{T_0}(X_\alpha)}
		+ 
		C \sqrt{T_0} \| \sqrt{t} W \cdot \nabla (\widehat y_1 - \widehat y_2)\|_{L^\infty_{T_0}(X_\alpha)}
		\\
		& \leq 
		C \sqrt{T_0}
		\left(
			\sqrt{T_0} \| q\|_{L^{\infty}_{T_0}(X_\alpha)} + \| W \|_{L^{\infty}_{T_0}(X_\alpha)}
		\right)\left(
			\|\widehat y_1 - \widehat y_2 \|_{L^{\infty}_{T_0}(X_\alpha)}
			+
			 \| \sqrt{t} \nabla_x (\widehat y_1 - \widehat y_2)\|_{L^{\infty}_{T_0}(X_\alpha)}
		\right).
	\end{align*}
	Accordingly, choosing $T_0$ small enough such that 
	\begin{equation}
		\label{Cond-T-0}
		C \sqrt{T_0} 
		\left(
			\sqrt{T_0} \| q\|_{L^{\infty}_{T_0}(X_\alpha)} + \| W \|_{L^{\infty}_{T_0}(X_\alpha)}
		\right)
		\leq \frac{1}{2},
	\end{equation}
	the map $\Lambda_{T_0}$ is contractive on $\mathscr{C}(T_0)$ endowed with the norm $\|\cdot \|_{L^{\infty}_{T_0}(X_\alpha)} +  \| \sqrt{t} \nabla_x (\cdot )\|_{L^{\infty}_{T_0}(X_\alpha)}$. It thus has a unique minimizer $y \in \mathscr{C}(T_0)$ which by construction solves \eqref{Heat-Eq-LOT}. We also check that estimate \eqref{Est-WP-source-terms-mieux} gives 
	\begin{multline*}
		\|y \|_{L^\infty_{T_0}(X_\alpha)} + 	 \| \sqrt{t} \nabla y\|_{L^\infty_{T_0}(X_\alpha)}
		\leq
		 C \| y_0 \|_{X_\alpha} + C\sqrt{T_0} \|\sqrt{t} f \|_{L^\infty_{T_0}(X_\alpha)} 
		\\
		+ C \sqrt{T_0}	\left(
			\sqrt{T_0} \| q\|_{L^{\infty}_{T_0}(X_\alpha)} + \| W \|_{L^{\infty}_{T_0}(X_\alpha)}
		\right)
		\left(	\|y \|_{L^{\infty}_{T_0}(X_\alpha)}
			+
			 \| \sqrt{t} \nabla_x y\|_{L^{\infty}_{T_0}(X_\alpha)}
		\right).
	\end{multline*}
	Using \eqref{Cond-T-0}, we deduce 
	\begin{equation}
		\label{Estimate-on-0-T0}
		\|y \|_{L^\infty_{T_0}(X_\alpha)} +  \| \sqrt{t} \nabla_x y\|_{L^{\infty}_{T_0}(X_\alpha)} 
		\leq C \| y_0 \|_{X_\alpha} + C \| \sqrt{t} f \|_{L^\infty_{T_0}(X_\alpha)}. 
	\end{equation}
	We can then iterate this process to construct the solution $y$ on $[0,T]$ by cutting it in intervals of time of length $T_0$. The number of iteration depends on $T_0$, defined in \eqref{Cond-T-0}, and one gets an estimate on the solution $y$ of \eqref{Heat-Eq-LOT} on each of these intervals similar to the one in \eqref{Estimate-on-0-T0} (shifted in time). Accordingly, the estimate \eqref{Est-WP-y-with-potentials} follows.
	%
	%
\end{proof}	

\begin{remark}\label{Rk-WP-BUC}	
	Note that, since the proof of Theorem {\ref{Thm-WP-linear-LOT}} is based on the analyticity of the heat semigroup $\T = (\T_t)_{t \geq 0}$ on $X_\alpha$ and the fact that the {operator $q+W\cdot\nabla$} maps $X_\alpha^1$ to $X_\alpha$, it can be adapted to several other contexts. In particular, for later use, let us point out that the above proof of Theorem {\ref{Thm-WP-linear-LOT}} can be easily adapted to get the following result. 
	
	{If $q$ and $W$ both belong to $L^\infty(0,T; \BUC(\R^d))$,  $y_0 \in \BUC^1(\R^d)$, and $f \in L^\infty(0,T; \BUC(\R^d))$, then
    there exists a unique solution $y$ of \eqref{Heat-Eq-LOT} such that $y \in \mathcal{C}([0,T]; \BUC^1(\R^d))$ with 
	$$
\| y \|_{L^{\infty}_{T}(W^{1,\infty})} 
		\leq 
		C
		\left(\|y_0\|_{W^{1,\infty}} + \|f  \|_{L^{\infty}_{T}(L^\infty)}\right).
	$$
   } 

\end{remark}

We now consider a fully semilinear heat equation with a nonlinear term $(t,x) \mapsto g(t,x, y(t,x), \nabla_x y(t,x))$ for some function $g$ depending on $(t,x, s, s_d) \in [0,T] \times \R^d \times \overline B_\C (\varepsilon) \times \overline B_{\C^d} (\varepsilon)$ for some $\varepsilon>0$.

\begin{theorem}
	\label{Thm-Semilinear-HE}
	Let $d \geq 1$, $T>0$ and $\alpha >1$. 

	Let $\varepsilon>0$ and 
	$$
		g: (t,x, s, s_d) \in [0,T] \times (\R^d \cup \overline\Omega_\alpha) \times \overline B_\C (\varepsilon) \times \overline B_{\C^d} (\varepsilon)\longmapsto g(t,x,s,s_d)\in\C.
	$$
	satisfying the following conditions:
	\begin{equation}
		\label{g-in-BUC}
		g \in L^\infty(0,T; \BUC((\R^d \cup\overline\Omega_\alpha)\times \overline B_\C (\varepsilon) \times  \overline B_{\C^d} (\varepsilon))))
	\end{equation}
	\begin{equation}
		\label{Holomorphic-Condition-g}
		\text{a.e. in } t, 
		\, 
		g(t, \cdot, \cdot,\cdot) \text{ is holomorphic in }\Omega_\alpha \times \overline B_\C(\varepsilon)\times \overline B_{\C^d}(\varepsilon), 
	\end{equation}
	\begin{equation}
		\label{g-Lipschitz}
		g \in L^\infty([0,T] \times (\R^d \cup \overline\Omega_\alpha); W^{1, \infty} (\overline B_\C (\varepsilon) \times  \overline B_{\C^d} (\varepsilon))).
	\end{equation}
	\begin{equation}
		\label{g-for-s=0}
		\forall (t,x) \in [0,T] \times (\R^d \cup \overline\Omega_\alpha), \quad 
		g(t,x, 0_\C, 0_{\C^d}) = 0,
	\end{equation}
	
	%

	Then for all $T >0$, there exists $\delta >0$ and $C>0$ such that for any initial datum $y_0 \in X_\alpha^1$ and $f \in L^{\infty}(0,T;X_\alpha)$ satisfying $\| y_0 \|_{X_\alpha^1}+ \| f  \|_{L^{\infty}_{T}(X_\alpha)}\leq \delta$, there exists a solution $y$ of 
	\begin{equation}
		\label{Heat-Eq-Non-Linear-semilinear}
		\left\{
			\begin{array}{ll}
				\partial_t y - \Delta y
				= 
				g(\cdot,\cdot,y, \nabla y) + f, \quad & \text{ in } (0, T) \times \R^d, 
				\\
				y(0, \cdot) = y_0, \quad & \text{ in } \R^d,
			\end{array}
		\right.
	\end{equation}
	in $\mathcal{C}([0,T]; X_\alpha^1)$ such that 
	\begin{equation}
		\label{Est-y-semi-Linear}
		\| y \|_{L^{\infty}_{T}(X_\alpha^1)} 
		\leq 
		C
		\left(\|y_0\|_{X_\alpha^1} + \|f  \|_{L^{\infty}_{T}(X_\alpha)}\right).
	\end{equation}
\end{theorem}

\begin{proof}
	For $R \in (0, \varepsilon)$ and $T_0\in (0,T]$, we set 
	$$
		\mathscr{C}_{R,T_0}^1 
		:= 
		\enstq{y \in L^{\infty}(0,T_0;X_{\alpha}^{1})}{\| y \|_{L^{\infty}_{T_0}(X_\alpha^1)}
 \leq R}, 
	$$
	and the map 
	$$
		\Lambda_{R,T_0} : 
		\widehat y \in  \mathscr{C}_{R,T_0}^1
		\mapsto
		y \text{ solution of }
		\left\{
			\begin{array}{ll}
				\partial_t y - \Delta y = g(\cdot,\cdot, \widehat y , \nabla \widehat y) +f , \quad & \text{ in } (0, T_0) \times \R^d, 
				\\
				y(0, \cdot) = y_0, \quad & \text{ in } \R^d.
			\end{array}
		\right.
	$$
	According to \eqref{g-in-BUC}--\eqref{g-for-s=0}, for $	\widehat y \in  \mathscr{C}_{R,T_0}^1$, we have that  $g(\cdot,\cdot, \widehat y , \nabla \widehat y) \in L^\infty(0,T_0; X_\alpha)$ with (due to \eqref{g-Lipschitz}--\eqref{g-for-s=0})
	$$
		\| g(\cdot,\cdot, \widehat y , \nabla \widehat y) \|_{L^\infty_{T_0}(X_\alpha)}
		\leq C_0 R.
	$$
	where we set 
	\begin{equation}
		\label{Lipschitz-g}
		C_0 := \|\nabla_{(s,s^d)} g\|_{L^{\infty}((0,T)\times(\R^d\cup\Omega_\alpha)\times B_{\C}(\varepsilon)\times B_{\C^d}(\varepsilon))}<+\infty.
	\end{equation}
	
	Using then estimate \eqref{Est-WP-source-terms-1-0}, we get a constant $C>0$ independent of $R$ and $T_0$ such that, for $\widehat y \in  \mathscr{C}_{R,T_0}^1$, 
	$$
		\| \Lambda_{R,T_0} (\widehat y)\|_{L^\infty_{T_0}(X_\alpha^1)} \leq C \left( \| y_0 \|_{X_\alpha^1} + \sqrt{T_0} \| f \|_{L^\infty_{T_0}(X_\alpha)} + \sqrt{T_0} C_0 R \right), 
	$$
	and similarly, if $\widehat y_1, \widehat y_2 \in  \mathscr{C}_{R,T_0}^1$, 	%
	$$
		\| \Lambda_{R,T_0} (\widehat y_1) - \Lambda_{R,T_0} (\widehat y_1)\|_{L^\infty_{T_0}(X_\alpha^1)} 
		\leq C \sqrt{T_0} \|  g(\cdot,\cdot, \widehat y_2 , \nabla \widehat y_2) -  g(\cdot,\cdot, \widehat y_1 , \nabla \widehat y_1)\|_{L^\infty_{T_0}(X_\alpha)}
		\leq
		C \sqrt{T_0} C_0 \| \widehat y_1 - \widehat y_2 \|_{L^\infty_{T_0}(X_\alpha^1)}.
	$$
	Accordingly, for $\delta \in (0, \varepsilon/(2C)]$, we choose $T_0$ and $R$ as follows:
	\begin{equation}
		\label{Choices-T-R}
		T_0 = \min\left\{1, \frac{1}{4 C^2 C_0^2} \right\}, 
		\quad \text{ and } \quad 
		R = 2 C \delta, 
	\end{equation}
	so that if $y_0 \in X_\alpha^1$ and $f \in L^{\infty}(0,{T_0};X_\alpha)$ satisfies $\| y_0 \|_{X_\alpha^1}+ \| f  \|_{L^{\infty}_{T_0}(X_\alpha)}\leq \delta$, the map $\Lambda_{R,T_0}$ maps $ \mathscr{C}_{R,T_0}^1$ into itself and is contractive, and the existence of a fixed point $y \in  \mathscr{C}_{R,T_0}^1$ is granted by Banach-Picard theorem. 
	
	We have thus obtained the existence of a solution $y$ of \eqref{Heat-Eq-Non-Linear-semilinear} for data $y_0 \in X_\alpha^1$ and $f \in L^{\infty}(0,{T_0};X_\alpha)$ satisfying $\| y_0 \|_{X_\alpha^1}+ \| f  \|_{L^{\infty}_{T_0}(X_\alpha)}\leq \varepsilon/(2C)$ locally in time, that is on $(0,T_0)$ with $T_0 = \min\{1, 1/(4 C^2 C_0^2) \}$ (given by \eqref{Choices-T-R}), and 
	$$
		\| y \|_{L^\infty_{T_0}(X_\alpha^1)} \leq 2C \left( \| y_0 \|_{X_\alpha^1}+ \| f  \|_{L^{\infty}_{T_0}(X_\alpha)}\right), 
	$$
	and in particular 
	$$
		\| y(T_0,\cdot) \|_{X_\alpha^1} \leq 2C \left( \| y_0 \|_{X_\alpha^1}+ \| f  \|_{L^{\infty}_{T_0}(X_\alpha)}\right)
	$$	
	
	To get a result for an arbitrary large time, we should iterate this construction. For $T >0$, set $n = \lfloor T/T_0 \rfloor$. It is easy to check that if 
	$$
		\| y_0 \|_{X_\alpha^1}+ \| f  \|_{L^{\infty}_{T}(X_\alpha)} \leq \frac{\varepsilon}{(2C)^{n+1}}, 
	$$
	we can iterate the above construction on each interval of the form $[jT_0, (j+1)T_0]$, with $j  \in \{0, \cdots, n\}$, and get 
	$$
		\| y \|_{L^\infty((jT_0, (j+1)T_0); X_\alpha^1)}
		+
		\| y((j+1) T_0,\cdot) \|_{X_\alpha^1} \leq (2C)^{j+1} \left( \| y_0 \|_{X_\alpha^1}+ \| f  \|_{L^{\infty}_{T}(X_\alpha)}\right)
		\leq 
		(2C)^{j-n}\varepsilon \leq \varepsilon.	
	$$
	This concludes the proof of Theorem \ref{Thm-Semilinear-HE}.
\end{proof}

\begin{remark}[Uniqueness of solutions of \eqref{Thm-Semilinear-HE}]
	\label{Rem-Uniqueness}
	For later use, let us point out that, under the assumptions of \eqref{Thm-Semilinear-HE}, for any $y_0 \in W^{1, \infty}(\R^d)$ and $f \in L^\infty(0,T;L^\infty(\R^d))$, there exists only one solution $y \in L^\infty(0,T;W^{1, \infty}(\R^d))$ with $\| y \|_{L^\infty_T(W^{1, \infty})} \leq \varepsilon$ to the equation \eqref{Heat-Eq-Non-Linear-semilinear}. 
	
	Indeed, if there are two solutions $y_1$ and $y_2$ of  \eqref{Heat-Eq-Non-Linear-semilinear} with $\| y_1 \|_{L^\infty_T(W^{1, \infty})} \leq \varepsilon$ and $\| y_2 \|_{L^\infty_T(W^{1, \infty})} \leq \varepsilon$, $z = y_2 - y_1$ satisfies 
	\begin{equation}
		\label{Heat-Eq-Non-Linear-semilinear-z}
		\left\{
			\begin{array}{ll}
				\partial_t z - \Delta z
				= 
				h
				, \quad & \text{ in } (0, T) \times \R^d, 
				\\
				z(0, \cdot) = 0, \quad & \text{ in } \R^d, 
			\end{array}
		\right.
	\end{equation}
	where we have set 
	$$
		h 
		= 
		g(\cdot,\cdot,y_2, \nabla y_2) 
			- 
		g(\cdot,\cdot,y_1, \nabla y_1). 
	$$
	Accordingly, using the explicit knowledge of the heat semigroup as the convolution with the Gaussian kernel in $\R^d$ and the Lipschitz property of $g$ in \eqref{g-Lipschitz}, we get that for all $t \in (0,T)$,
	\begin{align*}
		\| z(t)\|_{W^{1, \infty}(\R^d)} 
		& \leq 
		\int_0^t\left( \| G(t-s) \|_{L^1(\R^d)} + \| \nabla G(t-s)\|_{L^1(\R^d)} \right) \|   h(s) \|_{L^\infty(\R^d)} \, {\rm d}s
		\\
		&\leq C \int_0^t \left(1 + \frac{1}{\sqrt{t-s}} \right) \| z(s) \|_{W^{1, \infty}(\R^d)} \, {\rm d}s.
	\end{align*}
	Since $z(0) = 0$ and $z$ belongs to $\mathcal{C}([0,T]; W^{1,\infty}(\R^d))$, we can conclude easily that $z$ vanishes everywhere. Indeed, if $z(t) = 0$ for all $t \in [0,T_*]$ ($T_*$ can be chosen to be $0$ in this argument), the above estimate yields that for $t > T_*$, 
	$$
		\sup_{s \in [T_*, t]} \|z(s) \|_{W^{1, \infty}(\R^d)} \leq 
		C \sqrt{t - T_*} \sup_{s \in [T_*, t]} \|z(s) \|_{W^{1, \infty}(\R^d)}, 
	$$
	so that, setting $T' = \frac{1}{4C^2}$, we get that $z$ vanishes in $[0, T_* + T']$. Iterating this argument easily shows that $z$ vanishes everywhere. 
\end{remark}

\section{Null-controllability of heat-type equations in $X_\alpha$ and $X_\alpha^1$}\label{Sec-NC-in-X-alpha}

The goal of this section is to present several results of the null-controllability of heat-type equations in the space $X_\alpha$ and $X_\alpha^1$.

\subsection{Preliminaries: Local regularity properties for the heat equation} 

	We start by recalling the following classical estimates for the heat equation, which we prove below for completeness: 
	
	\begin{proposition}[See \cite{TheThreeDimensionalNavierStokesEquationsClassicalTheory} Theorem D.6, p. 397 and Theorem D.7, p. 399]
		\label{Prop-Classical-Estimates}
	Let us consider the equation 
	 \begin{equation}
	   \label{Heat-Eq-R-d}
     \left\{
			\begin{array}{ll}
				\partial_t y - \Delta y
				 = h, \quad & \text{ in } (0, T) \times \R^d, 
				\\
				y(0,\cdot) = 0 \quad & \text{ in } \R^d,
			\end{array}
		\right.
	 \end{equation} 
	 for some $h \in L^p(0,T; L^p(\R^d))$, $p \in [1, \infty]$.
	\begin{enumerate}
		\item if $p \in [1, d+2)$, then there exists $C_{p,T}$ such that the solution $y$ of \eqref{Heat-Eq-R-d} satisfies $y \in L^q(0,T; W^{1,q}(\R^d))$ with 
		\begin{equation}
			\label{Estimee-Lp-vers-Lq}
			\norm{y}_{L^q_T(W^{1,q})} \leq C_{p,T} \norm{h }_{L^p_T(L^p)}, 
			\quad \text{ where $q$ is defined by } 
			\frac{1}{q}  = \frac{1}{p} - \frac{1}{d+2}.
		\end{equation}  
		\item  if $p \in (d+2,\infty]$, then there exists $C_{p,T}$ such that the solution $y$ of \eqref{Heat-Eq-R-d} satisfies $y \in L^\infty(0,T;W^{1,\infty}(\R^d))$ with 
		\begin{equation}
			\label{Estimee-Linfty-vers-Linfty}
			\norm{y}_{L^\infty_T(W^{1,\infty})} \leq C_{p,T} \norm{h }_{L^p_T(L^p)}.
		\end{equation}  
      
      \item furthermore, if $p=\infty$, then $y \in L^{\infty}(0,T;\BUC^{1}(\R^d))$.
		
	\end{enumerate}
	\end{proposition} 
	
	\begin{proof}[Proof of Proposition \ref{Prop-Classical-Estimates}]
		The solution $y$ of  \eqref{Heat-Eq-R-d} is given by 
		$$
			y(t) = \int_0^t \T_{t-s} h(s) \ {\rm d}s, \quad t \in [0,T], 
		$$
		where $(\T_t)_{t\geq 0}$ is the heat semigroup, simply given by the convolution in space with the heat kernel, and which satisfies, for $1 \leq p \leq q \leq \infty$, for all $t \geq 0$, 
		$$
			t^{\frac{d}{2} \left(\frac{1}{p} - \frac{1}{q} \right)}  \norm{ \T_t}_{\mathscr{L}(L^p, L^q)} 
			+
			t^{ \frac{1}{2} +\frac{d}{2} \left(\frac{1}{p} - \frac{1}{q} \right)} \norm{ \T_t}_{\mathscr{L}(L^p, L^q)} 
			\leq C.
		$$
		
		{\it Proof of Item $1$: $p \in [1, d+2)$.}  For all $t \in [0,T]$, we have
		\begin{align*}
			\| y(t) \|_{L^q} + \| \nabla y(t) \|_{L^q} 
			&
			\leq 
			\int_0^t 
			\left(
			 |t-s|^{-\frac{d}{2} \left(\frac{1}{p} - \frac{1}{q} \right)} 
			+ 
			 |t-s|^{- \frac{1}{2} -\frac{d}{2} \left(\frac{1}{p} - \frac{1}{q} \right)} 
			\right)
			\| h(s) \|_{L^p} \, {\rm d}s
			\\ 
			&
			\leq 
			C_T \int_0^t
			 |t-s|^{- \frac{1}{2} -\frac{d}{2} \left(\frac{1}{p} - \frac{1}{q} \right)} 
			\| h(s) \|_{L^p} \, {\rm d}s
			\\
			&\leq 
			C_T \int_0^t
			 |t-s|^{- 1 +  \frac{1}{d+2}} 
			\| h(s) \|_{L^p} \, {\rm d}s.			
		\end{align*}
		Using the Hardy-Littlewood-Sobolev inequality
        , we deduce \eqref{Estimee-Lp-vers-Lq}.\\
		
		{\it Proof of Item $2$:
        $p > d+2$.} For all $t \in [0,T]$, we have
			\begin{align*}
			\| y(t) \|_{L^\infty} + \| \nabla y(t) \|_{L^\infty} 
			&
			\leq 
			\int_0^t 
			\left(
			 |t-s|^{-\frac{d}{2p}  } 
			+ 
			 |t-s|^{- \frac{1}{2} -\frac{d}{2p} } 
			\right)
			\| h(s) \|_{L^p} \, {\rm d}s
			\\ 
			&
			\leq 
			C_T \int_0^t
			 |t-s|^{- \frac{1}{2} -\frac{d}{2p}} 
			\| h(s) \|_{L^p} \, {\rm d}s.
		\end{align*}
		Since $ \frac{1}{2} +\frac{d}{2p} < \frac{1}{p'} $ for $p > d +2$, we deduce \eqref{Estimee-Linfty-vers-Linfty} by the {H\"older} inequality.\\

		{\it Proof of Item $3$: $p=\infty$.} From the previous case, we already have that $y \in L^\infty(0,T;W^{1,\infty}(\R^d))$. Therefore, it only remains to check that $\nabla y$ belongs to $L^\infty(0,T; \BUC(\R^d))$. For $t \in (0,T]$ and $x_1, x_2 \in \R^d$, we have 
		\begin{equation*}
			\left| \nabla y(t,x_1) - \nabla y(t, x_2) \right|
			\leq 
		 	\int_0^t \int_{\R^d} \left| \nabla G(t-s ,x_1 - x_0) - \nabla G(t-s,x_2 -x_0) \right| \, {\rm d}x_0 {\rm d}s \| h \|_{L^\infty_T(L^\infty)}. 
		\end{equation*}
		We then estimate, for $\tau >0$,
		$$
			\int_{\R^d} \left| \nabla G(\tau ,x_1 - x_0) - \nabla G(\tau,x_2 -x_0) \right| \, {\rm d}x_0
			\leq
			|x_1 - x_2| \int_{\R^d} \left|D^2G(\tau, x_0)  \right| \, {\rm d}x_0
			\leq C\frac{|x_1-x_2|}{\tau}.
		$$
		Accordingly, separating the integral on $(0,t)$ in an integral on $(0,(t- |x_1-x_2|^2)_+)$ and $((t- |x_1-x_2|^2)_+),t)$, we obtain, for all $t \in (0,T]$, 
		\begin{align*}
		&\int_0^t \int_{\R^d} \left| \nabla G(t-s ,x_1 - x_0) - \nabla G(t-s,x_2 -x_0) \right| \, {\rm d}x_0 {\rm d}s
		\\&\leq 
		\int_0^{(t- |x_1-x_2|^2)_+} \frac{C |x_1-x_2|}{t-s} \, {\rm d}s + \int_{(t- |x_1-x_2|^2)_+}^t \frac{C}{\sqrt{t-s}} \, {\rm d}s 
		\leq C |x_1 - x_2| \left( 1+ \log\left(1+ \frac{T}{|x_1-x_2|^2}  \right) \right).
		\end{align*}
		This of course implies that $\nabla y \in L^\infty(0,T;\BUC(\R^d))$ for a source term $h \in L^\infty(0,T;L^\infty(\R^d))$, and concludes the proof of Proposition \ref{Prop-Classical-Estimates}.
	\end{proof}

\subsection{Preliminaries: Null-controllability of the heat equation in the $L^2$ setting}

Let us recall the results regarding the null-controllability of the heat equation. Since we will need estimates later in Theorem \ref{Thm-NC-BUC}, we choose to recall Carleman estimates for the heat equation, which will allow to handle potentials and semi-linear terms. While we could have chosen to follow the approach in \cite{FursikovImanuvilov}, it seems clearer to us to follow the approach in \cite{BEG}, which presents a Carleman estimate with a weight function which does not blow up close to the time $t = 0$. 

In this section, $T>0$, $\Omega$ is a smooth bounded domain of $\R^d$ and $\omega$ is a non-empty open subset of $\Omega$. For $\omega_0$ a non-empty open subset of $\omega$ with $\overline\omega_0 \subset \omega$, we choose a smooth (at least $\mathcal{C}^2$) function $\psi$ such that 
\begin{equation}
	\label{Psi}
	\psi := \psi(t,x) \quad \hbox{ such that}
	\left\{
		\begin{array}{ll}
			& \forall x \in \overline\Omega,\,   \psi(x)  \in [6,7],
			\\
			& \forall x \in \partial\Omega, \, \partial_{n}  \psi(x) \leq 0,
			\\
			 & \psi_{|\partial \Omega}  \text{ is constant, and }  \psi_{|\partial \Omega} = \inf_{\Omega} \psi, 
			 \\
			 & \inf_{\Omega \setminus { \omega_0 } }\{|\nabla \psi| \} >0.
		\end{array}
	\right.
\end{equation}
We then set $T_0>0$ and $T_1>0$ such that $T_1 \leq 1/4$ and $T_0+2 T_1 < T$ and choose a weight function in time $ \theta_{\mu}(t)$ depending on the parameter $\mu\geq 2$ defined by
\begin{equation}
	\label{ThetaMu}
	\theta_{\mu } = \theta_{ \mu}(t)
	 \hbox{ such that}
		\left\{
			\begin{array}{l}
			\ds \forall t \in [0,T_0],\, \theta_{\mu}(t) = 1+ \left( 1- \frac{t}{T_0} \right)^\mu,
			\smallskip
			\\
			\ds \forall t \in [T_0, T- 2T_1], \, \theta_{\mu}(t) = 1,
			\smallskip
			\\
			\ds \forall t \in [T-T_1,T), \,  \theta_{\mu}(t) = \frac{1}{(T-t)},
			\smallskip
			\\
			\ds \theta_{\mu} \hbox{ is increasing on } [T-2T_1, T-T_1],
			\smallskip
			\\
			\ds \theta_{\mu} \in \mathcal{C}^2([0,T)).
			\end{array}
		\right.
\end{equation}
For simplicity of notations in the following we omit the dependence on $\mu$ and we simply write $\theta$ instead of $\theta_{\mu}$.
We will then take the following weight functions $\varphi = \varphi(t,x)$ and $\xi = \xi(t,x)$:
\begin{equation}
	\label{Phi-Xi}
	\varphi(t,x) = \theta(t) \left(\lambda e^{12 \lambda}- \exp(\lambda \psi(t,x)) \right), \quad \xi(t,x) := \theta (t) \exp(\lambda \psi(t,x)),
\end{equation}
where  $s,\, \lambda$ are positive parameters with $s\geq 1$, $\lambda\geq 1$ and $\mu$ is chosen as
\begin{equation}
	\label{Def-mu}
	\mu = s \lambda^2 e^{2 \lambda},
\end{equation}
which is always bigger than $2$, thus being compatible with the condition $\theta \in \mathcal{C}^2([0,T))$. 

Remark that, due to the definition of $\psi$ in \eqref{Psi}, we have that, for all $\beta \in (0,1)$, there exists $\lambda_{0,\beta}$ such that for all $\lambda \geq \lambda_{0,\beta}$ and $(t,x) \in (0,T) \times \Omega$,
\begin{equation}
	\label{Phi-bounds}
	\beta \theta(t) \lambda e^{ 12 \lambda } \leq \varphi(t,x) \leq \theta(t) \lambda e^{12 \lambda}.
\end{equation}

We are now in position to recall the Carleman estimate obtained in \cite{BEG}.

\begin{theorem}[{\cite[Theorem 2.5]{BEG}}]
	\label{CarlemanThm}
	Under the above setting, there exist constants $C_0>0$, $s_0\geq 1$ and $\lambda_0\geq 1$ such that for all smooth functions
$z$ on $(0,T)\times \Omega$ satisfying $z = 0$ on $(0,T)\times \partial\Omega$, for all $s \geq s_0$, $\lambda \geq \lambda_0$, we have
	\begin{multline}
		\label{CarlemanEst}
			  \int_\Omega |\nabla z(0)|^2 e^{-2 s \varphi(0)} {\rm d}x + s^2 \lambda^3 e^{14 \lambda} \int_\Omega |z(0)|^2 e^{-2 s \varphi(0)} {\rm d}x
		  \\
			  + s \lambda^2 \iint_{(0,T) \times \Omega}  \xi |\nabla z|^2  e^{-2 s \varphi} {\rm d}x {\rm d}t
			 \ds + s^3 \lambda^4  \iint_{(0,T) \times \Omega} \xi^3 |z|^2 e^{-2 s \varphi} 	{\rm d}x {\rm d}t
		\\
			  \leq C_{0} \iint_{(0,T)\times \Omega} |(-  \partial_t - \Delta) z|^2 e^{-2 s \varphi} {\rm d}x {\rm d}t
			 + C_{0} s^3 \lambda^4 \iint_{(0,T) \times \omega} \xi^3 |z|^2 e^{-2 s \varphi} {\rm d}x {\rm d}t.
	\end{multline}
\end{theorem}
For $\beta \in (0,1)$, we take $\lambda_\beta = \max \{ \lambda_0, \lambda_{0,\beta}\}$, we can bound $\varphi$ by $\theta$ from below and from above by a constant depending on $\lambda$. Accordingly for all $\beta \in (0,1)$, there exist a constant $C>0$ such that 
for all smooth functions
$z$ on $(0,T)\times \Omega$ satisfying $z = 0$ on $(0,T)\times \partial\Omega$, for all $s \geq s_0$, we have
	\begin{multline}
		\label{CarlemanEst-2}
			  \int_\Omega |\nabla z(0)|^2 e^{-2 s \varphi(0)}  {\rm d}x + s^2  \int_\Omega |z(0)|^2 e^{-2 s \varphi(0)} {\rm d}x 
		  \\
			  + s  \iint_{(0,T) \times \Omega}  \theta |\nabla z|^2  e^{-2 s \varphi} {\rm d}x {\rm d}t
			 \ds + s^3  \iint_{(0,T) \times \Omega} \theta^3 |z|^2 e^{-2 s \varphi} 	{\rm d}x {\rm d}t
		\\
			  \leq C \iint_{(0,T)\times \Omega} |(-  \partial_t - \Delta) z|^2 e^{-2 s \varphi} {\rm d}x {\rm d}t
			 + C s^3 \iint_{(0,T) \times \omega} \theta^3 |z|^2 e^{-2 s \varphi} {\rm d}x {\rm d}t,
	\end{multline}
	where $\varphi$ and $\theta$ are the ones given in \eqref{ThetaMu} and \eqref{Phi-Xi} with $\lambda = \lambda_\beta$.

By duality and basic Hilbertian estimates for the heat equation, one obtains the following result (see \cite[Theorem 2.6]{BEG}, where this is proved for $y_0 = 0$; the case $y_0 \in L^2(\Omega)$ can be done similarly and is left to the reader):
\begin{theorem}[{\cite[Theorem 2.6]{BEG}}]
	\label{Thm-Est-Y-Carl-Norms}
	Under the above setting, there exist positive constants $C>0$ and $s_0 \geq 1$ such that for all $s \geq s_0$, for all $f$ satisfying
	\begin{equation}
		\label{Conditions-f-F}
		\iint_{(0,T) \times \Omega} \theta^{-3} |f|^2 e^{2 s \varphi} {\rm d}x {\rm d}t < \infty,
	\end{equation}
	and $y_0 \in L^2(\Omega)$, 
	there exists a solution $(Y,H)$ of the control problem 
	\begin{equation}
		\left\{
			\begin{array}{ll}
				\partial_t Y - \Delta Y = \mathbf{1}_{\omega}  H + f, \quad & \text{ in } (0,T) \times \Omega, 
				\\
				Y = 0, \quad & \text{ on } (0,T) \times \partial \Omega, 
				\\ 
				Y(0,\cdot) = y_0, \quad Y(T,\cdot) = 0 & \text { in } \Omega,
			\end{array}
		\right.
	\end{equation}
	which furthermore satisfies the following estimate:
	\begin{multline}
	\label{Est-Y-Gal}
		s^3   \iint_{(0,T)\times\Omega}  |Y|^2 e^{2 s \varphi}{\rm d}x {\rm d}t+ \iint_{\widehat{\omega_T}} \theta^{-3} |H|^2 e^{2 s \varphi}{\rm d}x {\rm d}t
		+
		s \iint_{(0,T)\times\Omega} \theta^{-2} |\nabla Y|^2 e^{2 s \varphi}{\rm d}x {\rm d}t+
	\\
		\leq
		C \iint_{(0,T)\times\Omega} \theta^{-3} |f|^2 e^{2 s \varphi}{\rm d}x {\rm d}t
		+ 
		C s   \int_\Omega |y_0|^2 e^{2 s \varphi(0)} {\rm d}x.
	\end{multline}
	where $\varphi$ and $\theta$ are the ones given in \eqref{ThetaMu} and \eqref{Phi-Xi} with $\lambda = \lambda_\beta$.
\end{theorem}

Now, for $M>0$, we will consider potentials $q$ and $W$ in the class 
 \begin{equation}
	\label{Assumptions-Ref-q-W-NC}
	q \in L^{\infty} (0,T; L^\infty(\Omega)), 
		\quad
		W  \in L^{\infty}(0,T; (L^\infty( \Omega))^d) 
		\quad\text{with }
		\| q \|_{L^{\infty}_T(L^\infty(\Omega))}  + \| W\|_{L^{\infty}_T(L^\infty(\Omega))} \leq M.
\end{equation}
Taking $s \ggg M^2$, a straightforward fixed point argument proves the following corollary:

\begin{corollary}
	\label{Cor-Est-Y-Carl-Norms}
	Under the above setting, for all $M>0$, there exist positive constants $C>0$ and $s_0 \geq 1$ such that for all $s \geq s_0$, for all $f$ satisfying \eqref{Conditions-f-F}, for all $y_0 \in L^2(\Omega)$, and for all $(q,W) \in L^\infty(0,T; L^\infty(\Omega))$ satisfying \eqref{Assumptions-Ref-q-W-NC}, 
	there exists a solution $(Y,H)$ of the control problem 
	\begin{equation}
		\label{NC-Heat-Eq-L2-potentials}
		\left\{
			\begin{array}{ll}
				\partial_t Y - \Delta Y + q Y + W \cdot \nabla Y = \mathbf{1}_{\omega}  H + f, \quad & \text{ in } (0,T) \times \Omega, 
				\\
				Y = 0, \quad & \text{ on } (0,T) \times \partial \Omega, 
				\\ 
				Y( 0,\cdot) = y_0, \quad Y(T,\cdot)= 0 & \text { in } \Omega,
			\end{array}
		\right.
	\end{equation}
	which furthermore satisfies the following estimate:
	\begin{multline}
	\label{Est-Y-Gal-q-W}
		s^3   \iint_{(0,T)\times\Omega}  |Y|^2 e^{2 s \varphi}{\rm d}x {\rm d}t+ \iint_{\widehat{\omega_T}} \theta^{-3} |H|^2 e^{2 s \varphi}{\rm d}x {\rm d}t
		+
		s \iint_{(0,T)\times\Omega} \theta^{-2} |\nabla Y|^2 e^{2 s \varphi}{\rm d}x {\rm d}t+
	\\
		\leq
		C \iint_{(0,T)\times\Omega} \theta^{-3} |f|^2 e^{2 s \varphi}{\rm d}x {\rm d}t
		+ 
		C s   \int_\Omega |y_0|^2 e^{2 s \varphi(0)}{\rm d}x.
	\end{multline}
	where $\varphi$ and $\theta$ are the ones given in \eqref{ThetaMu} and \eqref{Phi-Xi} with $\lambda = \lambda_\beta$.
	Besides, the map $(y_0, f)\mapsto (Y, H)$ is linear.
\end{corollary}

\subsection{Null-controllability of the heat equation in $\BUC$}

In this section, we consider the null-controllability problem for the heat equation in $\R^d$ when the control acts in 
$$
	\omega = \R^d \setminus \overline B_{\R^d} (2).
$$
with controls in $L^\infty(0,T;\BUC(\omega))$, initial datum in $\BUC^1(\R^d)$ and source terms in $L^\infty(0,T;\BUC(\R^d))$.

To be more precise, we prove the following result : 

\begin{theorem}
	\label{Thm-NC-BUC}
	Let $M>0$ and $\beta \in (0,1)$ and set %
	$$
		\forall t \in [0, T), \quad \Phi(t) = s \lambda e^{12 \lambda} \theta(t),  
		\qquad \text{ with } \lambda = \lambda_\beta, \, \text{ $s = s_0$ as in Corollary \ref{Cor-Est-Y-Carl-Norms}},\text{ and } \theta \text{ as in \eqref{ThetaMu}}.
	$$

	Then, for all $\gamma \in (0, \beta)$ there exists $C>0$ such that 
	for all $y_0 \in \BUC^1(\R^d)$, for all $f \in L^\infty(0,T;\BUC(\R^d))$ satisfying 
	$$
		\| f e^{\Phi} \|_{L^\infty_T(L^\infty)} < \infty,
	$$
	and for all $(q,W) \in L^\infty(0,T;\BUC(\R^d))$ satisfying 
	\begin{equation}
	\label{Assumptions-Ref-q-W-NC-Global}
		\| q \|_{L^{\infty}_T(L^\infty)}  + \| W\|_{L^{\infty}_T(L^\infty)} 	\leq M.
	\end{equation}
	there exist a control function $h $ with $ \mathbf{1}_{\omega} h \in L^\infty(0,T; \BUC(\R^d))$ and a controlled trajectory {$y \in \mathcal{C}([0,T];\BUC^1(\R^d))$} solving 
	\begin{equation}
		\label{NC-Heat-Eq-LOT-local}
		\left\{
			\begin{array}{ll}
				\partial_t y - \Delta y + q y + W\cdot \nabla y
				 = \mathbf{1}_{\omega} h + f , \quad & \text{ in } (0, T) \times \R^d, 
				\\
				y(0, \cdot) = y_0, \quad & \text{ in } \R^d.
			\end{array}
		\right.
	\end{equation}
and satisfying  
	\begin{equation}
		\label{NC-Goal}
		y{( T,\cdot)} = 0 \quad \text{ in } \R^d, 
	\end{equation}
	with the estimate
	\begin{equation}
		\label{Est-for-fixed-Points}
		\| y e^{\gamma \Phi} \|_{L^\infty_T(W^{1,\infty})}
		+
		\| \mathbf{1}_{\omega}  h e^{\gamma \Phi} \|_{L^\infty_T(L^\infty)}
		\leq 
		C \left( 
			\| f e^{\Phi} \|_{L^\infty_T(L^\infty)} 
			+ 
			\| y_0 \|_{W^{1,\infty}}
		\right).
	\end{equation}
	Besides, the map $(y_0, f)\mapsto (y, h)$ is linear.
\end{theorem}

\begin{proof}
	Set $\Omega = B_{\R^d}(5)$ and $\omega_0 = \Omega \setminus \overline B_{\R^d}(4)$. 
	According to Corollary \ref{Cor-Est-Y-Carl-Norms}, choosing $\eta_{45}$ a smooth cut-off function taking value $1$ in $B_{\R^d}(4)$ and compactly supported in $B_{\R^d}(5)$, 
	there exists a solution $(Y,H)$ of 
	\begin{equation*}
		\left\{
			\begin{array}{ll}
				\partial_t Y - \Delta Y + q Y + W \cdot \nabla Y = H \mathbf{1}_{\omega_0} + \eta_{45} f, \quad & \text{ in } (0,T) \times \Omega, 
				\\
				Y = 0, \quad & \text{ on } (0,T) \times \partial \Omega, 
				\\ 
				Y{( 0,\cdot)} = \eta_{45} y_0, \quad Y{( T,\cdot)}= 0 & \text { in } \Omega,
			\end{array}
		\right.
	\end{equation*}
	with $\theta^{-1} Y \exp( \beta \Phi) \in L^2(0,T;H^1(\Omega))$ and 
	\begin{equation}
		\label{Estimate-Phi-Y}
		\| \theta^{-1} Y e^{\beta \Phi} \|_{L^2_T(H^1(\Omega))} 
		\leq 
		C \| f e^{\Phi}\|_{L^\infty_T(L^\infty)}
		+ 
		C \| y_0 \|_{W^{1,\infty}}. 
	\end{equation}
	
	Set $\gamma \in (0, \beta)$. For $N  = d/2+3$ if $d$ is even, or $N = (d+1)/2 +2$ if $d$ is odd, we set, for $j \in \{0, \cdots, N\}$,
	$$
		R_j = 4 - \frac{j}{N},
		\quad \text{ and } \quad
		\beta_j = \beta \left(1- \frac{j}{N} \right) + \frac{j}{N} \gamma.
	$$
	so that $R_0 = 4 $ and $R_N = 3$, $\beta_0 = \beta$ and $\beta_N = \gamma$.
	
	 For $j \in \{1, \cdots, N\}$ we choose smooth cut-off functions $\eta_j$ taking value one in $B_{\R^d} (R_{j})$ and vanishing on $\Omega \setminus B_{\R^d}(R_{j-1})$. 
	We then set, for $j \in \{1, \cdots, N\}$, 
	$$
		z_j(t,x)  = Y(t,x) \eta_j(x) e^{\beta_j \Phi(t)}, \quad \text{ for } (t, x) \in [0,T] \times \Omega, 
	$$
	which we extend to $[0,T]\times \R^d$ by $0$  outside $[0,T] \times \Omega$. One easily checks that for each $j \in \{1, \cdots, N\}$, $z_j$ satisfies 	
	$$
		    \left\{
			\begin{array}{ll}
				\partial_t z_j - \Delta z_j
				 = h_j, \quad & \text{ in } (0, T) \times \R^d, 
				\\
				z_j(0,\cdot) = y_0 \eta_j e^{\beta_j\Phi(0)} \quad & \text{ in } \R^d,
			\end{array}
		\right.
	$$
	with 
	$$
		h_j  = \beta_j \Phi' Y \eta_j e^{\beta_j \Phi} 
		+ \eta_j f e^{\beta_j \Phi}
		- \eta_j (q Y e^{\beta_j \Phi} + W \cdot \nabla Y e^{\beta_j \Phi})
		- 2 \nabla \eta_j\cdot \nabla Y e^{\beta_j \Phi} - \Delta \eta_j Y e^{\beta_j \Phi}.
	$$

	Now, since $Y$ satisfies \eqref{Estimate-Phi-Y}, we use Proposition \ref{Prop-Classical-Estimates} to get recursively that, while $1/2 - j/(d+2) >0$, $z_j \in L^{q_j}(0,T;W^{1,q_j})$ with $1/q_j = 1/2 - j /(d+2)$ with 
	$$
		\| z_j \|_{L^{q_j}_T(W^{1,q_j})}
		\leq 
		C \| y_0 \|_{L^\infty} + C \| f e^{\beta \Phi}\|_{L^\infty_T(L^\infty)}.
	$$
	This follows directly from the fact that $ \beta = \beta_0 > \beta_1 > \cdots > \beta_N$,
	$$
		\| h_1\|_{L^{2}_T(L^{2})} \leq C \| \theta^{-1} Y e^{\beta \Phi} \|_{L^{2}_T(W^{1,2})} + C \| f e^{\beta \Phi}\|_{L^\infty_T(L^\infty)}, 
	$$
	and, for all $j \in \{2, \cdots,N\}$ satisfying  $1/2 - j/(d+2) >0$, since $z_{j-1} = Y e^{\beta_{j-1} \Phi}$ on the support of $\eta_j$,
	$$
		\| h_j\|_{L^{q_{j-1}}_T(L^{q_{j-1}})} \leq C \| z_{j-1}\|_{L^{q_{j-1}}_T(W^{1,q_{j-1}})} + C \| f e^{\beta \Phi}\|_{L^\infty_T(L^\infty)}. 
	$$

	If $d$ is odd, this yields, for $j_d = (d+1)/2$, that $z_{j_d} \in L^{2d+4}(0,T;W^{1,2d+4}(\R^d))$. Thus, $h_{j_d+1}$ belongs to $L^{2d+4}(0,T; L^{2d+4}(\R^d))$, and from Proposition \ref{Prop-Classical-Estimates} item 2, $z_{j_d +1}  \in L^{\infty}(0,T;W^{1,\infty}(\R^d))$. Accordingly, $h_{j_d +2}  \in L^{\infty}(0,T;L^{\infty}(\R^d))$, and from Proposition \ref{Prop-Classical-Estimates} item 3, $z_{j_d+2} = z_N \in L^{\infty}(0,T; \BUC^{1}(\R^d))$.
	
	If $d$ is even, we get, for $j_d = d/2$, that $z_{j_d} \in L^{d+2}(0,T;W^{1,d+2}(\R^d))$. Since we also have $\theta^{-1} Y e^{\beta \Phi} \in L^{2}(0,T; H^1(\R^d))$ and $z_{j_d} = \eta_{j_d} Y e^{\beta_{j_d} \Phi}$, we deduce that $z_{j_d} \in L^{\tilde q} (0,T; W^{1,\tilde q}(\R^d))$ for all $\tilde q \in (2,d+2)$. Taking $\tilde q < d+2$ and close to $d+2$, $h_{j_d+1} \in L^{\tilde q}(0,T; L^{\tilde q}(\R^d))$, and from Proposition \ref{Prop-Classical-Estimates} item 1, we deduce $z_{j_d+1} \in L^{q}(0,T; W^{1,q}(\R^d))$ for some $q > d+2$. Accordingly, $h_{j_d+2} \in L^{q}(0,T;L^{q}(\R^d))$ for some $q > d+2$, and from Proposition \ref{Prop-Classical-Estimates} item 2, $z_{j_d+2} \in L^{\infty}(0,T; W^{1,\infty}(\R^d))$.  We finally deduce, from Proposition \ref{Prop-Classical-Estimates} item 3, $z_{j_d+3} = z_N \in L^{\infty}(0,T; \BUC^{1}(\R^d))$.
	
	We now simply remark that $z_N = Y e^{\gamma \Phi}$ in $\overline B_{\R^d}(3)$, so that $Y e^{\gamma \Phi} \in  L^{\infty}(0,T; \BUC^{1}(\overline B_{\R^d}(3)))$. Taking a smooth cut-off function $\eta_{23}$ taking value $1$ in $\overline B_{R^d}(2)$, and vanishing outside of $B_{\R^d}(3)$, $\tilde y = \eta_{23} Y$ satisfies
	$$
	    	\left\{
			\begin{array}{ll}
				\partial_t \tilde y - \Delta \tilde y + q \tilde y + W \cdot \nabla \tilde y
				 = \eta_{23} f + \tilde h, \quad & \text{ in } (0, T) \times \R^d, 
				\\
				\tilde y (0,\cdot) = \eta_{23} y_0 \quad & \text{ in } \R^d, 
			\end{array}
		\right.
	$$
	with
	$$
		\tilde h = - 2 \nabla \eta_{23} \cdot \nabla Y - \Delta \eta_{23} Y + W \cdot \nabla \eta_{23} Y, 
	$$
	supported in $[0,T] \times \omega$, 
	and 
	$$
		\tilde Y(T,\cdot) = 0 \text{ in } \R^d.
	$$	
	Besides, we have the estimate 
	\begin{equation}
		\label{Est-tilde-y-and-h}
		\|e^{\gamma \Phi} \tilde y \|_{L^\infty_T(W^{1,\infty})}
		+
		\|e^{\gamma \Phi} \tilde h \mathbf{1}_{\omega} \|_{L^\infty_T(L^\infty)} 
		\leq 
		C \left( \| y_0 \|_{W^{1,\infty}} + \|f e^\Phi \|_{L^\infty_T(L^\infty)}\right).
	\end{equation}

	We also introduce $\check y$ the solution of 
	$$
	    	\left\{
			\begin{array}{ll}
				\partial_t \check y - \Delta \check y + q \check y + W \cdot \nabla \check y
				 = 0 , \quad & \text{ in } (0, T) \times \R^d, 
				\\
				\check y (0,\cdot) = y_0 \quad & \text{ in } \R^d, 
			\end{array}
		\right.
	$$
	which satisfies, according to Remark \ref{Rk-WP-BUC}, $\check{y}\in \mathcal{C}([0,T];\BUC^1(\R^d))$ and
	\begin{equation}
		\label{Est-check-y}
		\| \check y \|_{L^\infty_T(W^{1,\infty})} 
		\leq 
		C \| y_0 \|_{W^{1,\infty}}.
	\end{equation}
	Let us choose a smooth function $\rho$ on $[0,T]$ taking value $1$ close to $t = 0$ and vanishing close to $t = T$. Then, the function
	%
	$$
		y  := \rho \check y (1 - \eta_{23}) + \tilde y , 
	$$
	satisfies \eqref{NC-Heat-Eq-LOT-local} with
	$$
		h = 2 \rho \nabla \eta_{23} \cdot \nabla \check y + \rho \Delta \eta_{23} \check y + \rho' \check y (1- \eta_{23}) - (1-\eta_{23})f+ \tilde h.
	$$
	The functions $y$ and $h$ satisfy the estimate \eqref{Est-for-fixed-Points} due to  the estimates \eqref{Est-tilde-y-and-h}, \eqref{Est-check-y}, the explicit expression of $y$ and $h$, and the fact that $\exp(\gamma \Phi)$ is bounded on the support of $\rho$.
	
	The fact that the map $(y_0, f)\mapsto (y, h)$ is linear comes from the fact that $(y_0,f) \mapsto (Y,H)$ is linear from Corollary \ref{Cor-Est-Y-Carl-Norms} and from the fact that all the above construction is linear in $(y_0, f)$.
\end{proof}

For later use, we point out that a similar proof left to the reader yields the following result:
\begin{lemma}
	\label{Thm-Local-Regularity-Result}
	Let $d \geq 1$, $T>0$, and 
	\begin{equation}
		q \in L^{\infty} (0,T; L^\infty(\R^d)), 
		\quad
		W  \in L^{\infty}(0,T; (L^\infty( \R^d))^d). 
	\end{equation}
	Then, for any closed set $K$ of $\R^d$ and compact set $K_1$ of $\R^d$ which do not intersect $K$, there exists a constant $C>0$ such that for all $f \in L^2(0,T; L^2(\R^d))$ supported in $[0,T] \times K$, the  solution $ y$ of 
 \begin{equation}     \label{NC-Heat-Eq-LOT-localregularity}
     \left\{
			\begin{array}{ll}
				\partial_t  y - \Delta y+q  y +W\cdot\nabla y
				 = f, \quad & \text{ in } (0, T) \times \R^d, 
				\\
				 y(0,\cdot) = 0 \quad & \text{ in } \R^d,
            
			\end{array}
		\right.
 \end{equation} 
 belongs to $L^\infty(0,T; \BUC(K_1))$ and satisfies
 $$
 \|y \|_{L^\infty(0,T; W^{1,\infty}(K_1))} \leq C \| f\|_{ L^2(0,T; L^2(\R^d))}.
 $$
\end{lemma}

One can for instance introduce the compact sets defined for any $j\in\N$ by
$$
K_{1}^{j}:=\enstq{x\in\R^d}{\mathrm{dist}(x,K_1)\leq \frac{\mathrm{dist}(K,K_1)}{2^{j+1}}},
$$
choose $\eta_j\in\mathcal{C}^{\infty}_{c}(K_{1}^{j})$ such that $\eta_j=1$ on $K_{1}^{j+1}$, and work by induction on the sequence $(y^{j})_{j\in\N}$ defined 
$$
	y^{0}=\eta_{0}\check y\ \text{ and }\ y^{j+1}=\eta_{j+1}y^{j}\text{ for }j\in\N.
$$
Details are left to the reader as they closely follow the ones of the proof of Theorem \ref{Thm-NC-BUC}. 

\subsection{Null-Controllability in $X_\alpha$ and $X_\alpha^1$.}

Let us start with the linear case. 
\begin{theorem}
	\label{Thm-NC-linear-LOT}
	Let $d \geq 1$, $T>0$, and $\alpha >1$. 
	
	Let 
	\begin{equation}
		q \in L^{\infty} (0,T; X_\alpha), 
		\quad
		W  \in L^{\infty}(0,T; X_\alpha). 
	\end{equation}
	Then there exists a constant $C>0$ such that for any $y_0 \in X_\alpha $, there exists $h \in L^\infty(0,T; X_\alpha)$ satisfying
	\begin{equation}
		\label{Estimate-f-NC}
		\| \mathbf{1}_{\omega} h \|_{L^{\infty}_{T}(X_\alpha)} \leq C \|y_0\|_{X_\alpha},
	\end{equation}
	%
	for which the solution $y$ of 
	\begin{equation}
		\label{NC-Heat-Eq-LOT}
		\left\{
			\begin{array}{ll}
				\partial_t y - \Delta y + q y + W\cdot \nabla y
				 = \mathbf{1}_{\omega} h, \quad & \text{ in } (0, T) \times \R^d, 
				\\
				y(0, \cdot) = y_0, \quad & \text{ in } \R^d.
			\end{array}
		\right.
	\end{equation}
	belongs to $\mathcal{C}([0,T]; X_\alpha)$, and satisfies
	\begin{equation}
		\label{Control-Req}
		y(T, \cdot) = 0 \hbox{ in } \R^d.
	\end{equation}
\end{theorem}	

\begin{proof}
	We can assume without loss of generality that $y_0 \in X^1_\alpha$. Indeed, if we manage to prove Theorem \ref{Thm-NC-linear-LOT} for initial state $y_0 \in X_\alpha^1$, we can simply choose $h = 0$ on $(0,T/2)$, so that the application of Theorem \ref{Thm-WP-linear-LOT} with $f = 0$ yields $y \in \mathcal{C}([0,T/2]; X_\alpha)$ and $y(T/2 ) \in X_\alpha^1$, and then use a control function on the time interval $(T/2, T)$.

	Let us then consider $y_0 \in X^1_\alpha$. Applying Theorem \ref{Thm-NC-BUC} with $f = 0$, there exists $h$ such that $ \mathbf{1}_{\omega} h \in L^\infty(0,T;\BUC(\R^d))$ and the solution $y$ of \eqref{NC-Heat-Eq-LOT} satisfies \eqref{NC-Goal}. Accordingly $\mathbf{1}_{\omega}  h \in L^\infty(0,T;X_\alpha)$. Theorem \ref{Thm-NC-linear-LOT} then guarantees that $y \in \mathcal{C}([0,T], X_\alpha)$ and gives the estimate \eqref{Estimate-f-NC} on $h$.
\end{proof}

The semilinear case is similar, except that it involves a fixed point argument based on the estimate \eqref{Est-for-fixed-Points}:

\begin{theorem}
\label{Thm-NC-linear-semilinear}
	Let $d\geq 1$, $T>0$, and $\alpha>1$. 
	Let $\varepsilon >0$ and $g:[0,T]\times (\R^d\cup \overline{\Omega}_\alpha)\times\overline{B}_\C(\varepsilon)\times \overline{B}_{\C^d}(\varepsilon)\rightarrow \C$ such that the conditions \eqref{g-in-BUC}--\eqref{Holomorphic-Condition-g}--\eqref{g-Lipschitz}--\eqref{g-for-s=0} hold, and 
	%
	\begin{equation}
		\label{g-Lipschitz-3}
		g \in L^\infty([0,T] \times (\R^d \cup \overline\Omega_\alpha); W^{3, \infty} (\overline B_\C (\varepsilon) \times  \overline B_{\C^d} (\varepsilon))).
	\end{equation}

Then there exists $\delta_\alpha>0$ and a constant $C$ such that for any initial data $y_0\in X_{\alpha}^1$ satisfying the smallness condition
\begin{equation}
\label{eq:smallness condition Thm-NC-linear-semilinear}
	\|y_0\|_{X_\alpha^1}\leq \delta_\alpha,
\end{equation}
there exist a control function $h $ with $ \mathbf{1}_{\omega} h \in L^{\infty}(0,T; X_\alpha)$ and a controlled trajectory $y\in\mathcal{C}([0,T],X_\alpha^1)$ solution of 
\begin{equation}
		\label{NC-Heat-Eq-semilinear}
		\left\{
			\begin{array}{ll}
				\partial_t y - \Delta y 
				 = \mathbf{1}_{\omega} h + g(\cdot, \cdot, y, \nabla y), \quad & \text{ in } (0, T) \times \R^d, 
				\\
				y(0, \cdot) = y_0, \quad & \text{ in } \R^d,
			\end{array}
		\right.
	\end{equation}
	and satisfying the controllability requirement \eqref{Control-Req}.
\end{theorem}

\begin{proof}
	We set 
	$$
		q = -\partial_s g(\cdot, \cdot, 0, 0), \qquad W = -\nabla_{s_d} g(\cdot, \cdot, 0, 0). 
	$$
	so that for all $(s,s_d) \in \overline B_\C (\varepsilon) \times  \overline B_{\C^d} (\varepsilon)$, 
	$$
		g(\cdot, \cdot, s, s_d) 
		= 
		- q (\cdot, \cdot) s - W(\cdot, \cdot) \cdot s_d + \tilde g(\cdot, \cdot, s, s_d)
	$$
	and $\tilde g$ satisfies 
	$$
		\tilde g(\cdot, \cdot, 0, 0) = 0, \quad 
		\partial_s \tilde g(\cdot, \cdot, 0, 0) = 0, 
		\quad
		\nabla_{s_d} \tilde g(\cdot, \cdot, 0, 0) = 0 , 
	$$
	and condition \eqref{g-Lipschitz-3} implies
	\begin{multline}
		\label{Cond-tilde-g}
		\forall (s_1, s_{1,d}), (s_2, s_{2,d}) \in \overline B_\C (\varepsilon) \times  \overline B_{\C^d} (\varepsilon), 
		\\
		\sup_{(t,x) \in (0,T) \times (\R^d \cup \Omega_\alpha) }\left| \tilde g(\cdot, \cdot, s_1, s_{1,d}) - \tilde g(\cdot, \cdot, s_2, s_{2,d}) \right| \leq C |(s_1, s_{1,d}) - (s_2, s_{2,d})|\left( |(s_1, s_{1,d})| + (s_2, s_{2,d})|\right).
	\end{multline}
	
	We set $\beta = 3/4$, $\gamma = 1/2$, 
	$M = \| q \|_{L^{\infty}_T(L^\infty)}  + \| W\|_{L^{\infty}_T(L^\infty)} $, so that Theorem \ref{Thm-NC-BUC} applies. We choose $\Phi$ as in Theorem \ref{Thm-NC-BUC}. 
	
	For $R \in (0, \varepsilon]$, we define 
	$$
		\mathscr{C}_R =
		\enstq{y \in L^\infty(0,T;\BUC^1(\R^d))}{\| e^{\Phi/2} y \|_{L^{\infty}_{T}(W^{1,\infty})}\leq R}, 
	$$
	and the fixed point map 
	\begin{multline*}
		\Lambda_R : \widehat y \in \mathscr{C}_R \mapsto y 
		\text{ given by Theorem \ref{Thm-NC-BUC} solving the control problem }
		\\
		\left\{
			\begin{array}{ll}
				\partial_t y - \Delta y + q y + W \cdot y = \tilde g(\cdot,\cdot, \widehat y , \nabla \widehat y) +\mathbf{1}_{\omega}  h  , \quad & \text{ in } (0, T) \times \R^d, 
				\\
				y(0, \cdot) = y_0, \quad & \text{ in } \R^d, 
				\\ 
				y(T, \cdot) = 0, \quad & \text{ in } \R^d.
			\end{array}
		\right.
	\end{multline*}
	 This application is well-defined provided $e^\Phi \tilde g(\cdot, \cdot, \widehat y, \nabla \widehat y)$ belongs to $L^\infty(0,T;\BUC(\R^d))$ for $ \widehat y \in \mathscr{C}_R$. This is true since, according to condition \eqref{Cond-tilde-g}, we have, for all $ \widehat y \in \mathscr{C}_R$,
	 $$
	 	\| e^\Phi \tilde g(\cdot, \cdot, \widehat y, \nabla \widehat y) \|_{L^\infty_T(L^\infty)}
		\leq
		C \| e^{\Phi} (|\widehat y|^2 + |\nabla \widehat y|^2) \|_{L^\infty_T(L^\infty)} 
		\leq
		C R^2. 
	 $$
	 Using then Theorem \ref{Thm-NC-BUC}, we obtain
	\begin{equation}
		\label{Est-PtFixe-NonLin}
	 	\| y e^{\Phi/2} \|_{L^\infty_T(W^{1,\infty})}
		\leq 
		C \left( 
			\| e^{\Phi} \tilde g(\cdot, \cdot, \widehat y, \nabla \widehat y) \|_{L^\infty_T(L^\infty)} 
			+ 
			\| y_0 \|_{W^{1,\infty}}
		\right)
		\leq 
		C R^2 + C \delta_\alpha.
	\end{equation}
	We also have that for all $ \widehat y_1, \widehat y_2 \in \mathscr{C}_R$, 
	\begin{align*}
		& \| e^\Phi (\tilde g(\cdot, \cdot, \widehat y_1, \nabla \widehat y_1) - \tilde g(\cdot, \cdot, \widehat y_2, \nabla \widehat y_2))\|_{L^\infty_T(L^\infty)}
		\\
		&\leq
		C \| e^{\Phi} (|\widehat y_1 - \widehat y_2| + |\nabla \widehat y_1 - \nabla \widehat y_2|) (|(\widehat y_1, \nabla \widehat y_1)| + (|(\widehat y_1, \nabla \widehat y_1)|) \|_{L^\infty_T(L^\infty)} 
		\\
		&\leq
		C R \| e^{\Phi/2}(\widehat y_1 - \widehat y_2) \|_{L^\infty_T(W^{1,\infty})}.  
	\end{align*}
	Accordingly, for $\delta_\alpha$ small enough, $R= 2 C \delta_\alpha$ is smaller than $\varepsilon$ and satisfies $CR^2 \leq R/2$ (where $C$ is the constant in \eqref{Est-PtFixe-NonLin}), the map $\Lambda_R$ maps $ \mathscr{C}_R$ into itself, and is contractive on $ \mathscr{C}_R$. 

	It follows that $\Lambda_R$ possesses a unique fixed point in the class $\mathscr{C}(R)$, which by construction satisfies \eqref{NC-Heat-Eq-semilinear} for a suitable $h$ satisfying $\mathbf{1}_{\omega}  h \in L^\infty(0,T;\BUC(\omega))$ and the controllability requirement \eqref{Control-Req}.
	
	Now, since $\mathbf{1}_{\omega}  h \in L^\infty(0,T;\BUC(\omega))$, it also belongs to $L^\infty(0,T;X_\alpha)$. The existence result in Theorem \ref{Thm-Semilinear-HE} gives a solution $\tilde y$ of \eqref{NC-Heat-Eq-semilinear} in the class $\mathcal{C}([0,T]; X^1_\alpha)$. This solution $\tilde y$ coincides with the function $y$ constructed above from the uniqueness result in Remark \ref{Rem-Uniqueness}, at least for $\delta_\alpha$ small enough.
\end{proof}

\begin{remark}
	\label{Rem-Control-Semilinear-in-BUC}
	The proof of Theorem \ref{Thm-NC-linear-semilinear} also proves that under the assumptions of Theorem \ref{Thm-NC-linear-semilinear}, there exists a constant $\delta>0$ such that for any initial condition $y_0 \in \BUC^1(\R^d)$ satisfying $\| y_0 \|_{W^{1, \infty}(\R^d)} \leq \delta$, one can find a control function $h$ with $\mathbf{1}_{\omega}  h \in L^\infty(0,T; \BUC(\R^d))$ and a controlled trajectory $y \in \mathcal{C} ([0,T]; \BUC^1(\R^d))$ solution of \eqref{NC-Heat-Eq-semilinear} and satisfying the controllability requirement \eqref{Control-Req}.
\end{remark}

\section{On the reachable sets of the heat equation with various lower order terms}
\label{Sec-Reachable-Heat}
\subsection{Holomorphy of the reachable states}

In this section, we will prove that all the reachable states are holomorphic functions on $\Omega_1$. This follows from the following lemma.
\begin{lemma}
\label{lem:smoothing effects}
Let $\alpha_0 \in (0,1)$, $T>0$, and consider $q$ and $W$ such that 
$$
	q \in L^\infty(0, T; \mathscr{R}_{\alpha_0}), 
	\quad 
	\hbox{ and } 
	\quad 
	W \in  L^\infty(0, T; (\mathscr{R}_{\alpha_0})^d), 
$$
Let $y_0\in L^2(\Omega)$. Let us consider the solution $y$ of
\begin{equation}
	\label{Controlled-Heat_smoothing}
	\left\{
		\begin{array}{ll}
			\partial_t y - \Delta y + q y + W \cdot \nabla y= 0, \quad & \hbox{ in } (0,T) \times \Omega,
			\\ 
			y = u, \quad & \hbox{ on } (0,T) \times \partial\Omega, 
			\\
			y(0, \cdot ) =y_0, \quad & \hbox{ in } \Omega,
		\end{array}
	\right.
\end{equation}
with $u\in L^{2}((0,T)\times\partial\Omega)$. Then $y(T,\cdot)\in \mathrm{Hol}(\Omega_1)$.
\end{lemma}

\begin{proof}[Proof of Lemma \ref{lem:smoothing effects}]
Let $y$ be a solution of \eqref{Controlled-Heat_smoothing} with $u\in L^{2}((0,T)\times\partial\Omega)$. Since \eqref{Controlled-Heat} is null controllable at any positive time (recall Theorem \ref{Thm-Est-Y-Carl-Norms}, coming from \cite[Theorem 2.6]{BEG}) and is linear, we also have exact controllability to trajectories. Therefore, we can find $z\in\mathcal{C}([0,T],H^{-1}(\Omega))$ such that
\begin{equation}
\label{eq: Chaleur control a zero}
\left\{
			\begin{array}{ll}
				\partial_t z- \Delta  z+q z +W\cdot\nabla z
				 = 0, \quad & \text{ in } (0, T) \times \Omega, 
				\\
				z(t,\cdot) = u(t,\cdot), \quad & \text{ on } (0,T) \times \Omega,\\
         		        z(0,\cdot) = 0, \quad & \text{ in } \Omega,
			\end{array}
		\right.
\end{equation}
satisfies
\begin{equation}
	z(T, \cdot)= y(T, \cdot) \quad \text{ in } \Omega. 
\end{equation}

Note that $z$ is smooth away from $\partial \Omega$ due to the local regularizing effects of the heat equation (the proof can be done along the same lines as the one of Lemma \ref{Thm-Local-Regularity-Result}).

Let $\alpha>1$ and $r \in (0,1)$.  Let us first mention that, by a simple scaling argument, Theorems \ref{Thm-Well-Posedness-X-alpha}-\ref{Thm-Well-Posedness-X-alpha-1}-\ref{Thm-WP-with-source-terms}-\ref{Thm-WP-linear-LOT}-\ref{Thm-Semilinear-HE} and Remarks \ref{Rk-Other-Time-Weights-on-f}-\ref{Rk-A-compact-form}-\ref{Rem-Uniqueness} can be generalized to the case $\Omega=B_{\R^d}(r)$, with the framework based on the space $X_{\alpha,r}$, defined in \eqref{Def-X-alpha-lambda}, instead of $X_{\alpha}$ for $\Omega=B_{\R^d}(1)$.

Now choose $r_1\in(r,1)$ and $\chi\in \mathcal{C}^{\infty}_{c}(B_{\R^d}(r_1))$ such that $\chi=1$ on a neighborhood of $\overline{B_{\R^d}}(r)$. 
Let us set $\tilde z:=\chi z$, $\tilde q:=\chi q$ and $\tilde W:= \chi W$. Since $\alpha>1>\alpha_0$, the functions $\tilde q$ and $\tilde W$ belong to  $L^\infty(0,T;X_{\alpha,r})$. Furthermore, extending $\tilde z$ by $0$ on $\R^d$, $\tilde z$ is a solution of
\begin{equation}
	\label{Controlled-Heat_cutoff}
	\left\{
		\begin{array}{ll}
			\partial_t \tilde z - \Delta \tilde z + \tilde q \tilde z + \tilde W \cdot \nabla \tilde z = f, \quad & \hbox{ in } (0,T) \times \R^d,
			\\
			\tilde z(0, \cdot ) =0, \quad & \hbox{ in } \R^d,
		\end{array}
	\right.
\end{equation}
where
$$
f:=[\chi,\Delta]z - [\chi,\tilde W \cdot \nabla]z-(1-\chi)\left(\tilde q+\Tilde W\cdot\nabla z \right)
$$
satisfies, for almost every $t\in [0, T]$, 
$$
\mathop{Supp\,}(f(t,\cdot))\subset 
\left(B_{\R^d}(r_1)\setminus \overline{B_{\R^d}}(r)\right).
$$
Since the support of $\tilde z$ is contain in $[0,T]\times B_{\R^d}(r)$, it follows from Lemma \ref{Thm-Local-Regularity-Result} that $\tilde z\in \mathcal{C}([0,T];\BUC^1(\R^d))$ and $f\in L^\infty(0,T;\BUC(\R^d))$. Since $\tilde{z}(0,\cdot)=0$, we deduce from Remark \ref{Rem-Uniqueness} that $\tilde{z}$ is the unique solution of \eqref{Controlled-Heat_cutoff} in $\mathcal{C}([0,T];\BUC^1(\R^d))$. Moreover, since $f\in \mathcal{C}([0,T];X_{\alpha,r})$, we deduce that $\tilde{z}$ is the solution of \eqref{Controlled-Heat_cutoff} in $\mathcal{C}([0,T];X_{\alpha,r})$ provided by Theorem \ref{Thm-WP-linear-LOT}. In particular, $y(T,\cdot)_{|B_{\R^d}(r)}=z(T,\cdot)_{|B_{\R^d}(r)} = \tilde z(T,\cdot)_{|B_{\R^d}(r)} $ has a unique holomorphic extension on $r\Omega_\alpha$. Since $r\in (0,1)$ and $\alpha>1$ are arbitrary chosen, we deduce that $y(T,\cdot)\in \mathrm{Hol}(\Omega_1)$.
\end{proof}

\subsection{Proof of Theorem \ref{Thm-Main-Linear}}
\label{Subsec: Proof of the main theorem in the linear case}
Lemma \ref{lem:smoothing effects} immediately implies that $\mathscr{R}_{lin}(y_0,T)\subset \mathrm{Hol}(\Omega_1)$. It thus remains to prove the inclusion 
$$
	\bigcup_{ \alpha \in (0,1)} \mathscr{R}_\alpha \subset \mathscr{R}_{lin}(y_0, T),
$$ 
where the spaces $\mathscr{R}_\alpha$ are defined in \eqref{Def-R-alpha}.

{\bf The case $y_0 = 0$.}
We consider the setting of Theorem \ref{Thm-Main-Linear} and first focus on the case $y_0 = 0$. Following the assumptions and notations of Theorem \ref{Thm-Main-Linear}, we let $\alpha_0 \in (0,1)$, we fix $T>0$, and consider $q$ and $W$ such that 
$$
	q \in L^\infty(0, T; \mathscr{R}_{\alpha_0}), 
	\quad 
	\hbox{ and } 
	\quad 
	W \in  L^\infty(0, T; (\mathscr{R}_{\alpha_0})^d). 
$$

For $\alpha \in (0,1)$, we choose 
$$
	\alpha_1  \in (\max \{ \alpha, \alpha_0\}, 1), 
$$
and we introduce a smooth cut-off function $\eta$ compactly supported on $B_{\R^d}(\alpha_1/\alpha)$ and taking value one in $B_{\R^d} (1)$. For any $y_1 \in \mathscr{R}_\alpha$, the function $\tilde y_1$ given by
\begin{equation}
	\label{Def-tilde-y-1}
	\tilde y_{1}(x) := 
	\left\{ 
		\begin{array} {ll}
			\ds \eta (x) y_1\left( \frac{\i x}{\alpha_1} \right) &\hbox{ for } x \in B_{\R^d} (\alpha_1/\alpha),
			\\
			0 &\hbox{ for } x \in \R^d \setminus B_{\R^d} (\alpha_1/\alpha),
		\end{array}
	\right. 
\end{equation}
belongs to $X_{1/\alpha_1}(\R^d)$, and 
\begin{equation}
\label{eq: estime de tilde y1 a y1}
\|\tilde y_1\|_{X_{1/\alpha_1}}\leq C\|y_1\|_{L^\infty(\Omega_{\alpha_1})}.
\end{equation}
This relies on the fact that $z \in \Omega_{1/\alpha_1}$ is equivalent to $ \i z/\alpha_1 \in \Omega_{\alpha_1}$.

Similarly, setting $T_1 := \alpha_1^2 T$, for $t \in (0, T_1)$, we set 
$$
	\tilde q(t,x) := 
	\left\{ 
		\begin{array} {ll}
			\ds - \alpha_1^2 \eta ( x) q\left(\frac{(T_1-t)}{\alpha_{1}^{2}}, \frac{\i x}{\alpha_1}\right) &\hbox{ for } x \in B_{\R^d} (\alpha_1/\alpha),
			\\
			0 &\hbox{ for } x \in \R^d \setminus B_{\R^d} (\alpha_1/\alpha),
		\end{array}
	\right. 
$$
and
$$
	\tilde W(t,x) := 
	\left\{ 
		\begin{array} {ll}
			\ds \i\alpha_1 \eta ( x) W\left(\frac{(T_1-t)}{\alpha_{1}^{2}} , \frac{\i x}{\alpha_1}\right) &\hbox{ for } x \in B_{\R^d} (\alpha_1/\alpha),
			\\
			0_{\R^d} &\hbox{ for } x \in \R^d \setminus B_{\R^d} (\alpha_1/\alpha),
		\end{array}
	\right. 
$$
for which we have 
$$
	\tilde q \in L^\infty(0, T_1; X_{1/\alpha_1}), 
	\quad 
	\hbox{ and } 
	\quad 
	\tilde W \in  L^\infty(0, T_1; X_{1/\alpha_1}). 
$$
Applying then Theorem \ref{Thm-NC-linear-LOT}, with $\omega = \R^d \setminus \overline B_{\R^d}(2)$, we get the existence of a controlled trajectory $\tilde y$ solution of 
\begin{equation}
	\label{NC-Heat-Eq-LOT-tilde-Last}
		\left\{
			\begin{array}{ll}
				\partial_t \tilde y - \Delta_x \tilde y + \tilde q\tilde y + \tilde W \cdot \nabla_x \tilde y
				 = \mathbf{1}_{\omega} h, \quad & \text{ in } (0, T_1) \times \R^d, 
				\\
				\tilde y(0, \cdot) = \tilde y_{1}, \quad & \text{ in } \R^d,
			\end{array}
		\right.
\end{equation}
which belongs to $\mathcal{C}([0,T_1]; X_{1/\alpha_1})$, and satisfies
\begin{equation}
	\label{Control-Req-T_1}
	\tilde y(T_1, \cdot) = 0 \hbox{ in } \R^d.
\end{equation}

We then set, for $t \in (0,T)$, and $x \in \Omega_{\alpha_1}$,
\begin{equation}
	\label{Def-y-controlled}
	y(t,x ) := \tilde y\left(T_1 -\alpha_1^2t, -\alpha_1 \i x \right). 
\end{equation}
Since $\tilde y$ belongs to $\mathcal{C}([0,T_1]; X_{1/\alpha_1})$, the equation \eqref{NC-Heat-Eq-LOT-tilde-Last} holds for $t \in [0,T_1]$ and $ x \in \Omega_{1/\alpha_1}$. Easy computations then show that $y$ satisfies
\begin{equation}
	\label{NC-Heat-Eq-LOT-Last-y}
		\left\{
			\begin{array}{ll}
				\partial_t y - \Delta_x y + q  y +  W \cdot \nabla_x y
				 = 0, \quad & \text{ in } (0,T) \times \Omega, 
				\\
				y(0, \cdot) = 0, \quad & \text{ in } \Omega,
				\\
				y(T, \cdot) = y_1 \quad & \text{ in } \Omega,
			\end{array}
		\right.
\end{equation}
and $y \in \mathcal{C}([0,T]; \mathscr{R}_{\alpha_1})$. 

Since $h$ in  \eqref{NC-Heat-Eq-LOT-tilde-Last} satisfies $\|\mathbf{1}_{\omega}  h \|_{L^{\infty}_{T_1}(X_{1/\alpha_1})} \leq C \| \tilde y_{1} \|_{X_{1/\alpha_1}} $ from Theorem \ref{Thm-NC-linear-LOT} and $\| \tilde y_{1} \|_{X_{1/\alpha_1}} \leq C \| y_1 \|_{L^{\infty}(\Omega_\alpha)}$, we obviously have $\| y \|_{L^{\infty}_{T}(L^{\infty}(\Omega_{\alpha_1}))} \leq C \| y_1\|_{L^{\infty}(\Omega_\alpha)}$. The control $u$ in \eqref{Controlled-Heat} is then simply given by the trace of $y$ at $x \in \S^{d-1}$ and thus immediately satisfies $\| u \|_{L^{\infty}_{T}( L^\infty(\partial\Omega))} \leq C \| y_1\|_{L^{\infty}(\Omega_\alpha)}$ as claimed in Theorem \ref{Thm-Main-Linear}.

{\bf The case $y_0 \neq 0$.} We simply remark that if $y_0 \in L^2(\Omega)$, the heat equation \eqref{Controlled-Heat} being null-controllable in any positive time by classical control results (see for instance \cite{Cabanillas-DeMenezes-Zuazua} in the presence of potentials), one can construct a control $u \in L^2(0,T/2; L^2(\partial\Omega))$ such that the solution $y$ of \eqref{Controlled-Heat} on $(0,T/2)$ satisfies $y(T/2,\cdot) = 0$ in $\Omega$. We then use the previous case on the time interval $(T/2,T)$ to show that, for any $\alpha \in (0,1)$, we can reach any state in $\mathscr{R}_\alpha$ for $\alpha \in (0,1)$ at time $T$.

\subsection{Proof of Theorem \ref{Thm-Main-Semilinear}: On the reachable sets of the semilinear heat equation}

The approach is the same as before. 

We consider a semilinearity $g$ as in Theorem \ref{Thm-Main-Semilinear} with $\alpha_0<1$.

\textbf{The case $y_0=0$.} We follow the same path than in Section \ref{Subsec: Proof of the main theorem in the linear case}. For $\alpha \in (0,1)$, we choose $\alpha_1  \in (\max \{ \alpha, \alpha_0\}, 1)$, 
and we introduce a smooth cut-off function $\eta$ compactly supported on $B_{\R^d}(\alpha_1/\alpha)$ and taking value one in $B_{\R^d} (1)$. Then, for $y_1 \in \mathscr{R}_\alpha$ with $\nabla y_1 \in \mathscr{R}_\alpha$, we introduce the function $\tilde y_1$ as in \eqref{Def-tilde-y-1}. It is easy to check, again based on the identity $\Omega_{\alpha_1}=\i/\alpha_1\Omega_{1/\alpha_{1}}$, that $\tilde y_1$ belongs to $X_{1/\alpha_1}(\R^d)$. 

Setting $T_1 := \alpha_1^2 T$, for $t \in (0, T_1)$, we define a function
$\tilde g:[0,T]\times (\R^d\cup \overline{\Omega}_{1/\alpha_1})\times\overline{B}_\C(\varepsilon)\times \overline{B}_{\C^d}(\alpha_1 \varepsilon)\rightarrow \C$ by 
$$
	\tilde g(t,x,s,s_d) := 
	\left\{ 
		\begin{array} {ll}
			\ds - \alpha_1^2 \eta ( x) g\left(\frac{(T_1-t)}{\alpha_{1}^{2}},\frac{\i x}{\alpha_1},s,-\frac{\i s_d}{\alpha_1} \right) &\hbox{ for } x \in B_{\R^d} (\alpha_1/\alpha),
			\\
			0 &\hbox{ for } x \in \R^d \setminus B_{\R^d} (\alpha_1/\alpha),
		\end{array}
	\right. 
$$
for which we have the hypotheses of Theorem \ref{Thm-NC-linear-semilinear}.

Then, by choosing $\delta_\alpha>0$ small enough so that Theorem \ref{Thm-NC-linear-semilinear} apply with initial data $\tilde y_1$ thanks \eqref{eq: estime de tilde y1 a y1}, we deduce that there exists a controlled trajectory $\tilde y$ solution of 
\begin{equation}
	\label{NC-Heat-Eq-LOT-tilde-Last-nonlinear}
		\left\{
			\begin{array}{ll}
				\partial_t \tilde y - \Delta \tilde y + \tilde g(\tilde y, \nabla \tilde y)
				 = \mathbf{1}_{\R^d\setminus \overline{B}_{\R^d}(1)} h, \quad & \text{ in } (0, T_1) \times \R^d, 
				\\
				\tilde y(0, \cdot) = \tilde y_{1}, \quad & \text{ in } \R^d,
			\end{array}
		\right.
\end{equation}
which belongs to $\mathcal{C}([0,T_1]; X_{1/\alpha_1})$, and satisfies \eqref{Control-Req-T_1}.

We then set $y$ as in \eqref{Def-y-controlled} for $t \in (0,T)$, and $x \in \Omega_{\alpha_1}$. 
Since $\tilde y$ belongs to $\mathcal{C}([0,T_1]; X_{1/\alpha_1})$, the equation \eqref{NC-Heat-Eq-LOT-tilde-Last-nonlinear} holds on $[0,T_1]\times \Omega_{1/\alpha_1}$, and we recover that $y$ solves 
\begin{equation}
	\label{NC-Heat-Eq-LOT-Last-y-nonlinear}
		\left\{
			\begin{array}{ll}
				\partial_t y - \Delta y + g(y, \nabla y)
				 = 0, \quad & \text{ in } (0,T) \times \Omega, 
				\\
				y(0, \cdot) = 0, \quad & \text{ in } \Omega,
				\\
				y(T, \cdot) = y_1 \quad & \text{ in } \Omega,
			\end{array}
		\right.
\end{equation}
where $y \in \mathcal{C}([0,T]; \mathscr{R}_{\alpha_1})$ with $\nabla y \in  \mathcal{C}([0,T]; \mathscr{R}_{\alpha_1})$ and the control $u$ is simply given by the restriction of $y$ to $(0,T)\times\partial\Omega$. 

{\bf The case $y_0 \neq 0$.} When $y_0 \in \mathcal{C}^1(\overline\Omega)$, we can simply extend $y_0$ as a $\mathcal{C}^1$ function to a neighborhood of $\overline\Omega$ which is compactly supported. We can then use Remark \ref{Rem-Control-Semilinear-in-BUC} to get the existence of a control function $h$ supported in $(0,T/2) \times (\R^d \setminus B) $ such that the solution $y$ of \eqref{NC-Heat-Eq-semilinear} vanishes at time $T/2$ provided the initial datum $y_0$ is small enough in $\mathcal{C}^1$. By restriction to $\Omega$, we obtain a control function $u \in L^\infty(0,T/2;W^{1,\infty}(\partial\Omega))$ such that the solution $y$ of \eqref{Controlled-Heat-Non-linear} vanishes at time $T/2$. We can then combine it with the previous case considered in the time interval $(T/2, T)$.


\begin{thebibliography}{10}

\bibitem{Aikawa-Hayashi-Saitoh-90}
H.~Aikawa, N.~Hayashi, and S.~Saitoh.
\newblock The {B}ergman space on a sector and the heat equation.
\newblock {\em Complex Variables Theory Appl.}, 15(1):27--36, 1990.

\bibitem{BEG}
M.~Badra, S.~Ervedoza, and S.~Guerrero.
\newblock Local controllability to trajectories for non-homogeneous
  incompressible {N}avier-{S}tokes equations.
\newblock {\em Ann. Inst. Henri Poincar{\'e}, Anal. Non Lin{\'e}aire},
  33(2):529--574, 2016.

\bibitem{Cabanillas-DeMenezes-Zuazua}
V.~R. Cabanillas, S.~B. De~Menezes, and E.~Zuazua.
\newblock Null controllability in unbounded domains for the semilinear heat
  equation with nonlinearities involving gradient terms.
\newblock {\em J. Optim. Theory Appl.}, 110(2):245--264, 2001.

\bibitem{Darde-Ervedoza-16}
J.~Dard\'{e} and S.~Ervedoza.
\newblock On the reachable set for the one-dimensional heat equation.
\newblock {\em SIAM J. Control Optim.}, 56(3):1692--1715, 2018.

\bibitem{Engel-Nagel}
K.-J. Engel and R.~Nagel.
\newblock {\em One-parameter semigroups for linear evolution equations}, volume
  194 of {\em Graduate Texts in Mathematics}.
\newblock Springer-Verlag, New York, 2000.
\newblock With contributions by S. Brendle, M. Campiti, T. Hahn, G. Metafune,
  G. Nickel, D. Pallara, C. Perazzoli, A. Rhandi, S. Romanelli and R.
  Schnaubelt.

\bibitem{Ervedoza-LeBalch-Tucsnak}
S.~Ervedoza, K.~Le~Balc'h, and M.~Tucsnak.
\newblock Reachability results for perturbed heat equations.
\newblock {\em J. Funct. Anal.}, 283(10):Paper No. 109666, 61, 2022.

\bibitem{ErvZuazuaARMA}
S.~Ervedoza and E.~Zuazua.
\newblock Sharp observability estimates for heat equations.
\newblock {\em Arch. Ration. Mech. Anal.}, 202(3):975--1017, 2011.

\bibitem{Fabre-Puel-Zuazua-Approx}
C.~Fabre, J.-P. Puel, and E.~Zuazua.
\newblock Approximate controllability of the semilinear heat equation.
\newblock {\em Proc. Roy. Soc. Edinburgh Sect. A}, 125(1):31--61, 1995.

\bibitem{FattoriniRussel71}
H.~O. Fattorini and D.~L. Russell.
\newblock Exact controllability theorems for linear parabolic equations in one
  space dimension.
\newblock {\em Arch. Rational Mech. Anal.}, 43:272--292, 1971.

\bibitem{FatRus}
H.~O. Fattorini and D.~L. Russell.
\newblock Uniform bounds on biorthogonal functions for real exponentials with
  an application to the control theory of parabolic equations.
\newblock {\em Quart. Appl. Math.}, 32:45--69, 1974/75.

\bibitem{FernandezCaraZuazua1}
E.~Fern{\'a}ndez-Cara and E.~Zuazua.
\newblock The cost of approximate controllability for heat equations: the
  linear case.
\newblock {\em Adv. Differential Equations}, 5(4-6):465--514, 2000.

\bibitem{FursikovImanuvilov}
A.~V. Fursikov and O.~Y. Imanuvilov.
\newblock {\em Controllability of evolution equations}, volume~34 of {\em
  Lecture Notes Series}.
\newblock Seoul National University Research Institute of Mathematics Global
  Analysis Research Center, Seoul, 1996.

\bibitem{Hartmann-Kellay-Tucsnak}
A.~Hartmann, K.~Kellay, and M.~Tucsnak.
\newblock From the reachable space of the heat equation to hilbert spaces of
  holomorphic functions.
\newblock {\em J. Eur. Math. Soc. (JEMS)}, 22(10):3417--3440, 2020.

\bibitem{Hartmann-Orsoni-2021}
A.~Hartmann and M.-A. Orsoni.
\newblock Separation of singularities for the {B}ergman space and application
  to control theory.
\newblock {\em J. Math. Pures Appl. (9)}, 150:181--201, 2021.

\bibitem{Hartman-Orsoni-2022-Hermite}
A.~Hartmann and M.-A. Orsoni.
\newblock Reachable space of the {H}ermite heat equation with boundary control.
\newblock {\em SIAM J. Control Optim.}, 60(6):3409--3429, 2022.

\bibitem{Imanuvilov-Yamamoto-2001}
O.~Yu. Imanuvilov and M.~Yamamoto.
\newblock Carleman estimate for a parabolic equation in a {S}obolev space of
  negative order and its applications.
\newblock In {\em Control of nonlinear distributed parameter systems ({C}ollege
  {S}tation, {TX}, 1999)}, volume 218 of {\em Lecture Notes in Pure and Appl.
  Math.}, pages 113--137. Dekker, New York, 2001.

\bibitem{Kellay-Normand-Tucsnak}
K.~Kellay, T.~Normand, and M.~Tucsnak.
\newblock Sharp reachability results for the heat equation in one space
  dimension.
\newblock {\em Anal. PDE}, 15(4):891--920, 2022.

\bibitem{Exactcontrollabilityofanisotropic1Dpartialdifferentialequationsinspacesofanalyticfunctions}
C.~Laurent, I.~Rivas, and L.~Rosier.
\newblock Exact controllability of anisotropic 1d partial differential
  equations in spaces of analytic functions.
\newblock 2025.

\bibitem{Laurent-Rosier}
C.~Laurent and L.~Rosier.
\newblock Exact controllability of semilinear heat equations in spaces of
  analytic functions.
\newblock {\em Ann. Inst. H. Poincar\'{e} C Anal. Non Lin\'{e}aire},
  37(4):1047--1073, 2020.

\bibitem{MartinRosierRouchon-2016}
P.~Martin, L.~Rosier, and P.~Rouchon.
\newblock On the reachable sets for the boundary control of the heat equation.
\newblock {\em Applied Mathematics Research eXpress}, 2016.

\bibitem{Orsoni-JFA-2021}
M.-A. Orsoni.
\newblock Reachable states and holomorphic function spaces for the 1-{D} heat
  equation.
\newblock {\em J. Funct. Anal.}, 280(7):Paper No. 108852, 17, 2021.

\bibitem{TheThreeDimensionalNavierStokesEquationsClassicalTheory}
J.~C. Robinson, J.~L. Rodrigo, and W.~Sadowski.
\newblock {\em The Three-Dimensional Navier--Stokes Equations: Classical
  Theory}.
\newblock Cambridge Studies in Advanced Mathematics. Cambridge University
  Press, 2016.

\bibitem{Seidman}
T.~I. Seidman.
\newblock Time-invariance of the reachable set for linear control problems.
\newblock {\em J. Math. Anal. Appl.}, 72(1):17--20, 1979.

\bibitem{Strohmaier-Waters-2022}
A.~Strohmaier and A.~Waters.
\newblock Analytic properties of heat equation solutions and reachable sets.
\newblock {\em Math. Z.}, 302(1):259--274, 2022.

\bibitem{VanNeerven-2005}
J.~M. A.~M. van Neerven.
\newblock Null controllability and the algebraic {R}iccati equation in {B}anach
  spaces.
\newblock {\em SIAM J. Control Optim.}, 43(4):1313--1327, 2004/05.

\end{thebibliography}
\end{document}